\documentclass[12pt,a4wide]{article}
\usepackage{subfiles}

\usepackage[letterpaper, left=1truein, right=1truein, top = 1truein, bottom = 1truein]{geometry}
\usepackage{xr}
\usepackage{extarrows}

\usepackage{amsmath, amssymb, amsfonts, enumerate,amsbsy, amsthm, epsfig, graphicx,rotating}
\usepackage{mathtools,bm}
\usepackage[usenames]{color}
\usepackage{enumerate,enumitem}
\RequirePackage[colorlinks,linkcolor=blue,citecolor=blue,urlcolor=blue]{hyperref}
\usepackage{natbib}

\usepackage[normalem]{ulem}
\usepackage{multirow,booktabs}
 \usepackage{fancyhdr}
\usepackage{tabularx}
\usepackage{tabulary}

\usepackage{caption}
\usepackage{subcaption}

\usepackage{algcompatible}
\usepackage{algorithm}

\usepackage{etoolbox}
\makeatletter
\newcommand\fs@myruledtop{\def\@fs@cfont{\bfseries}\let\@fs@capt\floatc@ruled
  \def\@fs@pre{\hrule height.8pt depth0pt \kern2pt}%
  \def\@fs@post{\kern2pt}%
  \def\@fs@mid{\kern2pt\hrule\kern2pt}%
  \let\@fs@iftopcapt\iftrue}
\newcommand\fs@myruledbottom{\def\@fs@cfont{\bfseries}\let\@fs@capt\floatc@ruled
  \def\@fs@pre{\kern2pt}%
  \def\@fs@post{\kern2pt\hrule\relax}%
  \def\@fs@mid{}%
  \let\@fs@iftopcapt\iftrue}
\makeatother

\newif\ifalgobreak
\newif\ifalgobreaksecond

\AtBeginEnvironment{algorithm}{
\ifalgobreak
  \floatstyle{myruledtop}
  \restylefloat{algorithm}
\fi
\ifalgobreaksecond
  \floatstyle{myruledbottom}
  \restylefloat{algorithm}
\else
\fi
}

\AfterEndEnvironment{algorithm}{%
  \algobreakfalse\algobreaksecondfalse
  \floatstyle{ruled}
  \restylefloat{algorithm}%
}

\usepackage{url}

\label{equation notation}
\numberwithin{equation}{section}
\def\be{\begin{equation}}
\def\ee{\end{equation}}
\def\bea{\begin{eqnarray}}
\def\eea{\end{eqnarray}}
\def\bd{\begin{displaymath}}
\def\ed{\end{displaymath}}
\def\bda{\begin{eqnarray*}}
\def\eda{\end{eqnarray*}}

\newcommand{\beq}{\begin{eqnarray*}}
\newcommand{\eeq}{\end{eqnarray*}}

\newcommand{\beqn}{\begin{eqnarray}}
\newcommand{\eeqn}{\end{eqnarray}}

\def\bsm{\begin{small}}
\def\esm{\end{small}}

\usepackage{pgfplots}
\usepackage{tikz}
\usetikzlibrary{patterns,arrows, decorations.markings}
\usepgfplotslibrary{fillbetween}

\label{parenthesis notation}
 \def\lsk{\left(}
\def\rsk{\right)}
\def\lbk{\left \{ }
\def\rbk{\right \} }
\def\lmk{\left [ }	
\def\rmk{\right ] }
\def\labs{\left | }
\def\rabs{\right | }

\def\t0{\theta_0}

\def\ha1{\wh \beta_1}

\newcommand{\bone}{\mbox{\bf 1}}

\def\bnt{\begin{enumerate}}
\def\ent{\end{enumerate}}
\def\T{{ \mathrm{\scriptscriptstyle T} }}

\def\bsc{\begin{scriptsize}}
\def\esc{\end{scriptsize}}

\label{theorem notation}
\newtheorem{theorem}{Theorem}[section]
\newtheorem{lemma}{Lemma}[section]
\newtheorem{corollary}{Corollary}[section]

\newtheorem{assumption}{Assumption}[section]
\newtheorem{remark}{Remark}[section]

\newcommand{\wh}{\widehat}

\makeatletter
\newcommand{\figcaption}{\def\@captype{figure}\caption}
\newcommand{\tabcaption}{\def\@captype{table}\caption}
\makeatother

\newcommand{\diag}{{\rm diag}}

\label{Var notation}
\newcommand{\var}{\mbox{$\mathrm{Var}$}}

\newcommand{\cov}{{\rm \mathbb{C}ov}}

\renewcommand{\P}{\mathbb{P}}

\newcommand{\e}{\mathbb{E}}
\newcommand{\p}{\mathbb{P}}

\label{ABC notation} 
\newcommand{\bA}{{\mathbf A}}

\newcommand{\bI}{{\mathbf I}}

\newcommand{\bM}{{\mathbf M}}

\newcommand{\bP}{{\mathbf P}}
\newcommand{\bQ}{{\mathbf Q}}

\newcommand{\bS}{{\mathbf S}}

\newcommand{\bV}{{\mathbf V}}

\newcommand{\bX}{{\mathbf X}}
\newcommand{\bY}{{\mathbf Y}}
\newcommand{\bZ}{{\mathbf Z}}

\label{abc notation}
\newcommand{\ba}{{\mathbf a}}
\newcommand{\bb}{{\mathbf b}}

\newcommand{\bfe}{{\mathbf e}}

\newcommand{\bq}{{\mathbf q}}

\newcommand{\bt}{{\mathbf t}}
\newcommand{\bu}{{\mathbf u}}
\newcommand{\bv}{{\mathbf v}}
\newcommand{\bw}{{\mathbf w}}

\newcommand{\bz}{{\mathbf z}}

\label{Greek letter notation}

\newcommand{\bvphi}  {\boldsymbol{\varphi}}

\newcommand{\bepsilon}{\boldsymbol{\varepsilon}}

\newcommand{\bxi} {\boldsymbol{\xi}}
\newcommand{\bmu} {\boldsymbol{{\mathbf \mu}}}

\newcommand{\bSigma}{{\mathbf \Sigma}}

\newcommand{\bLambda} {\boldsymbol{\Lambda}}

\newcommand{\deq}{\overset{L}{=}}
\newcommand{\tr}{\mathrm{tr}}
\newcommand{\mt}{^{\T}}
\newcommand{\inverse}{^{-1}}
\newcommand{\pwtwo}{^2}

\newcommand{\bzero}{{\mathbf 0}}

\def\T{{ \mathrm{\scriptscriptstyle T} }}
\def\v{{\varepsilon}}

\newcommand{\Perp}{\perp \! \! \! \perp}

\usepackage{scalerel,stackengine}
\stackMath
\newcommand\reallywidecheck[1]{%
\savestack{\tmpbox}{\stretchto{%
  \scaleto{%
    \scalerel*[\widthof{\ensuremath{#1}}]{\kern-.6pt\bigwedge\kern-.6pt}%
    {\rule[-\textheight/2]{1ex}{\textheight}}
  }{\textheight}%
}{0.5ex}}%
\stackon[1pt]{#1}{\scalebox{-1}{\tmpbox}}%
}

\makeatletter
\def\@makefnmark{%
  \leavevmode
  \raise.9ex\hbox{\fontsize\sf@size\z@\normalfont\tiny\@thefnmark}}
\makeatother

\makeatletter
    \let\@fnsymbol\@arabic
\makeatother

\usepackage{setspace} 

\newcommand{\Viiup}{V_{\mathrm{up}}}
\newcommand{\Viilow}{V_{\mathrm{lo}}}

\makeatletter
\tikzset{
    nomorepostaction/.code=\makeatletter\let\tikz@postactions\pgfutil@empty, 
    my axis/.style={
        postaction={
            decoration={
                markings,
                mark=at position 1 with {
                    \arrow[ultra thick]{latex}
                }
            },
            decorate,
            nomorepostaction
        },
        thin,
        -, 
        every path/.append style=my axis 
    }
}
\makeatother

\begin{document}

\date{}

\title{With random regressors, least squares inference is robust to correlated errors with unknown correlation structure
}

\author{
 Zifeng Zhang$^{\sharp}$, ~Peng Ding$^{\S}$, Wen Zhou$^*$, and  Haonan Wang$^{\sharp}$ 
 \\[3ex]
\normalsize  $^{\sharp}$ Department of Statistics, Colorado State University, Fort Collins, CO\\[0.5ex] 
\normalsize  $^{\S}$ Department of Statistics, University of California, Berkeley, Berkeley, CA\\[0.5ex]
\normalsize  $^*$ Department of Biostatistics, School of Global Public Health, \\ \normalsize   New York University, New York, NY
}
\maketitle

\begin{abstract}
Linear regression is arguably the most widely used statistical method. 
With fixed regressors and correlated errors, the conventional wisdom is to modify the variance-covariance estimator to accommodate the known correlation structure of the errors.  We depart from the literature by showing that with random regressors, linear regression inference is robust to correlated errors with unknown correlation structure.  The existing theoretical analyses for linear regression are no longer valid because even the asymptotic normality of the least-squares coefficients breaks down in this regime.  We first prove the asymptotic normality of the $t$ statistics by establishing their Berry--Esseen bounds based on a novel probabilistic analysis of self-normalized statistics. We then study the local power of the corresponding $t$ tests and show that, perhaps surprisingly, error correlation can even enhance power in the regime of weak signals.  Overall, our results show that linear regression is applicable more broadly than the conventional theory suggests, and further demonstrate the value of randomization to ensure robustness of inference. 

\vspace{0.25cm}
\noindent {\bf Keywords:} Asymptotic normality; 
Linear regression; 
Random design; 
Randomization.
\end{abstract}

\section{Linear regression: fixed design, random design, and error distribution} 

\subsection{Literature review and our perspective}

Linear regression is widely used in many disciplines and has attracted continued interest in statistics research; see \citet[][Appendix A]{lei2021assumption} for a recent review.  The classic linear model $y=X \beta + \bepsilon$ assumes that the $n\times d$ covariate matrix $X$ is fixed and the  $n$-dimensional error vector $\bepsilon$ has independent and identically distributed normal components with mean $0$ and variance $\sigma^2$. 
Under this model, (a) the ordinary least squares (OLS) estimator $\widehat{\beta}=(X\mt X)^{-1}X\mt y$ is normal with mean $\beta$ and covariance $\mathrm{cov}(\widehat{\beta}) = \sigma^2 (X^{\T} X)^{-1}$, (b) $\widehat{\sigma}^2 = y^\T ( I- P_{X}) y /(n-d)$ is unbiased for $\sigma^2$,
with $\widehat{\sigma}^2 / \sigma^2\sim \chi^2_{n-d} / (n-d)$,
where $P_{X} = X(X\mt X)^{-1}X\mt $ is the projection matrix onto the column space of $X$, 
and (c) $\widehat{\beta}$ and $\widehat{\sigma}^2$ are independent. 
The results (a)--(c) justify the statistical inference based on the pivotal quantity $ T_j = L_j^{-1}(\widehat{\beta}_j-\beta_j) \sim t_{n-d}$, where $\beta_j$ and $\widehat{\beta}_j$ are the $j$th coordinate of $\beta$ and $\widehat{\beta}$, respectively, and $L_j =
\widehat{\sigma} \{ e_j^\T (X^\T X)^{-1} e_j \}^{1/2}$ is the standard error of $\widehat{\beta}_j$ with $e_j$ being the $j$th basis vector in the $d$-dimensional space.

There is a large literature on relaxing the assumption of independent and identically distributed normal errors. 
First, we can relax the normality assumption but can still show  $T_j \rightarrow \mathcal{N}(0,1) $ in distribution by the law of large numbers and central limit theorem. The change from the $t$ quantiles to normal quantiles is small when $n$ is large compared with $d$. 
Second, we can relax the homoskedasticity assumption on the errors. With heteroskedastic errors, \citet{eicker1967limit} and \citet{white1980heteroskedasticity} proposed a heteroskedasticity-robust covariance estimator. 
Third, we can further allow for dependence among the errors. With clustered errors, \citet{liang1986longitudinal} proposed to use the cluster-robust covariance estimator. With time series errors, \citet{newey1987simple} proposed to use the autocorrelation-robust covariance estimator. With spatial or network correlated errors, we can construct the corresponding robust covariance estimators.

As reviewed above, the literature focuses on modifying the standard error in constructing the $t$ statistic, under various known correlation structures of the errors. Departing from the literature, 
we study the robustness of the original OLS inference procedure with respect to correlated errors with unknown correlation structure. 
We do not modify the original definition of the $t$ statistic but show that $T_j  \rightarrow \mathcal{N}(0,1)$ in distribution still holds under the assumption of random regressors even if the errors have an unknown correlation structure. 
With correlated errors, the asymptotic normality of $\widehat{\beta}_j$ breaks down in general. However, the central limit theorem for the $t$-statistic $T_j$ can still hold in that regime. Intuitively, $T_j$ has a ratio form and the correlation effect of the error $\bepsilon$ cancels out because it appears in both the numerator and the denominator. To make this argument rigorous, we will first show that the key stochastic component in $T_j$ is approximately $X\mt  \bepsilon  /  \| \bepsilon \| $ and then show the self-normalized error $\bepsilon  /  \| \bepsilon \| $ is nearly uniformly distributed over the unit sphere as long as the correlation is not extremely strong. Due to these two facts, the randomness of $T_j$ is approximately driven by the sample mean of the rows of $X$, which follows the central limit theorem with random regressors. See Section \ref{all of coverage prob} for more details.

In short, our theory demonstrates that OLS inference is valid even with correlated errors, as long as the regressors are random. With fixed regressors, the correlated errors will invalidate OLS inference in general. Therefore, with fixed regressors or conditional on random regressors, the $p$-values from OLS can be non-uniform under the null hypotheses due to correlated errors. However, averaged over the randomness of the regressors, the $p$-values become uniform under the null hypotheses even if the errors are correlated in unknown ways. Overall, our theory shows that OLS inference is applicable more broadly than the classic theory suggests.

Importantly, the regime of random regressors arises naturally from randomized experiments, in which the experimenter has control over the distribution of the treatment. Therefore, our theory further demonstrates the value of randomization to ensure robustness of inference. 
Our setting with random regressors is reminiscent of the framework of randomization-based inference. 
In that literature, the focus was the robustness of inference with misspecified models \citep{lin2013agnostic}. 
In contrast, we focus on the robustness of  inference with correlated errors.

\subsection{A simulated example to motivate the theory}
\label{sec::motivating-example}

To motivate the development of the theory, we start with the following simple yet nontrivial example. 
We generate data from the linear model $y_i = x_i \beta_1 + \bepsilon_i$ for $i=1,\ldots,n$ with $n=100$, where the $x_i$'s are independent and identically distributed Rademacher random variables, each with a probability of $1/2$ being either $+1$ or $-1$, and the $\bepsilon_i$'s are multivariate normal with  $\mathrm{cov}(\bepsilon_i, \bepsilon_j) = V_{ij} = \rho^{|i-j|}$. We will vary $\rho$ from $-0.9$ to $0.9$ in the simulation to investigate the impact of the strength of correlation on inference. This simple model is not completely unrealistic. For instance, if $x_i$ is the unit-level randomized treatment status with $+1$ for the treatment and $-1$ for the control, then the average treatment effect equals $E ( y_i\mid  x_i=1 ) - E ( y_i \mid  x_i = -1 ) = 2 \beta_1$.

\begin{figure}[h]
\centering
               \vspace{-0.5\baselineskip}
        \begin{subfigure}[h]{0.5\textwidth}
        \caption{}
                \includegraphics[height=2.75cm, width=0.9\linewidth]{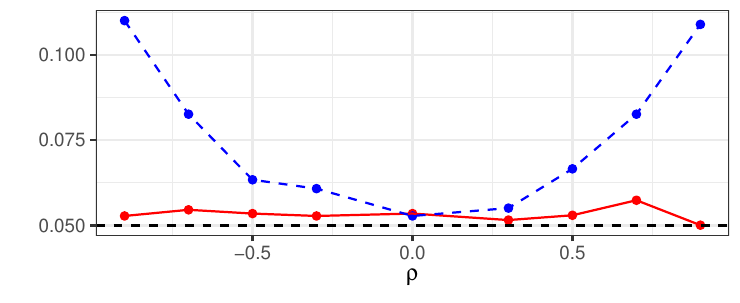}
          \label{fig: cp comparison fixed vs random X with AR1 error}  
        \end{subfigure}%
        \begin{subfigure}[h]{0.5\textwidth}
        \caption{}
                \includegraphics[height=2.75cm, width=0.9\linewidth]{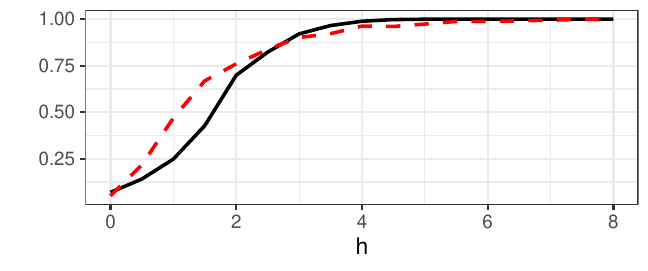}
                \label{fig: empirical power curves}
       \end{subfigure}%
               \vspace{-0.25\baselineskip}
        \caption{(a) Non-coverage probabilities of random (solid) and fixed (dashed) design. (b) Empirical power of independent and identically distributed (solid) and correlated (dashed) errors.}
         \label{fig: motivating example}
                        \vspace{-0.75\baselineskip}
\end{figure}

Based on $10,000$ replications, we calculate the non-coverage probabilities of the confidence interval $\widehat{\beta}_1 \pm 1.96 L_1$ of the true parameter $\beta_1=1$. 
While the $x_i$'s are regenerated for each replication under the random design, they are kept unchanged across replications under the fixed design. Figure \ref{fig: motivating example}(a) shows that under the random design, the confidence interval remains valid, but under the fixed design, it is not valid due to the correlated errors.

Moreover, we study the local power of the one-sided test $\mathcal{I} (\widehat{\beta}_1/L_1 > 1.64)$. 
We consider the regime of $\beta_1 \asymp (n-1)^{-1/2} h$ with $h$ varying from $0$ to $8$, and 
set $\mathrm{cov} (\bepsilon) = V_{1,\rho} = (1-\rho)I_n + \rho 1_n 1_n^{\T}  \in \mathbb{R}^{n \times n}$ 
with the diagonal entries all being 1 and  off-diagonal entries all being $\rho$. In Figure. \ref{fig: motivating example}(b), 
we set $\rho = 0.9$.   
Perhaps surprisingly,  Figure \ref{fig: motivating example}(b) shows that the $t$ test has larger empirical power with correlated errors compared with independent errors when the signal is small.

Figure \ref{fig: motivating example} reveals some new phenomena which only appear with random regressors.  
To demystify Figure \ref{fig: motivating example}(a), we will demonstrate the validity of  $T_j$ under classic OLS by establishing its Berry--Esseen bound in Section \ref{all of coverage prob}. 
To demystify Figure \ref{fig: motivating example}(b), we will study the local power function of the $t$ test in Section \ref{all of power}. We first introduce the regularity conditions for our theory below. 
For a random variable $A$, let $\| A \|_{\psi_2} = \inf [ t>0: E\{ \exp (A^2 / t^2) \} \leq 2 ]$.

\begin{assumption}
\label{assumption::regularity-conditions}
Define the average variance as $\sigma^2 = n^{-1} \sum_{i=1}^n \mathrm{var}(\bepsilon_i)$ with possibly nonconstant values of $\mathrm{var}(\bepsilon_i)$.  Define $V = \sigma^{-2} \mathrm{cov}(\bepsilon)$ such that $\mathrm{tr}(V)=n$, which equals the correlation matrix of the errors when $\mathrm{var}(\bepsilon_i) = \sigma^2$ for all $i=1, \ldots, n$. The $(n, d)$ satisfies $d / n^{1/2} \rightarrow 0$ as $n\rightarrow \infty$. 

(i) $\bepsilon=\sigma V^{1/2} w$, where $V \in \mathbb{R}^{n \times n}$ is positive definite, and $w = (w_1, \ldots,  w_n)^{\T} \in \mathbb{R}^n$ has independent sub-Gaussian entries $w_i$ with zero mean, unit variance, and $\max_{1 \leq i \leq n} \| w_i\|_{\psi_2} \leq K_w$ for some constant $K_w>0$.

(ii) $X = Z \Sigma^{1/2}$, where $\Sigma \in \mathbb{R}^{d \times d}$ is positive definite, and $Z \in \mathbb{R}^{n \times d}$ has independent sub-Gaussian entries $z_{ij}$ with zero mean, unit variance, and $\max_{1 \leq i \leq n, 1 \leq j \leq d} \|z_{ij}\|_{\psi_2} \leq K_z$ for some constant $K_z > 0$.

(iii) $Z$ and $w$ are independent, and  $\mathrm{pr} ( Z\mt Z\ \mathrm{is\ singular} ) = 0$. 

(iv) $V_{\mathrm{lo}} \leq \mathrm{var} ( \bepsilon_i ) \leq V_{\mathrm{up}}$ for some constants $V_{\mathrm{lo}}, V_{\mathrm{up}}>0$, for all $i=1, \ldots, n$.

(v) $\mathrm{pr} ( \widehat{\sigma}^2 = 0 ) = 0$. 

\end{assumption}

Assumption \ref{assumption::regularity-conditions}(i) excludes heavy-tailed errors. 
Assumption \ref{assumption::regularity-conditions}(ii) emphasizes the condition on random regressors, and specifies the rows of $X$, denoted by $x_i$ for $i=1,\ldots,n$, as independent with zero mean and covariance $\Sigma$.
Assumption \ref{assumption::regularity-conditions}(iii) imposes the standard assumption of independence between the regressors and errors, and rules out degeneracy in the regressors. 
Assumption \ref{assumption::regularity-conditions}(iv) allows for heteroskedasticity but bounds the relative heteroskedasticity across units. 
Assumption \ref{assumption::regularity-conditions}(v) rules out the possibility of degenerate residuals, which is useful for simplifying the proofs. 
 
We use the following notation throughout the paper. 
For sequences $\{a_n\}$ and $\{b_n\}$, we write $a_n \lesssim b_n$ and $a_n \gtrsim b_n$ if there exists a positive integer $N$ such that for all $n > N$, we have $a_n\leq C_1 b_n$ and $a_n\geq C_2 b_n$ for some absolute constants $C_1$ and $C_2$, respectively.
Let $\Phi(\cdot)$ denote the cumulative distribution function of $\mathcal{N}(0,1)$, and let $z_{\alpha}$ denote the $\alpha$ upper quantile of $\mathcal{N}(0,1)$.
Let $\lambda_{\min}(V)$, $\lambda_{\max}(V)$ and $\lambda_i(V)$ denote the smallest, the largest, and the $i$th largest eigenvalues of the matrix $V$, respectively.

\section{Validity of OLS inference with correlated errors}
\label{all of coverage prob}

The key theoretical result to ensure the robustness of the classic OLS inference is the asymptotic normality of the $t$ statistic. 
Theorem \ref{eqn: subG coverage prob} below gives the Berry--Esseen bound on $T_j$.

\begin{theorem}
[Berry--Esseen bound on $T_j$]\label{Thm: subG coverage prob}
Under Assumption \ref{assumption::regularity-conditions}, 
we have
\begin{equation}
\sup_{t \in \mathbb{R}} \big| \mathrm{pr} ( T_j \leq t )  - \Phi(t) \big| \lesssim  \lambda_{\min}^{-3/2} ( V ) \cdot  \max ( d, \log n ) \cdot n^{-1/2} .
\label{eqn: subG coverage prob}
\end{equation}
\end{theorem}

If  $\lambda_{\min}(V) \geq c_{\mathrm{min}}>0$ for an absolute constant $c_{\mathrm{min}}$, the bound in \eqref{eqn: subG coverage prob} converges to 0 as long as $d / n^{1/2} \rightarrow 0$, which is required by Assumption \ref{assumption::regularity-conditions} and matches the condition invoked by \citet{bickel1983bootstrapping} to prove the asymptotic normality of the least squares coefficient with fixed regressors.
Even if the errors are strongly correlated with $\lambda_{\min}(V) \to 0$, the bound in \eqref{eqn: subG coverage prob} is still useful for establishing the central limit theorem of $T_j$ as long as the bound converges to 0.

Although Theorem \ref{Thm: subG coverage prob} looks similar to the classic Berry--Esseen bound with fixed regressors and independent errors, the mathematical details differ fundamentally. 
In particular, if we standardize the OLS coefficient by its true standard error, $T_j' = \{\sigma^2 e_j^{\T} (X\mt X)\inverse e_j \}^{-1/2} (\widehat{\beta}_j - \beta_j)$ does not satisfy the central limit theorem if the errors have a general correlation structure. 
Only when we standardize the OLS coefficient by its estimated standard error, $T_j = \{\widehat{\sigma}^2 e_j^\T (X\mt X)\\ \inverse e_j \}^{-1/2} (\widehat{\beta}_j - \beta_j)$ satisfies the central limit theorem. 
We revisit the simulation example in Section \ref{sec::motivating-example} to illustrate this phenomenon. 
Panels A and B of Figure \ref{fig: density of tstat with sigma and S} show the empirical densities of $T_1'$ and $T_1$, respectively. In the simulation, we set $\beta_1=1$ and 
$\mathrm{cov}(\bepsilon) = V = V_{1, \rho}$, 
with $\rho=0$ for independent errors and $\rho = 0.9$ for equally correlated errors. In Panel A, the empirical density of $T_1'$ does not match that of $\mathcal{N}(0,1)$ when the errors are correlated, whereas in Panel B, the empirical density of $T_1$ matches that of $\mathcal{N}(0,1)$ regardless of the correlation of the errors.

\begin{figure}[h]
\centering
        \includegraphics[height=3.5cm, width=0.9\linewidth]{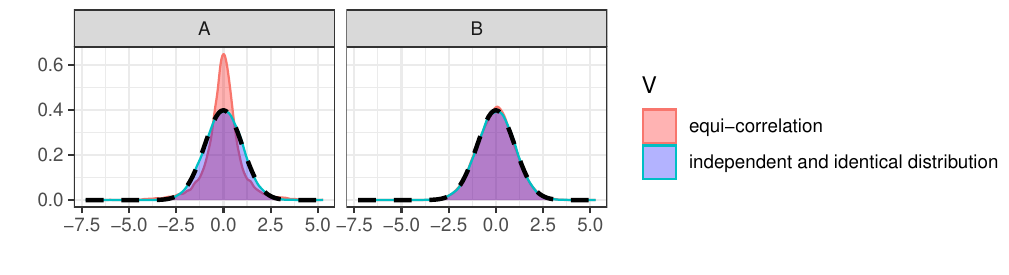}   
       \caption{Empirical densities of $T_1' = \{ \sigma^2 e_1^\T (X\mt X)\inverse e_1 \}^{-1/2} (\widehat{\beta}_1 - \beta_1)$ in Panel A and empirical densities of $T_1 = \{ \widehat{\sigma}^2 e_1^\T (X\mt X)\inverse e_1 \}^{-1/2} (\widehat{\beta}_1 - \beta_1)$ in Panel B, under different correlation structures of the errors. The black dashed curves are the $\mathcal{N}(0,1)$ density.}
       \label{fig: density of tstat with sigma and S}
\end{figure}

To understand the central limit theorem ensured by Theorem \ref{Thm: subG coverage prob}, we provide some heuristics below. 
Let $\widehat{\Sigma} = n^{-1} X \mt X = n^{-1}\sum_{i=1}^n x_i x_i^{\T} $ denote the empirical second moment of covariates. 
Then we can rewrite $T_j$ as
$$
T_j = 
\{\widehat{\sigma}^2 e_j^\T (X\mt X)\inverse e_j \}^{-1/2} (\widehat{\beta}_j - \beta_j)
=
  ( e_j^{\T} \widehat{\Sigma}^{-1} e_j )^{-1/2} e_j^{\T}   \widehat{\Sigma}^{-1}    \cdot \widehat{\sigma}^{-1} \cdot ( n^{-1/2} X\mt \bepsilon ),
$$
where the first term $  ( e_j^{\T} \widehat{\Sigma}^{-1} e_j )^{-1/2} e_j^{\T}   \widehat{\Sigma}^{-1} $ is unrelated to $V$, while $\widehat{\sigma}^{-1}$ and $n^{-1/2} X\mt \bepsilon$ are related to $V$. The classic theory of OLS proves (a) the consistency of $\widehat{\sigma}$ and (b) the asymptotic normality of $ n^{-1/2} X\mt \bepsilon$. 
Then Slutsky's theorem ensures the validity of the inference based on the asymptotic normality of $T_j$. However, both (a) and (b) break down if the errors have 
an unknown correlation structure $V$. 
Nevertheless, the asymptotic normality of $T_j$ still holds even though (a) and (b) do not hold. 
The theoretical justification of the asymptotic normality of $T_j$ is completely different from the classic theory. 
We will provide the heuristics for the asymptotic normality of $\widehat{\sigma}^{-1} \cdot ( n^{-1/2} X\mt \bepsilon )$. 
Assume $d$ is small compared with $n$. 
Approximately, we have 
$$
\widehat{\sigma}^{-1} \cdot ( n^{-1/2} X\mt \bepsilon )
\approx \{ \bepsilon^{\T} (I - P_{X}) \bepsilon  \}^{-1/2} \cdot  X\mt \bepsilon 
\approx ( \bepsilon^{\T} \bepsilon  )^{-1/2} \cdot  X\mt \bepsilon 
=  X\mt \frac{ \bepsilon  }{ \| \bepsilon \| }.
$$
Lemma \ref{lemma: self-normlz epsilon sub-G property} in the Appendix ensures that the entries of the self-normalized vector $ \bepsilon  /  \| \bepsilon \| $ are around $ n^{-1/2}$. This key probabilistic result then ensures 
\begin{equation}\label{eq::concentration-CLT}
\widehat{\sigma}^{-1} \cdot ( n^{-1/2} X\mt \bepsilon )
\approx n^{-1/2} \sum_{i=1}^n x_i,
\end{equation}
which is asymptotically normal by the standard central limit theorem with random $x_i$'s. 

\begin{remark}
\label{remark::single-cluster}
Consider the extreme case with a single cluster so that all the $\v_i$'s are correlated. The central limit theorem for $\widehat{\beta}$ breaks down and the cluster-robust covariance estimator degenerates to $0$ \citep{liang1986longitudinal}, making the corresponding inference useless. By contrast, the standard OLS inference based on $T_j$ can still be valid as long as Assumption \ref{assumption::regularity-conditions} holds. 
\end{remark}

 \begin{remark}\label{remark::2randomization}
The Berry--Esseen bound in Theorem \ref{Thm: subG coverage prob} relies crucially on the assumption of random regressors. 
 \cite{chetverikov2023standard} reported a  similar robustness property of the classic OLS inference. 
Our result is also related to the randomization-based inference with randomized treatment \citep{barrios2012clustering, lin2013agnostic,
abadie2023should}.
However, our theory is fundamentally different. 
Their theories deal with the regime of asymptotically normal estimators and consistent variance estimators, whereas our theory can deal with the regime in which the asymptotic normality of the OLS coefficient breaks down and our proof relies on the concentration properties of the self-normalized vector $ \bepsilon  /  \| \bepsilon \| $ as shown in Lemma  \ref{lemma: self-normlz epsilon sub-G property} in the Appendix. 
 \end{remark}

\section{Power analysis under a local alternative hypothesis}
\label{all of power}

Based on the asymptotic normality in Theorem \ref{Thm: subG coverage prob}, the one-sided test for the null hypothesis of $H_0: \beta_j = 0$ is $\mathcal{I}( L_j^{-1} \widehat{\beta}_j > z_{\alpha})$. We will further study the local power of this test under the alternative hypothesis of 
\begin{equation}
\label{eq::alternative}
H_1: \beta_j = h \left(  \frac{ \sigma^2 e_j^\T \Sigma^{-1} e_j  }{ n - d}  \right)^{1/2}  \text{ with } h>0.
\end{equation}
The choice of the alternative hypothesis $H_1$ in \eqref{eq::alternative} is motivated by the form of $L_j$ to simplify the form of the asymptotic power function, which will be clear in \eqref{eq::shift-term} below. 

\begin{theorem}
[power]
\label{theorem: power}
Under Assumption \ref{assumption::regularity-conditions} and $H_1$ in \eqref{eq::alternative}, 
we have
\begin{equation}
\big| \mathrm{pr} ( L_j^{-1}\widehat{\beta}_j > z_{\alpha} ) - \pi(h, V)\big| 
\lesssim  \lambda_{\min}^{-3/2}(V) \cdot   \max ( d,\log n ) \cdot   (\log n)^{1/2} \cdot   n^{-1/2} , 
\label{eqn: power difference general} 
\end{equation}
where the asymptotic power function equals
$
\pi ( h, V ) = E\{  \Phi (  h \delta^{-1/2} 
 - z_{\alpha} ) \} ,
$
with the expectation taken over $\delta =  \v^\T \v / (n \sigma^2) =  w^{\T} V w / n$. 
\end{theorem}

In the classic regime with independent errors, $\delta \rightarrow 1$ in probability and the asymptotic power function reduces to $\Phi(  h - z_{\alpha} )$. 
In the regime with strongly correlated errors, $\delta$ converges to a random variable as shown in Lemma \ref{lemma: non stable for non standardized epsilon} in the Appendix.
Therefore, the asymptotic power function has the form of $\pi ( h, V )$ in Theorem \ref{theorem: power}. 

We provide some heuristics for the asymptotic power function. The statistic $L_j^{-1} \widehat{\beta}_j$ decomposes as $L_j^{-1} \widehat{\beta}_j = L_j^{-1} (\widehat{\beta}_j - \beta_j)  +  L_j^{-1} \beta_j$, where (a) the first term, $T_j = L_j^{-1} (\widehat{\beta}_j - \beta_j)$, is approximately $\mathcal{N}(0,1)$ by Theorem \ref{Thm: subG coverage prob}; (b) the second term is approximately \begin{eqnarray*}\label{eq::shift-term} L_j^{-1} \beta_j=h \left(  \frac{ \sigma^2 e_j^\T \Sigma^{-1} e_j  }{ n - d}  \right)^{1/2}  \Big / \left\{ \frac{\widehat{\sigma}^2 e_j^\T (n^{-1}X^\T X)^{-1} e_j}{n} \right\}^{1/2} \approx h \delta^{- 1/2},\end{eqnarray*} because $n^{-1}X^\T X \approx \Sigma$ and $\widehat{\sigma}^2 = \v^\T (I-P_X) \v/(n-d) \approx  \v^\T \v/(n-d)$; and (c) the first and second terms are asymptotically independent. We can use (a)--(c) to derive the asymptotic power function
\begin{equation}
\mathrm{pr} ( {L_j}^{-1} \widehat{\beta}_j > z_\alpha ) = \mathrm{pr} ( T_j > z_\alpha - {L_j}^{-1} \beta_j )  \approx  \Phi( h\delta^{-1/2} -  z_\alpha),
\label{eqn: power phase intuition}
\end{equation}
with random $\delta$. 
Therefore, the final asymptotic power function needs to take expectation over $\delta$, as stated in Theorem \ref{theorem: power}.

Although $\delta$ has mean 1, it can have large variability around 1 with strongly correlated errors as shown in Lemma \ref{lemma: non stable for non standardized epsilon} in the Appendix.
Compared with the power function for independent errors, $\Phi( h -  z_\alpha)$, the integrand for correlated errors, $\Phi( h\delta^{-1/2} -  z_\alpha)$, has a larger value if $\delta < 1$ and a smaller value if $\delta > 1$. 
Averaged over $\delta$, whether correlated errors benefit or harm power depends on the variability of $\delta$ relative to $h$. 
Overall, with small $h$, correlated errors benefit power, whereas with large $h$, they harm power. 
To gain insights into this phenomenon, we simplify the power function under normal errors with the exchangeable correlation structure $V_{1,\rho}$ below.

{}

\begin{corollary}[power function under exchangeable correlation structure]
Under Assumption \ref{assumption::regularity-conditions} and $H_1$ in \eqref{eq::alternative}, if $w \sim \mathcal{N}(0, I_n)$ and $V = V_{1,\rho}$ with $\rho \in [0,1-c_{\mathrm{min}}]$ for an absolute constant $c_{\mathrm{min}} \in (0,1)$,
then
$$
\big| \mathrm{pr}  (L_j^{-1}{\widehat{\beta}_j} > z_\alpha ) - \pi( h, \rho ) \big|  
\lesssim  ( \log n )^{1/2} \cdot  \max\left( d,\log n \right) \cdot  n^{-1/2} , 
$$
where
$
\pi( h, \rho) =
E\{ \Phi(h ( \rho \chi_1^2 + 1 - \rho)^{-1/2} - z_\alpha)   \},
$
with expectation taken over $\chi_1^2$.
%
\label{corollary: power pihat_G(h,rho)}
\end{corollary}

The asymptotic power function $\pi(h, \rho)$ in Corollary \ref{corollary: power pihat_G(h,rho)} is a special case of the  general $\pi(h, V)$ in Theorem \ref{theorem: power}, with $\delta$ replaced by its asymptotic distribution $\rho\chi^2_{1} + 1 - \rho $. Corollary \ref{corollary: power pihat_G(h,rho)} offers insights into the dependence of power on the correlation structure. 
Figure \ref{fig: power theoretical} highlights the region of $(h, \rho)$ with $\pi(h, \rho) - \pi(h, 0) > 0$ such that the $t$-test based on OLS is more powerful with correlated errors than with independent errors. It shows that with small $h$, correlated errors improve the power, whereas with large $h$, correlated errors harm the power.

\begin{figure}[h]
\centering
               \vspace{-0.5\baselineskip}
        \begin{subfigure}[h]{0.5\textwidth}
        \caption{}
                \includegraphics[width=0.85\linewidth]{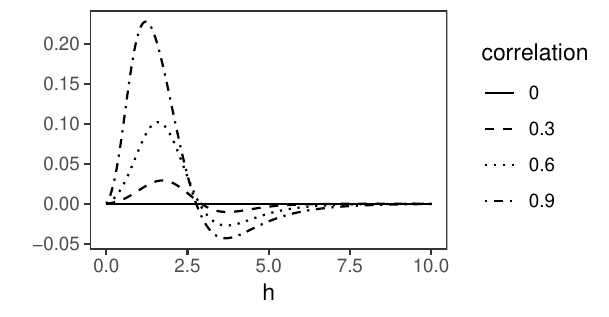}
          \label{fig: theo power 6 curves}  
        \end{subfigure}%
        \begin{subfigure}[h]{0.5\textwidth}
        \caption{}
                \includegraphics[width=0.9\linewidth]{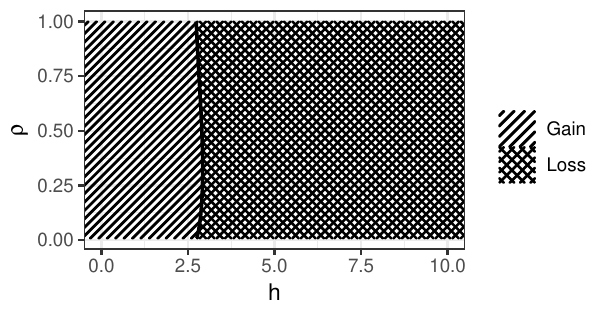}
                \label{fig: theo power phase transition}
       \end{subfigure}%
               \vspace{-0.25\baselineskip}
        \caption{(a) $\pi(h, \rho) - \pi(h,0)$ as a function of $h$, given different values of $\rho$. (b) Region with power gain such that $\pi(h, \rho) - \pi(h,0)>0$ and region with power loss such that $\pi(h, \rho) - \pi(h,0)<0$. }
         \label{fig: power theoretical}
                        \vspace{-0.75\baselineskip}
\end{figure}

\section{Fixed, random, or mixed regressors}
\label{discussion}

With random regressors, we have demonstrated the robustness of OLS inference with respect to correlated errors. With fixed regressors, the theory breaks down. We can construct a counterexample. For instance, if $y=\beta 1_n+\v$ where $\v \sim \mathcal{N}(0,  V_{1,\rho})$, then $\widehat{\beta} - \beta\sim \mathcal{N}(0, \rho + n\inverse(1-\rho))$ is bounded in probability, and $L_1 \approx \{ n^{-2} \v\mt \v \}^{1/2}$ converges to $0$ in probability. Therefore, wrongly assuming asymptotic normality of $L_1^{-1} (\widehat{\beta} - \beta)$ does not give valid inference.

When the regressors contain both fixed and random components, the OLS inference for the coefficients of the fixed components is not valid whereas that of the random components is still valid asymptotically. 
Consider model $y = X \beta + \bepsilon$ with $n=100$, $\beta = [1,1,1]^\T$, $X = [X_1, X_2, X_3] \in \mathbb{R}^{100 \times 3}$ where $X_1 = 1_n$ is the intercept. 
Both $X_2$ and $X_3$ have independent and identically distributed Rademacher entries.
Yet $X_2$ is fixed across replications and $X_3$ is regenerated for each replication.
The errors satisfy $\bepsilon \sim \mathcal{N}(0, V)$,  where $V_{ij} =\rho^{|i-j|}$ with $\rho$ varying from $-0.9$ to $0.9$. 
Figure \ref{fig: mixed regressor d=3} demonstrates the non-coverage probabilities of the confidence interval $\widehat{\beta}_j \pm 1.96 L_j$ and the densities of $T_j$. The OLS inference for the coefficient of the random regressor $X_3$ is valid because $T_3$ is close to $\mathcal{N}(0,1)$. By contrast, the OLS inference for other coefficients is not valid.

\begin{figure}[h]
\centering
               \vspace{-0.5\baselineskip}
        \begin{subfigure}[h]{0.4\textwidth}
        \caption{}
                \includegraphics[scale=0.5]{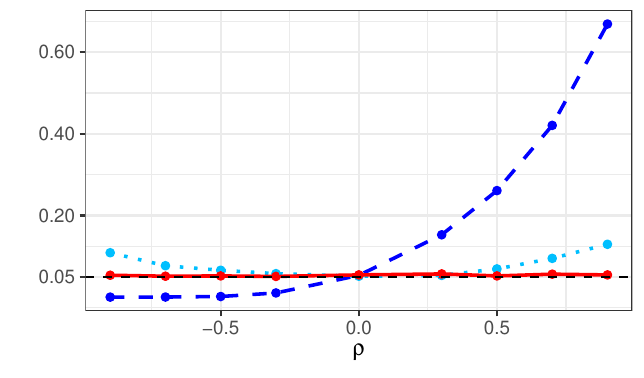} 
        \end{subfigure}%
        \begin{subfigure}[h]{0.6\textwidth}
        \caption{}
                \includegraphics[scale=0.5]{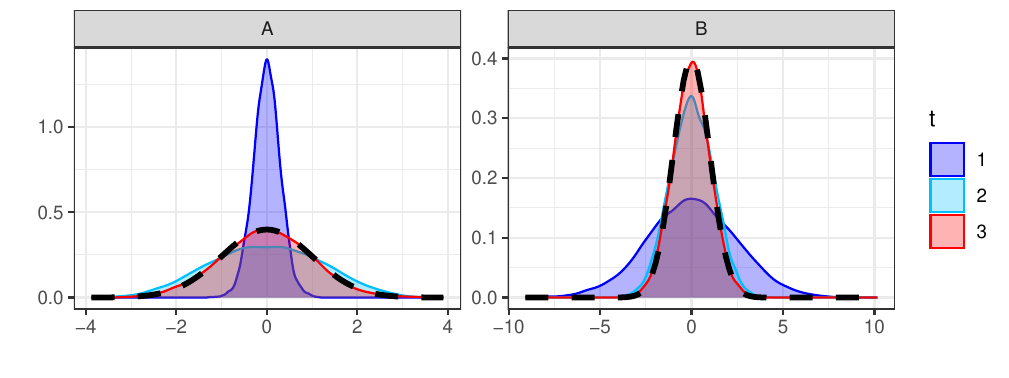}
       \end{subfigure}%
               \vspace{-0.25\baselineskip}
        \caption{(a) Non-coverage probabilities of $\beta_1$ (dashed), $\beta_2$ (dotted), and $\beta_3$ (solid). (b) Empirical densities of  $T_j$ for $j=1,2,3$ when $\rho = -0.9$ (Panel A) and 0.7 (Panel B). The black dashed curves are the $\mathcal{N}(0, 1)$ density.}
         \label{fig: mixed regressor d=3}
                        \vspace{-0.75\baselineskip}
\end{figure}


\medskip

\appendix

 \medskip

\section*{Appendix}

\renewcommand\thesection{\Alph{section}}
\renewcommand\thesubsection{\thesection.\arabic{subsection}}
 
\setcounter{lemma}{0}
 \makeatletter
\renewcommand{\thelemma}{A.\@arabic\c@lemma}
\makeatother

Lemmas \ref{lemma: self-normlz epsilon sub-G property} and \ref{lemma: non stable for non standardized epsilon} below characterize  $\bepsilon / \| \bepsilon \|_2$ and $\delta =  \bepsilon\mt \bepsilon / (n\sigma^2 ) = w\mt V w / n$, respectively.

\begin{lemma}
Assume Assumption \ref{assumption::regularity-conditions} (i), (iv), and  $\mathrm{pr}(\bepsilon = 0) = 0$. 
Define $v = \bepsilon / \| \bepsilon \|_2$ with entries $v_i$, $i=1,\ldots,n$. 
Then for any $t > 0$,  we have
$$ \mathrm{pr}( n^{1/2} |v_i| \geq t) \leq 4 \exp \{ -c \lambda_{\min}(V) t^2 \}, $$
where $c$ is an absolute constant as a function of $K_w$, $V_{\mathrm{lo}}$, and  $V_{\mathrm{up}}$.
\label{lemma: self-normlz epsilon sub-G property}
\end{lemma}

Lemma \ref{lemma: self-normlz epsilon sub-G property} extends \citet[][Theorem 3.4.6]{vershynin2018high} for the uniform distribution over the sphere with radius $n^{1/2}$. 
We adopt a similar proving technique but establish a stronger result for a general vector $\v$ with correlation structure $V$.
If we further 
assume  $\lambda_{\min}(V) \geq c_{\mathrm{min}}$ for some absolute constant $c_{\mathrm{min}} > 0$, 
then  
Lemma \ref{lemma: self-normlz epsilon sub-G property} ensures that the ratio between $|v_i|$ and $n^{-1/2}$ has a sub-Gaussian tail, with $\|v_i\|_{\psi_2} \lesssim  ( n c_{\min} )^{-1/2}$, which further ensures that the $v_i$'s behave like $n^{-1/2}$. 
This is a key property to prove Theorems \ref{Thm: subG coverage prob} and \ref{theorem: power}.

\begin{lemma}
Under Assumption \ref{assumption::regularity-conditions} (i), suppose that $\lambda_{\max}(V) \asymp n^\iota$ where $\iota \in [0,1]$.
(a) If $\iota \in [0,1)$, then $\delta \rightarrow 1$ in probability. 
(b) If $\iota=1$, suppose that $\lambda_i(V)/n \to \alpha_i \in (0,1]$, $i=1,\ldots,K$, $\sum_{i=1}^K \alpha_i  \in (0,1]$, and $\lambda_i(V)/n \to 0$ for $i=K+1,\ldots,n$ with fixed $K$.
Under Assumption \ref{assumption::regularity-conditions} (iv), we have 
$
\delta  \rightarrow \sum_{i=1}^K \alpha_i Z_i^2 + (1-\sum_{i=1}^K  \alpha_i)
$
in distribution, 
where $Z_1,\ldots,Z_K$ are independent and identically distributed $\mathcal{N}(0,1)$.
\label{lemma: non stable for non standardized epsilon}
\end{lemma}

\bibliographystyle{chicago}

\clearpage

\section*{\Large Supplementary Files for ``With random regressors, least squares inference is robust to correlated errors with unknown correlation structure"
}

\setcounter{page}{1}
\renewcommand\thepage{S\arabic{page}}

\vspace{1cm}
Section \ref{sec: lemmas} proves the two lemmas in the Appendix of the main paper and documents technical lemmas used throughout this supplementary file.
Section \ref{sec: pf of all subG coverage prob} presents the proof of the Berry--Esseen bound in Theorem \ref{Thm: subG coverage prob}.
Section \ref{pf of all of power} provides the proofs for all the results about power, with Section \ref{subsec: pf of power approximation} for Theorem \ref{theorem: power}, Section \ref{subsec: pf of power pihat_G(h, rho)} for Corollary \ref{corollary: power pihat_G(h,rho)}, and Section \ref{subsec: pf of Dhat(h, rho) property} for some properties of $\pi (h, \rho) - \pi (h, 0)$ summarized in Lemma \ref{lemma: properties of D pi}, respectively.

\renewcommand\thesection{\Alph{section}}
\renewcommand\thesubsection{\thesection.\arabic{subsection}}
\renewcommand\thefigure{S.\arabic{figure}}
\renewcommand\thetable{S.\arabic{table}}

\setcounter{equation}{0}
 \makeatletter
 \renewcommand{\theequation}{S.\@arabic\c@equation}
\makeatother

\setcounter{lemma}{0}
 \makeatletter
\renewcommand{\thelemma}{S.\@arabic\c@lemma}
\makeatother

{\bf Notation.}   
For sequences $\{a_n\}$ and $\{b_n\}$, we write $a_n \asymp b_n$ if there are constants $m$, $M$ and $N$ such that $0<m<|a_n/b_n|<M<\infty$ for all $n>N$. 
We write $a_n \lesssim b_n$ and $a_n \gtrsim b_n$ if there exists a positive integer $N$ such that for all $n > N$, we have $a_n\leq C_1 b_n$ and $a_n\geq C_2 b_n$ for some absolute constants $C_1$ and $C_2$, respectively. 
We write $a_n = o(b_n)$ if $\lim_{n \to \infty} a_n / b_n = 0$.
We write $a_n = o(1)$ if $\lim_{n \to \infty} a_n = 0$.
For a vector $\ba$, let $\Vert\ba\Vert_1$ and $\Vert\ba\Vert_2$ denote the $\mathcal{L}_1$ and $\mathcal{L}_2$ norms, respectively. 
Let $\bone_n \in \mathbb{R}^n$ denote the vector whose entries are all $1$.
Let $\diag\{a_1,\dots,a_n\}$ denote the diagonal matrix in $\mathbb{R}^{n \times n}$ with diagonal entries $a_1,\dots, a_n$.
Given an $n\times d$ matrix $\bM$, let $s_{\min}(\bM)$ and $s_{\max}(\bM)$ denote the minimum and maximum nonzero singular values of $\bM$, respectively. 
Let $\|\bM\|$ and $\|\bM\|_{\textup{F}}$ denote the operator norm and the Frobenius norm of $\bM$, respectively. 
Let $\bP_{\bM}$ denote the projection matrix of the column space of $\bM$, that is, $\bP_{\bM} = \bM (\bM\mt \bM) \inverse \bM\mt$ if $\bM\mt\bM$ is nonsingular. 
For a square matrix $\bM$, let $\mathrm{det}(\bM)$ denote the determinant of $\bM$ 
and let $\diag(\bM) \in \mathbb{R}^{d \times d}$ denote a diagonal matrix with diagonal entries from $\bM$.
Let $\mathrm{diag}\{ \bM_1,\dots,\bM_K \}$ denote the block diagonal matrix with diagonal blocks $\bM_1,\dots, \bM_K$.
For any symmetric matrix $\bM \in \mathbb{R}^{d \times d}$, 
let $\lambda_{\min}(\bM)$ and $\lambda_{\max}(\bM)$ denote the minimum and maximum eigenvalues of $\bM$, respectively;
let $\lambda_i(\bM)$ denote the $i$th largest eigenvalue of $\bM$ for $i=1,\dots,d$.
Let $\bI_d$ denote the $d\times d$ identity matrix. 
Let $\tr(\bM)=\sum_{i=1}^d M_{ii}$ denote the trace of $\bM$. 
Let $\bfe_i \in \mathbb{R}^{d}$ denote the vector with a $1$ in the $i$th position and  zeros elsewhere.
Let $\deq$ denote equality in distribution. 
For random vectors $\{\bX_n\}_{n=1}^{\infty}$, $\bX$, and $\bY$, let $\bX_n \overset{L}{\longrightarrow} \bX$ and $\bX_n \overset{P}{\longrightarrow} \bX$ denote the convergence in law and convergence in probability, respectively.
Let $\bX \Perp \bY$ denote that $\bX$ is independent of $\bY$. 
Let $\mathcal{N}(\bmu, \bSigma)$ denote the multivariate Gaussian distribution with mean $\bmu$ and covariance matrix $\bSigma$.
Let $\sim$ denote following certain distribution, for example, $\bX \sim \mathcal{N}(\bmu, \bSigma)$.
For a random variable $X$, define $\| X \|_{\psi_p} = \inf\{t > 0:\ \e [\exp(|X|^p / t^p)] \leq 2\}$, where $p=1$ represents the sub-Exponential norm and $p=2$ represents the sub-Gaussian norm.
We use $c, C$, and $\widetilde{c}$ to denote positive and generic absolute constants. 
With slight abuse of notation, $\mathcal{E}_1$, $\mathcal{E}_2$, and $\mathcal{E}_{\eta}$, referring to different events, have different meanings across different sections.

\setcounter{equation}{0}
 \makeatletter
 \renewcommand{\theequation}{A.\@arabic\c@equation}
\makeatother

\setcounter{lemma}{0}
 \makeatletter
\renewcommand{\thelemma}{A.\@arabic\c@lemma}
\makeatother

\section{Proofs of Lemmas \ref{lemma: self-normlz epsilon sub-G property} and \ref{lemma: non stable for non standardized epsilon} }
\label{sec: lemmas}

\subsection{Proof of Lemma \ref{lemma: self-normlz epsilon sub-G property} }

The proof of Lemma \ref{lemma: self-normlz epsilon sub-G property} uses the idea similar to that of Theorem 3.4.6 in \cite{vershynin2018high}.
Let $v_i = \bfe_i \mt \bv$, where $\bfe_i\in \mathbb{R}^n$ is the $i$th canonical basis of $\mathbb{R}^n$. Hence 
\beq
\P \lbk \sqrt{n}\labs \bfe_i\mt\bv \rabs \geq t \rbk 
&=& \P \lbk \frac{\labs \bfe_i\mt \bV^{1/2} \bw \rabs}{\sqrt{\bw\mt \bV \bw}} \geq \frac{t}{\sqrt{n}} \rbk \\ & \leq & \P \lbk \frac{\labs \bfe_i\mt \bV^{1/2} \bw \rabs}{\sqrt{\lambda_{\min}(\bV) \bw\mt \bw}} \geq \frac{t}{\sqrt{n}} \rbk\\
&\leq & \P \lbk \sqrt{\bw\mt\bw} < \frac{\sqrt{n}}{2} \rbk + \P \lbk \sqrt{\bw\mt\bw} \geq \frac{\sqrt{n}}{2}, \frac{\labs \bfe_i\mt \bV^{1/2} \bw \rabs}{\sqrt{\bw\mt\bw}} \geq \frac{t \sqrt{\lambda_{\min}(\bV)}}{\sqrt{n}} \rbk\\
&\leq & \P \lbk \bw\mt \bw < \frac{n}{4} \rbk + \P \lbk \labs \bfe_i\mt \bV^{1/2} \bw \rabs \geq \frac{\sqrt{\lambda_{\min}(\bV)}t}{2} \rbk.
\eeq 
By the Bernstein's inequality in Lemma \ref{lemma: Bernstein's inequality} below, we have
\begin{equation}
\P \lbk \bw\mt \bw < \frac{n}{4} \rbk = \P \lbk {\bw\mt \bw} - n < -\frac{3}{4} n \rbk \leq 2 \exp \lsk -c n \rsk \leq \exp[-c \lambda_{\min}(\bV) n],
\label{eqn: lemma of self-nmlz epsilon wtw tail bd}    
\end{equation}
where the last step in \eqref{eqn: lemma of self-nmlz epsilon wtw tail bd} follows from $\lambda_{\min}(\bV) \leq 1$.
By Assumption \ref{assumption::regularity-conditions} (iv), we have $\Viilow \leq \sigma^2 \leq \Viiup$ so that $\bfe_i\mt \bV \bfe_i = \var (\epsilon_i) / \sigma^2 \leq \Viiup / \Viilow$.
Then by Lemma \ref{lemma: sum of indep sub-Gaussians} and Assumption \ref{assumption::regularity-conditions} (i), we have $\| \bfe_i\mt \bV^{1/2} \bw \|_{\psi_2}^2 \leq C K_w^2 \bfe_i\mt \bV \bfe_i \leq C K_w^2 \Viiup / \Viilow$ for an absolute constant $C$.
Thus we have 
\begin{equation}
\P \lbk \labs \bfe_i\mt \bV^{1/2} \bw \rabs \geq \frac{\sqrt{\lambda_{\min}(\bV)}t}{2} \rbk \leq 2 \exp \lmk -c \lambda_{\min}(\bV) t^2 \rmk,
\label{eqn: lemma of self-nmlz epsilon eit V^1/2 w bd}    
\end{equation}
where $c$ in \eqref{eqn: lemma of self-nmlz epsilon eit V^1/2 w bd} absorbs the constants $\Viiup$, $\Viilow$, and $K_w$.
Comparing $\exp[-c \lambda_{\min}(\bV) n]$ in \eqref{eqn: lemma of self-nmlz epsilon wtw tail bd} and $\exp[-c \lambda_{\min}(\bV) t^2]$ in \eqref{eqn: lemma of self-nmlz epsilon eit V^1/2 w bd}, we consider two cases:
\begin{enumerate}
    \item If ${t^2} \leq {n}$, then 
$\exp[-c \lambda_{\min}(\bV) n] \leq \exp[-c \lambda_{\min}(\bV) t^2]$
and therefore, 
\beq
\P \lbk \sqrt{n} \labs \bfe_i\mt\bv \rabs \geq t \rbk \leq 4 \exp[-c \lambda_{\min}(\bV) t^2 ].
\eeq
\item If ${t^2} > {n}$, then $t/\sqrt{n} > 1$. 
So we have $\p \{ \sqrt{n} |\bfe_i\mt \bv| \geq t \} \leq \p\{ |\bfe_i\mt \bv| > 1 \}$.
However, we always have
\beq
| \bfe_i\mt\bV^{1/2}\bw | \leq \sqrt{\bfe_i\mt\bfe_i}\sqrt{\bw\mt \bV \bw}= \sqrt{\bw\mt \bV \bw},
\eeq
which implies that $| \bfe_i\mt\bv | \leq 1$.
Hence we have
$\P \{ \sqrt{n} | \bfe_i\mt\bv | \geq t \} = 0$, which is still bounded by $ 4 \exp[-c \lambda_{\min}(\bV) t^2]$. 
\end{enumerate}

Combining cases 1 and 2 above, we obtain the desired result.

\subsection{Proof of Lemma \ref{lemma: non stable for non standardized epsilon}}

The proof of Lemma \ref{lemma: non stable for non standardized epsilon} (a) is derived from the standard Hanson--Wright inequality (Theorem 6.2.1 in \cite{vershynin2018high}). For Lemma \ref{lemma: non stable for non standardized epsilon} (b), the proof follows from the asymptotic theory results reviewed in Section \ref{subsubsec: lemmas for the asymp theory}.

\paragraph{Part (a)}  
For any $r\geq 0$, we have
\begin{eqnarray}
   & & \P \lbk \labs \delta - 1 \rabs \geq \max(r,r^2) \rbk \nonumber \\  
 &=&  \P \lbk \labs \bw\mt\bV\bw - \mathbb{E}\lsk \bw\mt\bV\bw \rsk \rabs \geq \max(r,r^2)n \rbk 
\label{eqn: typical Hanson-Wright step i}
\\
& {\leq} & 2 \exp \lbk -c \min \lmk \frac{ {\max}^2 \lsk r^2,r \rsk n^2}{\tr(\bV^2)}, \frac{\max \lsk r^2,r \rsk n}{\lambda_{\max}\lsk \bV \rsk} \rmk \rbk 
\label{eqn: typical Hanson-Wright step ii}
\\
&\leq & 2 \exp \lbk - c \min \lmk {\max}^2  \lsk r^2,r \rsk, \max\lsk r^2,r \rsk \rmk \min \lmk \frac{n^2}{\tr(\bV^2)}, \frac{n}{\lambda_{\max}(\bV)} \rmk \rbk \nonumber \\
& {\leq} &  2 \exp \lmk - \frac{ c n r^2 }{\lambda_{\max}(\bV)} \rmk,
\label{eqn: typical Hanson-Wright step iii} 
\end{eqnarray}
where  \eqref{eqn: typical Hanson-Wright step i} follows from $\bepsilon = \sigma \bV^{1/2} \bw$ in Assumption \ref{assumption::regularity-conditions} (i) and  $\tr(\bV)=n$, 
\eqref{eqn: typical Hanson-Wright step ii} follows from the Hanson--Wright inequality in Lemma \ref{lemma: Hanson-Wright inequality} below, with $c$ only depending on $K_w$,  
and \eqref{eqn: typical Hanson-Wright step iii} follows from the fact that 
for any $r\geq 0$, $\min [ \max^2 ( r^2,r ), \max(r^2,r) ] = r^2$ and  $\tr(\bV^2)\leq 
n\lambda_{\max}(\bV)$, so that
$
n^2 / \tr(\bV^2) 
\geq 
n^2 /[ n \lambda_{\max}(\bV) ] 
= 
n / \lambda_{\max}(\bV)$.

By \eqref{eqn: typical Hanson-Wright step iii}, if $\lambda_{\max}(\bV) \asymp n^{\iota}$ with $\iota \in [0,1)$, we have $\delta \overset{P}{\longrightarrow} 1$.

\paragraph{Part (b)} 
Let $\bV=\bQ\bLambda\bQ\mt$, where $\bQ=[\bq_1 ~ \ldots ~ \bq_K ~ \bq_{K+1} ~ \ldots ~ \bq_n]$ with $\bq_i \in \mathbb{R}^n$ denoting the eigenvector corresponding to $\lambda_i(\bV)$ for $i=1,\dots,n$, and $\bLambda=\mathrm{diag}\{ \lambda_1(\bV),\ldots,$ $ \lambda_K(\bV), $ $\lambda_{K+1}(\bV), \ldots, \lambda_n(\bV) \}$. Then we have $\bV=\bV_K + \bV_{-K}$ where
\begin{equation}
    \bV_K=\sum_{i=1}^K \lambda_i(\bV) \bq_i\bq_i^{\T},\ \bV_{-K}=\sum_{j=K+1}^n \lambda_j(\bV)\bq_j\bq_j^{\T}.
    \label{eqn: V_K and V_-K}
\end{equation}
Thus, $\bw\mt\bV\bw=\bw\mt\bV_K\bw+\bw\mt\bV_{-K}\bw$. 

First, we consider $n\inverse ( \bw\mt \bV_{-K} \bw )$. For any $r\geq  0$, \allowdisplaybreaks
\begin{eqnarray}
& &\p \lbk \labs n\inverse ( \bw\mt \bV_{-K} \bw ) - \lsk  1-\sum_{i=1}^K \alpha_i \rsk \rabs \geq r \rbk
\nonumber\\
&\leq & 
\p \lbk \labs n\inverse ( \bw\mt\bV_{-K}\bw )  -n\inverse \e ( \bw\mt\bV_{-K}\bw  ) \rabs + 
\labs n\inverse \e ( \bw\mt\bV_{-K}\bw ) - \lsk  1 - \sum_{i=1}^K \alpha_i \rsk \rabs \geq r \rbk \nonumber\\
& = & \p \{ | \bw\mt\bV_{-K}\bw - \e ( \bw\mt \bV_{-K} \bw ) | \geq n  [ r- o(1) ]  \} 
\label{eqn: pf of wt V_-K w concentration used twice step i}
\\
& \leq & 2\exp \lsk \frac{- c n}{\lambda_{K+1}(\bV)} \min \lbk \lmk r- o(1)\rmk^2, \lmk r- o(1)\rmk \rbk \rsk,
\label{eqn: pf of wt V_-K w concentration used twice step ii}
\end{eqnarray}
where \eqref{eqn: pf of wt V_-K w concentration used twice step i} follows from $\tr(\bV) = n$ and
\beq
\labs \frac{\e  (\bw\mt\bV_{-K}\bw )}{n} - \lsk 1 - \sum_{i=1}^K \alpha_i \rsk \rabs
&=& 
\labs \frac{\tr (\bV_{-K})}{n} - \lsk 1 - \sum_{i=1}^K \alpha_i \rsk \rabs 
\\
&=&
\labs \lmk 1 - \frac{ \sum_{i=1}^K \lambda_i(\bV)}{n} \rmk - \lsk 1 - \sum_{i=1}^K \alpha_i \rsk \rabs
\\
&=&
o(1),
\eeq
and \eqref{eqn: pf of wt V_-K w concentration used twice step ii} follows from proofs similar to \eqref{eqn: typical Hanson-Wright step ii} and \eqref{eqn: typical Hanson-Wright step iii} which uses the Hanson--Wright inequality in Lemma \ref{lemma: Hanson-Wright inequality} below.

Next we consider $n\inverse( \bw\mt\bV_K\bw )$. Denote $[\bq_1\ \dots\ \bq_K]\mt = [\bvphi_1\ \dots\ \bvphi_n] \in \mathbb{R}^{K \times n}$, where $\bvphi_i=[q_{i1}\ \dots\ q_{iK}]\mt \in \mathbb{R}^K$ for $i=1,\dots,n$. 
By the Cauchy--Schwarz inequality, we have
\beq
|q_{ij}|
=
|\bfe_i\mt\bq_j|
=
|\bfe_i\mt \bV^{1/2} \bV^{-1/2}\bq_j| 
\leq 
(\bq_j\mt \bV\inverse \bq_j)^{1/2}(\bfe_i\mt\bV\bfe_i )^{1/2} 
= 
[ \lambda_j\inverse(\bV) V_{ii} ]^{1/2},
\eeq
for $j=1,\ldots, K$,  which implies that
\begin{equation}
\bvphi_i\mt \bvphi_i 
= 
q_{i1}^2 + \dots + q_{iK}^2 
\leq 
V_{ii} [ \lambda_1\inverse (\bV) + \ldots + \lambda_K\inverse (\bV)  ]
\leq 
\lambda_K\inverse(\bV) K V_{ii}.
\label{eqn: K Vii over lambda_K bd}
\end{equation}
Under Assumption \ref{assumption::regularity-conditions} (iv), we have $\Viilow / \Viiup \leq  V_{ii} \leq \Viiup / \Viilow$ for $ i=1,\dots,n$.
Define
\begin{equation}
\bS_n
:=\begin{bmatrix}
\bvphi_1 & \dots & \bvphi_n
\end{bmatrix}\bw = \sum_{i=1}^n w_i \bvphi_i,
\label{eqn: Sn = sum of phi_i w_i}
\end{equation}
so that
\begin{equation}
   \frac{\bw\mt\bV_K\bw}{n}= \bS_n\mt \mathrm{diag}\lbk \frac{\lambda_1(\bV)}{n},\dots,\frac{\lambda_K(\bV)}{n} \rbk \bS_n.
\label{eqn: wt V_K w / n} 
\end{equation}

We will use the Cramér--Wold  device to show the asymptotic normality of $\bS_n$, that is, we will show that for any real vector $\bb \in \mathbb{R}^K$,
$
\bb\mt\bS_n=\sum_{i=1}^n w_i \bb\mt \bvphi_i \overset{L}{\longrightarrow} \mathcal{N}(0, \bb\mt\bI_K \bb)
$.
Since $\bvphi_i$, $i=1,\dots,n$ changes with $\bV$ as $n$ changes, $\{ w_i \bb\mt \bvphi_i \}_{i=1}^n$ forms a triangular array as $n$ increases.
So we will use the Lyapounov Central Limit Theorem (Lemma \ref{lemma: Lyapounov triangle CLT} in Section \ref{subsubsec: lemmas for the asymp theory}) for this triangular array.
Here are some basic facts: 
\begin{enumerate}
    \item For each $n$, $w_i\bb\mt\bvphi_i,\ i=1,\dots,n$, are independent;
    \item $\e( w_i\bb\mt\bvphi_i ) = 0$ and $
\e ( w_i \bb \mt \bvphi_i )^2 
= 
(\bb\mt \bvphi_i)^2
\leq 
(\bb\mt \bb) (\bvphi_i\mt \bvphi_i)
< 
\infty
$;
\item we have 
\begin{eqnarray}
& & \lmk \sum_{i=1}^n \e (w_i \bb\mt \bvphi_i)^2 \rmk^{-3/2}\sum_{i=1}^n \e \labs w_i\bb\mt\bvphi_i \rabs^3 \nonumber\\ 
&=&
( \bb\mt\bb )^{-3/2}\sum_{i=1}^n \e | w_i\bb\mt\bvphi_i |^3
\label{eqn: btb^(-3/2) sum of E|wibtphi_i|^3 step i}
\\
& \leq & 
( \bb\mt\bb )^{-3/2} (\max_i \e | w_i |^3) \sum_{i=1}^n | \bb\mt\bvphi_i |^3  \nonumber \\
&\leq &
( \bb\mt\bb )^{-3/2} C \sum_{i=1}^n  ( \bb\mt\bb)^{3/2} (\bvphi_i\mt\bvphi_i) ^{3/2} 
\label{eqn: btb^(-3/2) sum of E|wibtphi_i|^3 step ii}\\
&\leq& C  K^{3/2} \Viiup^{3/2} n  \lambda_K^{-3/2}(\bV) \longrightarrow 0,
\label{eqn: btb^(-3/2) sum of E|wibtphi_i|^3 step iii}
\end{eqnarray}
where \eqref{eqn: btb^(-3/2) sum of E|wibtphi_i|^3 step i} follows from 
\beq
\sum_{i=1}^n\e \lsk w_i \bb\mt\bvphi_i \rsk^2 
= \bb\mt\lsk \sum_{i=1}^n \bvphi_i\bvphi_i\mt \rsk \bb 
= \bb\mt \begin{bmatrix}
\bq_1 & \dots & \bq_K
\end{bmatrix}\mt \begin{bmatrix}
\bq_1 & \dots & \bq_K
\end{bmatrix}\bb=\bb\mt\bb, 
\eeq
\eqref{eqn: btb^(-3/2) sum of E|wibtphi_i|^3 step ii} follows from the fact that under Assumption \ref{assumption::regularity-conditions} (i), $\max_i \e | w_i |^3$ is bounded by $C$ related to $K_w$ and the  Cauchy--Schwarz inequality,
\eqref{eqn: btb^(-3/2) sum of E|wibtphi_i|^3 step iii} follows from \eqref{eqn: K Vii over lambda_K bd} and $\sum_{i=1}^n V_{ii}^{3/2} \leq n V_{\mathrm{up}}^{3/2}$.
Since $\lambda_K(\bV) \asymp n$, \eqref{eqn: btb^(-3/2) sum of E|wibtphi_i|^3 step iii} converges to zero.
\end{enumerate}

By the Lyapounov Central Limit Theorem in Lemma \ref{lemma: Lyapounov triangle CLT} below, we have 
\beq
(\bb\mt \bb)^{-1/2}
\sum_{i=1}^n w_i \bb\mt \bvphi_i 
\overset{L}{\longrightarrow} 
\mathcal{N}(0,1) 
\Rightarrow 
\bb\mt \bS_n=\sum_{i=1}^n w_i \bb\mt\bvphi_i \overset{L}{\longrightarrow}\mathcal{N}(0, \bb\mt\bI_K\bb).
\eeq
By the Cramér--Wold Theorem in Lemma \ref{lemma: Cramer-Wold Theorem} below,
we have $\bS_n \overset{L}{\longrightarrow} [Z_1\ \dots\ Z_K]\mt$, 
where $Z_1,\dots,Z_K \overset{iid}{\sim} \mathcal{N}(0,1)$. 
By $n^{-1/2}\lambda_i^{1/2}(\bV) \to \alpha_i^{1/2}$ for $i=1,\dots,K$, Slutsky's theorem, and the definition of $\bS_n$ in \eqref{eqn: Sn = sum of phi_i w_i}, we have
\beq
\mathrm{diag}\lbk
\sqrt{\frac{\lambda_1(\bV)}{n}}, \dots,\sqrt{{\frac{\lambda_K(\bV)}{n}}}
\rbk \bS_n\overset{L}{\longrightarrow} \mathrm{diag}\lbk
\sqrt{\alpha_1},\dots, \sqrt{\alpha_K}
\rbk\lmk
Z_1\ \dots\ Z_K\rmk\mt.
\eeq
By the continuous mapping theorem and \eqref{eqn: wt V_K w / n}, we have
$
n\inverse {\bw\mt\bV_K\bw}
\overset{L}{\longrightarrow} \sum_{i=1}^K \alpha_i Z_i^2$.
Finally using Slutsky's theorem again, we have
\beq
\frac{\bw\mt \bV \bw}{n} = \frac{\bw\mt\bV_K\bw}{n} + \frac{\bw\mt\bV_{-K}\bw}{n} \overset{L}{\longrightarrow} \sum_{i=1}^K \alpha_i Z_i^2 + 1-\sum_{i=1}^K \alpha_i.
\eeq

\subsection{Other Technical Lemmas}

\subsubsection{Lemmas for the Asymptotic Theory}
\label{subsubsec: lemmas for the asymp theory}

\setcounter{lemma}{2}

\begin{lemma}
(Lyapounov Central Limit Theorem, (\cite{lehmann2005testing} Corollary 11.2.1)).
Suppose for each $n$, $\xi_{n, 1}, \dots, \xi_{n, r_n}$ are independent with $\e (\xi_{n, i}) = 0$ and $\sigma_{n, i}^2 = \e (\xi_{n, i}^2) < \infty$.
Let $s_n^2 = \sum_{i=1}^{r_n} \sigma_{n, i}^2 $.
Assume that for some $\eta > 0$, it holds that
$$ \lim_{n \to \infty} \sum_{i=1}^{r_n} \frac{1}{s_n^{2+\eta}} \e( |\xi_{n, i}^{2+\eta}| ) < \infty.$$
Then $\sum_{i=1}^{r_n} X_{n, i}/s_n \overset{L}{\longrightarrow} \mathcal{N}(0,1)$.
\label{lemma: Lyapounov triangle CLT}
\end{lemma}

\begin{lemma}
(Cramér–Wold Theorem, (\cite{patrick1995probability}, Theorem 29.4)) 
For random vectors $\{ \bxi_n \}_{n=1}^\infty  \subset \mathbb{R}^K$ and $\bxi   \in \mathbb{R}^K$, a necessary and sufficient condition for $\bxi_n \overset{L}{\longrightarrow} \bxi$ is that $\bb\mt \bxi_n \overset{L}{\longrightarrow} \bb\mt \bxi$ for all $\bb \in \mathbb{R}^K$.
\label{lemma: Cramer-Wold Theorem}
\end{lemma}

\subsubsection{Lemmas about Concentration Inequalities}

\begin{lemma}\label{lemma: sum of indep sub-Gaussians}
\textup{(Sums of independent sub-Gaussian, \citet[Proposition 2.6.1]{vershynin2018high})}
Let $\xi_1,\dots,\xi_n$ be independent, mean zero, sub-Gaussian random variables.
Then $\sum_{i=1}^n \xi_i$ is also a sub-Gaussian random variable, and
\beq
\left\| \sum_{i=1}^n \xi_i \right\|_{\psi_2}^2 \leq C \sum_{i=1}^n  \| \xi_i \|_{\psi_2}^2,
\eeq
where $C$ is an absolute constant.
\end{lemma}

\begin{lemma}\label{lemma: Bernstein's inequality}
\textup{(Bernstein's inequality, \citet[Theorem 2.8.1]{vershynin2018high})}
Let $\xi_1, \ldots, \xi_N$ be independent, mean
zero, sub-exponential random variables. Then, for every $t \geq 0$, we have,
\begin{equation}
\p \lbk \sum_{i=1}^N \xi_i \geq t \rbk \leq \exp  \lmk  - c \min \lsk  \frac{t^2}{\sum_{i=1}^N \| \xi_i \|_{\psi_1}^2},  \frac{t}{\max_i \| \xi_i \|_{\psi_1}}   \rsk  \rmk,
\label{eqn: Bernstein inequality right side bd} 
\end{equation}
where $c$ is an absolute constant.
The bound for $\p \{ \sum_{i=1}^N \xi_i \leq -t  \}$ is the same as \eqref{eqn: Bernstein inequality right side bd}.
\end{lemma}

\begin{lemma}\label{lemma: Hanson-Wright inequality}
\textup{(Hanson--Wright inequality, \citet[Theorem 6.2.1]{vershynin2018high})}
Let $\bxi = (\xi_1, \dots, \xi_n)\mt \in \mathbb{R}^n$ be a random vector with independent, mean zero, sub-Gaussian coordinates. Let $\bA \in \mathbb{R}^{n \times n}$ be a fixed matrix. Then, for every $t \geq 0$, we have
\begin{equation*}
\p \lbk \labs \bxi\mt \bA \bxi - \e \lsk \bxi\mt \bA \bxi \rsk  \rabs \geq t \rbk 
\leq 
2 \exp \lmk  -c \min \lsk  \frac{t^2}{ \max_i^4 \| \xi \|_{\psi_2}  \| \bA \|_{\textup{F}}^2 }, \frac{t}{ \max_i^2 \| \xi \|_{\psi_2}  \| \bA \| }  \rsk  \rmk.   
\end{equation*}
\end{lemma}

\begin{lemma}\label{lemma: concentration of matrix singular values}
\textup{(\citet[Theorem 5.39]{vershynin2010introduction})}
Let $\bZ = [\bz_1,\dots,\bz_n]\mt$ be an $n \times d$ matrix whose rows $\bz_i\mt$ are independent sub-Gaussian vectors such that $\e (\bz_i \bz_i\mt) = \bI_d$.
Then for every $\alpha \geq 0$, with probability at least $ 1 - 2 \exp(- \widetilde{c}_1 \alpha^2) $, we have
$$ \sqrt{n} - \widetilde{C} \sqrt{d} - \alpha \leq s_{\min}(\bZ) \leq s_{\max}(\bZ) \leq \sqrt{n} + \widetilde{C} \sqrt{d} + \alpha, $$
where $\widetilde{C},\ \widetilde{c}_1 >0$ depend on $\max_{1 \leq i \leq n} \| \bz_i \|_{\psi_2}$ and $\|\bz_i\|_{\psi_2} = \sup_{\bt \in \mathbb{R}^d} \| \bt\mt \bz_i \|_{\psi_2}$.
Additionally, if $\bZ$ has independent entries $z_{ij}$ for $i=1,\dots,n$ and $j=1,\dots,d$, then  $\widetilde{C},\ \widetilde{c}_1 >0$ depend on $K_z = \max_{1 \leq i \leq n,1 \leq j \leq d} \| z_{ij} \|_{\psi_2}$.
\end{lemma}

\begin{lemma}\label{lemma: concentration of eigenvalues of ZtZ inverse}
Assume $\bZ\in \mathbb{R}^{n \times d}$ satisfies Assumption \ref{assumption::regularity-conditions} (ii) and (iii).
Let $\widetilde{C}$ and $\widetilde{c}_1$ be absolute constants depending only on $K_z$.
\begin{enumerate}
\item[(i)] For $0 < \alpha < \sqrt{n}-\widetilde{C}\sqrt{d}$, we have
\begin{equation}
\P \lbk \lambda_{\max}\lmk \lsk \bZ\mt\bZ \rsk\inverse \rmk > \lsk \sqrt{n}-\widetilde{C} \sqrt{d}-\alpha \rsk^{-2} \rbk \leq 2 \exp \lsk -\widetilde{c}_1 \alpha^2 \rsk.
\label{eqn:concentration of lambda_max(ZtZ inverse)}
\end{equation}

\item[(ii)] For $\alpha >0$, we have
\begin{equation}
\P \lbk \lambda_{\min}\lmk \lsk \bZ\mt\bZ \rsk\inverse \rmk < \lsk \sqrt{n}+\widetilde{C}\sqrt{d}+\alpha \rsk^{-2} \rbk \leq 2 \exp \lsk -\widetilde{c}_1 \alpha^2 \rsk.
\label{eqn:concentration of lambda_min(ZtZ inverse)}
\end{equation}
\end{enumerate}
\end{lemma}

\begin{proof}[Proof of Lemma \ref{lemma: concentration of eigenvalues of ZtZ inverse}]
Since $\p \{ \bZ\mt\bZ \mathrm{\ is\ singular} \} = 0$ in Assumption \ref{assumption::regularity-conditions} (iii), we have $\lambda_{\min}[(\bZ\mt\bZ)\inverse]$ $ > 0$.

For \eqref{eqn:concentration of lambda_max(ZtZ inverse)}, since  $\lambda_{\max}[ ( \bZ\mt\bZ )\inverse  ]= 1/\lambda_{\min}( \bZ\mt\bZ )$, we have
\begin{eqnarray}
& & \P \lbk \lambda_{\max}\lmk \lsk \bZ\mt\bZ \rsk \inverse \rmk >  \lsk \sqrt{n}-\widetilde{C}\sqrt{d}-\alpha \rsk^{-2} \rbk \nonumber\\ 
&\leq &   \P \lbk \lambda_{\min}\lsk \bZ\mt\bZ \rsk < \lsk \sqrt{n}-\widetilde{C}\sqrt{d}-\alpha \rsk^2 \rbk \nonumber\\
&\leq & \P \lbk \sqrt{\lambda_{\min}\lsk \bZ\mt\bZ \rsk } < \labs \sqrt{n}-\widetilde{C}\sqrt{d}-\alpha \rabs \rbk \nonumber\\
& = &
\P \lbk s_{\min}\lsk \bZ \rsk <  \sqrt{n}-\widetilde{C}\sqrt{d}-\alpha  \rbk \\
\label{eqn: lambda_max(ZtZ inverse) bd step i}
& \leq & 2 \exp \lsk -\widetilde{c}_1 \alpha^2 \rsk,
\label{eqn: lambda_max(ZtZ inverse) bd step ii}  
\end{eqnarray}
where \eqref{eqn: lambda_max(ZtZ inverse) bd step i} follows from $\alpha < \sqrt{n}-\widetilde{C}\sqrt{d}$ so $| \sqrt{n}-\widetilde{C}\sqrt{d}-\alpha |=\sqrt{n}-\widetilde{C}\sqrt{d}-\alpha$, and \eqref{eqn: lambda_max(ZtZ inverse) bd step ii} follows from Lemma \ref{lemma: concentration of matrix singular values}. 

Since $\lambda_{\min}[ ( \bZ\mt\bZ )\inverse  ]= 1/\lambda_{\max}( \bZ\mt\bZ )$,
\eqref{eqn:concentration of lambda_min(ZtZ inverse)} can be proven by a symmetric argument.
\end{proof}

\begin{lemma}\label{lemma: concentration of vtZZtv}
Let $\widetilde{\bv} \in \mathbb{R}^n$ be any random vector distributed on the unit sphere ${S}^{n-1}$. 
Assume $\bZ\in \mathbb{R}^{n \times d}$ satisfies Assumption \ref{assumption::regularity-conditions} (ii), and $\bZ$ is independent of $\widetilde{\bv}$. 
Let $\widetilde{c}_2$ be an absolute constant depending only on $K_z$. 
Then for any $\kappa >0$, we have
\begin{equation}
\P \lbk \widetilde{\bv}\mt\bZ\bZ\mt \widetilde{\bv}>d+\kappa \rbk \leq \exp \lmk -\widetilde{c}_2 \min\lsk \kappa^2/d,\kappa \rsk \rmk,
\label{eqn: concentration of vtZZtv right side}
\end{equation}
\begin{equation}
\P \lbk \widetilde{\bv}\mt\bZ\bZ\mt \widetilde{\bv} < d - \kappa \rbk \leq \exp \lmk -\widetilde{c}_2 \min\lsk \kappa^2/d,\kappa \rsk \rmk.
\label{eqn: concentration of vtZZtv left side}
\end{equation}
\end{lemma}

\begin{proof}[Proof of Lemma \ref{lemma: concentration of vtZZtv}]
For \eqref{eqn: concentration of vtZZtv right side}, by the independence of $\widetilde{\bv}$ and $\bZ$, we have
\beq
& & \P \lbk \widetilde{\bv}\mt\bZ\bZ\mt \widetilde{\bv} > d + \kappa \rbk \\
&=&  \mathbb{E}_{\widetilde{\bv}}\mathbb{E}_{\bZ}\mathcal{I}\lbk \widetilde{\bv}\mt\bZ\bZ\mt\widetilde{\bv} > d + \kappa \rbk \\
&=& \mathbb{E}_{\widetilde{\bv}} \P_{\bZ} \lbk \widetilde{\bv}\mt\bZ\bZ\mt\widetilde{\bv}>d+\kappa \rbk,
\eeq
where the inner expectation $\e_{\bZ}$ is taken with respect to $\bZ$ while treating $\widetilde{\bv}$ as a constant.

If for any given $\widetilde{\bv} \in S^{n-1}$, $\p \{ \widetilde{\bv}\mt \bZ \bZ\mt \widetilde{\bv} > d + \kappa \} \leq \exp [ - \widetilde{c}_2 \min ( \kappa^2/ d, \kappa ) ]$, then we can prove \eqref{eqn: concentration of vtZZtv right side}.
We have
$
\widetilde{\bv} \mt\bZ\bZ\mt \widetilde{\bv} = \sum_{j=1}^d ( \widetilde{\bv}\mt \boldsymbol{z}_j )^2
$
where $\boldsymbol{z}_j \in \mathbb{R}^n,$ $j=1,\dots,d$, are the $j$th column of $\bZ$.
For any given $\widetilde{\bv} \in S^{n-1}$, $\widetilde{\bv}\mt \boldsymbol{z}_j$, $j=1,\dots,d$, are independent sub-Gaussian variables with zero mean, unit variance, and bounded sub-Gaussian norm $\| \widetilde{\bv}\mt \boldsymbol{z}_j \|_{\psi_2}^2 \leq C K_z^2$. 
Therefore, $[ ( \widetilde{\bv}\mt \boldsymbol{z}_j )^2-1 ]$, $j=1,\dots,d$ are independent sub-exponential variables with zero mean and bounded sub-exponential norm $\| [ ( \widetilde{\bv}\mt \boldsymbol{z}_j )^2-1 ] \|_{\psi_1} \leq C \| ( \widetilde{\bv}\mt \boldsymbol{z}_j )^2 \|_{\psi_1}=C\| \widetilde{\bv}\mt \boldsymbol{z}_j \|_{\psi_2}^2 \leq C K_z^2$. 
Thus given any $\widetilde{\bv} \in S^{n-1}$, apply Lemma \ref{lemma: Bernstein's inequality} and we have $\p \{ \widetilde{\bv}\mt \bZ \bZ\mt \widetilde{\bv} > d + \kappa \} \leq \exp [ - \widetilde{c}_2 \min ( \kappa^2/ d, \kappa ) ]$, which completes the proof of \eqref{eqn: concentration of vtZZtv left side}. 

By a similar argument, we can prove \eqref{eqn: concentration of vtZZtv left side}.
\end{proof}

\begin{lemma}
\textup{(\citet[2.1.b.]{lin2010probability})} \label{lemma: gaussian tail} 
Let $\Phi(\cdot)$ and $\phi(\cdot)$ denote the cumulative distribution function and probability density function of $\mathcal{N}(0,1)$.
For all $x > 0$, we have
\beq
\lsk  \frac{1}{x} - \frac{1}{x^3} \rsk \phi(x) 
<
\frac{x}{1 + x^2} \phi(x)
<
1 - \Phi(x)
<
\frac{1}{x} \phi(x).
\eeq
\end{lemma}

\subsubsection{Technical Lemmas for Theorem \ref{theorem: power}}

\begin{lemma}\label{lemma: E_1 of w}
For $\bw \in \mathbb{R}^n$ under Assumption \ref{assumption::regularity-conditions} (i), we have $\p  ( \bw\mt\bw \leq n/2 ) \leq \exp \lsk -\widetilde{c}_{4} n \rsk$, where $\widetilde{c}_{4}$ is a constant depending only on $K_w$.
\end{lemma}

\begin{proof}[Proof of Lemma \ref{lemma: E_1 of w}]
It follows from Lemma \ref{lemma: Bernstein's inequality}  as 
$
\P ( \bw\mt\bw \leq n/2 )
= \p  \{ \sum_{i=1}^n ( w_i^2 - 1 ) \leq -n/2\}
\leq \exp (-\widetilde{c}_{4} n)
$.
\end{proof}

\begin{lemma}\label{lemma: E_2 of w}
Define
$
\gamma = \sqrt{ \widetilde{c}_5 \log n}
$
with $\widetilde{c}_5 \geq 1 / \{ 2 \min ( \widetilde{c}_1, \widetilde{c}_2 ) \}$, where $\widetilde{c}_1$ is from Lemma \ref{lemma: concentration of eigenvalues of ZtZ inverse} and $\widetilde{c}_2$ is from Lemma \ref{lemma: concentration of vtZZtv}.
If $\bw \in \mathbb{R}^n$ and $\bV \in \mathbb{R}^{n \times n}$ are from Assumption \ref{assumption::regularity-conditions} (i) and $n > \exp \lsk 1/\widetilde{c}_5 \rsk$, then 
$$\p \{ \bw\mt\bV\bw - n \geq \gamma^2 n \}\leq 2/\sqrt{n}.$$
\end{lemma}

\begin{proof}[Proof of Lemma \ref{lemma: E_2 of w}] 
We have
$$
\P \lbk \bw\mt \bV \bw  - n  \geq \gamma^2 n \rbk 
 \leq   2 \exp \lbk -\frac{1}{2 \widetilde{c}_5}   \frac{ \min \lsk \gamma^4, \gamma^2  \rsk n}{\lambda_{\max}(\bV)}  \rbk 
 \leq   2 \exp \lmk - \frac{1}{2 \widetilde{c}_5} \frac{n \gamma^2 }{\lambda_{\max}(\bV)} \rmk
\leq
\frac{2}{\sqrt{n}},
$$
where the first inequality follows from the same arguments as in \eqref{eqn: typical Hanson-Wright step i}-\eqref{eqn: typical Hanson-Wright step iii}, with $c$ in \eqref{eqn: typical Hanson-Wright step ii} replaced by $(2 \widetilde{c}_5)\inverse$ for a simpler form of $\gamma$, the second inequality uses $n > \exp (1/\widetilde{c}_5)$ so that $\gamma \geq 1$, 
and the last inequality uses $n/\lambda_{\max}(\bV) \geq 1$ and the definition of $\gamma$.
\end{proof}

\begin{lemma}\label{lemma: E_3 of w}
If $v_i$ is the $i$th coordinate of $\bv$ defined in Lemma \ref{lemma: self-normlz epsilon sub-G property} and $c$ is the corresponding constant, then under Assumption \ref{assumption::regularity-conditions} (i), (iv), and $\p\{ \bepsilon = \bzero_n \} = 0$, we have
$$\p \lsk \mathrm{there\ exists\ an}\ i \in \{ 1,\dots,n \}\ \mathrm{such\ that}\ | v_i | > \sqrt{ \frac{3}{ 2c \lambda_{\min}(\bV) } } \sqrt{\frac{\log n}{n}}  \rsk \leq \frac{4}{\sqrt{n}}.$$
\end{lemma}
\begin{proof}[Proof of Lemma \ref{lemma: E_3 of w}] 
We have
\beq
& & \p \lsk \mathrm{there\ exists\ an}\ i \in\lbk 1,\dots,n\rbk\ \mathrm{such\ that}\ \labs v_i \rabs > \sqrt{ \frac{3}{2c \lambda_{\min}(\bV) } } \sqrt{\frac{\log n}{n}} \rsk\\
&\leq & 
\sum_{i=1}^n \p \lbk \sqrt{n}\labs \bfe_i\mt\bv \rabs > \sqrt{\frac{3}{2c \lambda_{\min}(\bV) }} \sqrt{\log n} \rbk \\
& {\leq } &
\sum_{i=1}^n 4 \exp \lsk -\frac{3}{2} \log n \rsk \\
&=& \frac{4n}{n^{3/2}} = \frac{4}{\sqrt{n}},
\eeq  
where the first inequality follows from the union bound and the second inequality follows from Lemma \ref{lemma: self-normlz epsilon sub-G property} in the main paper. 
\end{proof}

\subsubsection{Technical Lemmas for Corollary \ref{corollary: power pihat_G(h,rho)}}

\begin{lemma}\label{lemma: power pihat_G(h, vecrho, vecr) Concentration}
Suppose the block diagonal matrix $\bLambda'$ has the form $\bLambda'= \mathrm{diag}\{ (1-\rho_1)\bI_{n_1}, \dots, (1-\rho_K) \bI_{n_K} \} \in \mathbb{R}^{n \times n}$ for some constants $\rho_k \in [0,1)$,  $k=1,\dots, K$ and a fixed integer $K$.
The block sizes $n_k$ satisfy $\sum_{k=1}^K n_k = n$ and  $| n_k/n - r_k | \leq 1/\sqrt{n}$ for some constants $r_k \in  (0,1]$, $k=1,\dots,K$, such that $\sum_{k=1}^K r_k=1$.
Suppose $\bw = [\bw_1\mt,\dots,\bw_K\mt]\mt \sim \mathcal{N}(\bzero_n,\bI_n)$, where $\bw_k \in \mathbb{R}^{n_k}$, $k=1,\dots,K$ are sub-vectors of $\bw$.
If $n \geq \lsk 4/\min_{1\leq k \leq K} r_k \rsk^2$, then we have
\begin{equation}
\p \lbk \labs \frac{\bw\mt \bLambda' \bw}{n} - \sum_{k=1}^K r_k (1-\rho_k) \rabs > \frac{1}{2} \sum_{k=1}^K r_k (1-\rho_k) \rbk \leq 2K \exp \lmk - \widetilde{c}_6 \lsk\min_{1 \leq k \leq K} r_k \rsk n \rmk.
\label{eqn: wt Lambda' w / n concentration conclusion}    
\end{equation}
\end{lemma}
\begin{proof}[Proof of Lemma \ref{lemma: power pihat_G(h, vecrho, vecr) Concentration}]
On the left-hand side of \eqref{eqn: wt Lambda' w / n concentration conclusion},  
we have
\beq
\labs \frac{\bw\mt \bLambda' \bw}{n} - \sum_{k=1}^K r_k (1-\rho_k) \rabs
>
\frac{1}{2} \sum_{k=1}^K r_k (1-\rho_k),
\eeq
which implies that
\begin{equation}
\sum_{k=1}^K (1-\rho_k)| \bw_k\mt \bw_k - n r_k |
\geq 
\labs \bw\mt\bLambda' \bw  - \sum_{k=1}^K n r_k (1-\rho_k) 
\rabs
> 
\sum_{k=1}^K \frac{1}{2} n r_k (1-\rho_k), 
\label{eqn: wt Lambda' w into sum of (1-rhok)wktwk} 
\end{equation}
by $\bw\mt \bLambda' \bw = \sum_{k=1}^K (1-\rho_k) \bw_k\mt \bw_k$.
Then \eqref{eqn: wt Lambda' w into sum of (1-rhok)wktwk} further implies that there exists $k \in \{1,\dots,K\}$ such that
\begin{equation} 
(1-\rho_k) | \bw_k\mt \bw_k - n r_k | > \frac{1}{2} n r_k (1-\rho_k).  
\label{eqn: wt Lambda' w union in words}
\end{equation}

Based on \eqref{eqn: wt Lambda' w union in words}, we apply the union bound to obtain
\begin{equation}
\p \lbk \labs \frac{\bw\mt \bLambda' \bw}{n} - \sum_{k=1}^n r_k (1-\rho_k) \rabs > \frac{1}{2} \sum_{k=1}^K r_k (1-\rho_k) \rbk \leq \sum_{k=1}^K \p \lbk \labs \bw_k\mt \bw_k - nr_k \rabs > \frac{n}{2}r_k \rbk.
\label{eqn: wt Lambda' w / n concentration union bd}    
\end{equation}
For each $k \in \lbk 1,\dots,K\rbk$, we have
\begin{eqnarray}
& &\p \lbk \labs \bw_k\mt \bw_k - nr_k \rabs > \frac{n}{2}r_k \rbk \nonumber\\
& \leq  & \p \lbk \labs \bw_k\mt \bw_k - n_k \rabs + \labs n_k - n r_k \rabs > \frac{n}{2}r_k \rbk \nonumber \\
& {\leq }& \p \lbk \labs \bw_k\mt \bw_k - n_k \rabs  > \frac{n}{4}r_k \rbk 
\label{eqn: sum nk terms of wi^2 - n rk Bernstein bd step i} 
\\
& {\leq }& 2 \exp \lmk  -\widetilde{c}_6 \min \lsk \frac{n^2 r_k^2}{n_k \|w_1^2-1\|_{\psi_1}^2}, \frac{n r_k}{\|w_1^2 - 1\|_{\psi_1}} \rsk \rmk 
\label{eqn: sum nk terms of wi^2 - n rk Bernstein bd step ii} \\
&{\leq} & 2 \exp \lmk -\widetilde{c}_6 \lsk\min_{1 \leq k \leq K} r_k \rsk n \rmk,
\label{eqn: sum nk terms of wi^2 - n rk Bernstein bd step iii}    
\end{eqnarray}
where \eqref{eqn: sum nk terms of wi^2 - n rk Bernstein bd step i} follows from $|n_k/n - r_k| \leq n^{-1/2}$ and $n(r_k/2 -1/\sqrt{n}) \geq nr_k/4$ resulted from $1/\sqrt{n} \leq r_k/4$, 
\eqref{eqn: sum nk terms of wi^2 - n rk Bernstein bd step ii} follows from the Bernstein’s inequality in Lemma \ref{lemma: Bernstein's inequality},
and \eqref{eqn: sum nk terms of wi^2 - n rk Bernstein bd step iii} follows from $r_k/( n_k/n ) \geq r_k/ ( r_k + 1/\sqrt{n} ) \geq r_k/(  r_k + r_k/4 )=4/5$, with $\|w_1^2 - 1\|_{\psi_1}$ and $4/5$ absorbed by $\widetilde{c}_6$. Combining \eqref{eqn: sum nk terms of wi^2 - n rk Bernstein bd step iii}  with \eqref{eqn: wt Lambda' w / n concentration union bd} completes the proof.
\end{proof}

\begin{lemma}\label{lemma: pwr pihat_G(h, vecrho, vecr) bd dist conc}
Suppose that $\bLambda'$, $\rho_k$, $r_k$, and $\bw$ are from Lemma \ref{lemma: power pihat_G(h, vecrho, vecr) Concentration} and $| n_k/n - r_k | \leq 1/\sqrt{n} \leq r_k/4$ for $k=1,\dots,K$. 
Then we have
\begin{eqnarray}
\mathbb{E} \lmk \frac{ \labs \sum_{k=1}^K \frac{n_k}{n} \rho_k w_k^2  - \sum_{k=1}^K r_k \rho_k w_k^2 \rabs}{ 
\sqrt{\sum_{k=1}^K \rho_k r_k w_k^2 + \sum_{k=1}^K r_k (1-\rho_k)}} \rmk 
&\leq &
\sqrt{\frac{K}{n \lsk \min_{1 \leq k \leq K} r_k \rsk}},
\label{eqn: power pihat_G(h, vecrho, vecr) Expectation bound distribution}  \\
\mathbb{E} \lmk \frac{\labs \frac{\bw\mt \bLambda' \bw}{n} - \sum_{k=1}^K r_k (1-\rho_k) \rabs}{
\sqrt{\sum_{k=1}^K \rho_k r_k w_k^2 + \sum_{k=1}^K r_k (1-\rho_k)} } \rmk 
& \leq &
\frac{20}{\sqrt{3} \widetilde{c}_6 }  \frac{\sqrt{\sum_{k=1}^K r_k(1-\rho_k)}}{\lsk \min_{1 \leq k \leq K} r_k \rsk \sqrt{n}}.
\label{eqn: power pihat_G(h, vecrho, vecr) Expectation bound concentration}
\end{eqnarray}
\end{lemma}

\begin{proof}[Proof of Lemma \ref{lemma: pwr pihat_G(h, vecrho, vecr) bd dist conc}]
For \eqref{eqn: power pihat_G(h, vecrho, vecr) Expectation bound distribution}, if $\rho_1=\dots=\rho_K=0$, the left-hand side is zero.

If there exists $\rho_k >0$, the denominator on the left-hand side is lower bounded by
$\{ ( \min_{1 \leq k \leq K} r_k )$ $\sum_{k=1}^K \rho_k w_k^2 \}^{1/2}$
and the numerator on the left-hand side satisfies
$
| \sum_{k=1}^K $ $( n_k/n ) \rho_k w_k^2  - \sum_{k=1}^K r_k \rho_k w_k^2 |
\leq 
\sum_{k=1}^K  | ( n_k/n ) - r_k | \rho_k w_k^2 
\leq 
n^{-1/2} \sum_{k=1}^K \rho_k w_k^2
$.
Combining the bounds of the numerator and denominator on the left-hand side, we have
$$
\mathrm{LHS} \leq  \frac{\mathbb{E} \sqrt{ \sum_{k=1}^K \rho_k w_k^2}}{\sqrt{n \lsk \min_{1 \leq k \leq K} r_k \rsk}}
\leq \frac{\sqrt{ \sum_{k=1}^K \rho_k \mathbb{E}w_k^2}}{\sqrt{n \lsk \min_{1 \leq k \leq K} r_k \rsk}}
\leq  \sqrt{\frac{K}{n \lsk \min_{1 \leq k \leq K} r_k \rsk}},
$$
where the second inequality follows from Jensen's inequality.

For \eqref{eqn: power pihat_G(h, vecrho, vecr) Expectation bound concentration}, on the left-hand side, the denominator is lower bounded by
$[\sum_{k=1}^K r_k(1-\rho_k)]^{1/2}$ and the numerator satisfies
\beq
\labs \frac{\bw\mt \bLambda' \bw}{n} - \sum_{k=1}^K r_k (1-\rho_k) \rabs 
= \labs \sum_{k=1}^K (1-\rho_k) \frac{\bw_k\mt \bw_k}{n} - \sum_{k=1}^K r_k (1-\rho_k) \rabs 
\leq \sum_{k=1}^K (1-\rho_k)r_k \labs   \frac{\bw_k\mt \bw_k}{n r_k} - 1 \rabs.
\eeq
Combining the bounds of the numerator and denominator on the left-hand side, we have
\begin{equation}
\mathrm{LHS} \leq \sqrt{\sum_{k=1}^K r_k(1-\rho_k)} \max_{1 \leq k \leq K} \mathbb{E} \labs  \frac{\bw_k\mt \bw_k}{n r_k} - 1 \rabs. 
\label{eqn: LHS of E(nmlzd deviation of wt Lambda' w) before integral}    
\end{equation}
It remains to bound $\e | (nr)\inverse \bw_k\mt \bw_k - 1 |$ for $k=1,\dots,K$.
Since $| (nr)\inverse \bw_k\mt \bw_k - 1 | \geq 0$, we have
\begin{eqnarray}
& &\mathbb{E} \labs {(n r_k)\inverse}{\bw_k\mt \bw_k} - 1 \rabs \nonumber\\
&=& 
\int_0^\infty \p \lbk \labs {(n r_k)\inverse}{\bw_k\mt \bw_k} - 1 \rabs > t\rbk \textup{d}t \nonumber\\
&\leq &
\int_0^\infty \p \lbk \labs \bw_k\mt \bw_k - n_k \rabs + \labs n_k - n r_k \rabs > t n r_k \rbk \textup{d}t \nonumber\\
&\leq & 
\int_0^\infty \p \lbk \labs \bw_k\mt \bw_k - n_k \rabs >  2\inverse t n r_k \rbk \textup{d}t
+
\int_0^\infty \p \lbk \labs n_k - n r_k \rabs > 2\inverse t n r_k \rbk \textup{d}t.
\label{eqn: bd of E | wkt wk / n rk - 1 | by integral}    
\end{eqnarray}

By Lemma \ref{lemma: Bernstein's inequality}, the first term in \eqref{eqn: bd of E | wkt wk / n rk - 1 | by integral} is bounded by
\begin{eqnarray}
& &\int_0^\infty \p \{ | \sum_{n_k\ \mathrm{terms}} (w_i^2-1) | >  2\inverse t n r_k \} \textup{d}t \nonumber\\
&\leq &
2 \int_0^\infty \exp[ - \widetilde{c}_6 n \min ( t^2 r_k^2, t r_k )] \textup{d}t 
\label{eqn: eqn: bd of E | wkt wk / n rk - 1 | by integral 1st c6tilde} \\
& \leq &
2 \int_0^{1/r_k} \exp( - \widetilde{c}_6 n t^2 r_k^2 ) \textup{d}t
+
2 \int_{1/r_k}^{\infty} \exp( - \widetilde{c}_6 n t r_k ) \textup{d}t \nonumber \\
& \leq &
\frac{\sqrt{\pi}}{ r_k \sqrt{\widetilde{c}_6} \sqrt{n}}  
+
\frac{2}{ r_k \widetilde{c}_6  n }
\leq
\frac{4}{r_k \widetilde{c}_6 \sqrt{n}},
\label{eqn: eqn: bd of E | wkt wk / n rk - 1 | by integral 1st}   
\end{eqnarray}
where in the last inequality, we set $\widetilde{c}_6 < 1$ for a simpler upper bound, which does not affect the validity of the tail bound in \eqref{eqn: eqn: bd of E | wkt wk / n rk - 1 | by integral 1st c6tilde}.

The second term in \eqref{eqn: bd of E | wkt wk / n rk - 1 | by integral} is bounded by
\begin{equation}
\int_0^\infty \p \lbk \labs\frac{n_k}{n} - r_k\rabs > \frac{t r_k}{2} \rbk \textup{d}t
\leq 
\int_0^{\frac{2}{\sqrt{n} r_k}} \p \lbk \labs \frac{n_k}{n} - r_k \rabs > \frac{t r_k}{2} \rbk \textup{d}t
\leq
\frac{2}{\sqrt{n} r_k}
\leq 
\frac{2}{\sqrt{n} r_k \widetilde{c}_6},
\label{eqn: eqn: bd of E | wkt wk / n rk - 1 | by integral 2ed}    
\end{equation}
where the first inequality follows from $|n_k/n - r_k| \leq 1/\sqrt{n}$.

Combining \eqref{eqn: bd of E | wkt wk / n rk - 1 | by integral}, \eqref{eqn: eqn: bd of E | wkt wk / n rk - 1 | by integral 1st}, and \eqref{eqn: eqn: bd of E | wkt wk / n rk - 1 | by integral 2ed}, we have
\begin{equation}
\max_{1 \leq k \leq K} \mathbb{E} \labs  \frac{\bw_k\mt \bw_k}{n r_k} - 1 \rabs {\leq}  \frac{6}{   \widetilde{c}_6 \sqrt{n} \lsk \min_{1 \leq k \leq K} r_k \rsk}.
\label{eqn: eqn: LHS of E(nmlzd deviation of wt Lambda' w) after integral}   
\end{equation}

Plugging \eqref{eqn: eqn: LHS of E(nmlzd deviation of wt Lambda' w) after integral} into \eqref{eqn: LHS of E(nmlzd deviation of wt Lambda' w) before integral} completes the proof.
\end{proof}

\begin{lemma}\label{lemma: power pihat_G(h, vecrho, vecr) I2 I3}
If $n \geq e^2$ and $w_1,\dots,w_K$ follow i.i.d. $\mathcal{N}(0,1)$, given $0 \leq \rho_K \leq \dots \leq \rho_1 \leq 1-c_{\mathrm{min}}$, $c_{\mathrm{min}} \in (0,1)$ and $r_k$, $n_k$ from Lemma \ref{lemma: power pihat_G(h, vecrho, vecr) Concentration}, we have 
\begin{eqnarray}
\p \lbk \sum_{k=1}^K r_k \rho_k w_k^2 - \sum_{k=1}^K r_k \rho_k > \frac{\log n }{2 \widetilde{c}_6} \rbk 
& \leq & 
\frac{1}{\sqrt{n}},
\label{eqn: rkrhokwk2 concentration} \\
\p \lbk \sum_{k=1}^K \frac{n_k}{n} \rho_k w_k^2 - \sum_{k=1}^K \frac{n_k}{n} \rho_k > \frac{\log n }{2  \widetilde{c}_6 } \rbk 
& \leq & 
\frac{1}{\sqrt{n}}.
\label{eqn: nkrhokwk2/n concentration}   
\end{eqnarray}
\end{lemma}
\begin{proof}[Proof of Lemma \ref{lemma: power pihat_G(h, vecrho, vecr) I2 I3}]
For \eqref{eqn: rkrhokwk2 concentration}, if $\rho_1=0$, the left-hand side is zero. 
If $\rho_1>0$,  then 
\begin{eqnarray}
&&\p \lbk \sum_{k=1}^K r_k \rho_k w_k^2 - \sum_{k=1}^K r_k \rho_k > \frac{\log n }{2 \widetilde{c}_6 } \rbk \\
&\leq &\exp \lbk -\widetilde{c}_6 \min \lmk \frac{\lsk \frac{\log n}{2 \widetilde{c}_6} \rsk^2}{\sum_{k=1}^K \rho_k^2 r_k^2}, \frac{\frac{\log n}{2 \widetilde{c}_6}}{\max_{\footnotesize 1 \leq k \leq K} ( \rho_k r_k ) }  \rmk  \rbk \label{eqn: sum rk rhok wk^2 - sum rk rhok concentration Bernstein step 0}
\\
& {\leq} & \exp \lbk  -\widetilde{c}_6 \min \lmk \lsk \frac{\log n}{2  \widetilde{c}_6 } \rsk^2, \lsk \frac{\log n}{2 \widetilde{c}_6 } \rsk  \rmk \rbk 
\label{eqn: sum rk rhok wk^2 - sum rk rhok concentration Bernstein step i}
\\
& {=} & \exp \lsk -\widetilde{c}_6 \frac{\log n}{2 \widetilde{c}_6} \rsk = \frac{1}{\sqrt{n}},
\label{eqn: sum rk rhok wk^2 - sum rk rhok concentration Bernstein step ii}    
\end{eqnarray} 
where  \eqref{eqn: sum rk rhok wk^2 - sum rk rhok concentration Bernstein step 0} follows from Lemma \ref{lemma: Bernstein's inequality},
\eqref{eqn: sum rk rhok wk^2 - sum rk rhok concentration Bernstein step i} follows from $\sum_{k=1}^K \rho_k^2  r_k^2 \leq \sum_{k=1}^K \rho_k r_k \leq \sum_{k=1}^K r_k =1$ and $\max_{1 \leq k \leq K} \rho_k r_k \leq 1$. 
In \eqref{eqn: sum rk rhok wk^2 - sum rk rhok concentration Bernstein step ii}, we set $\widetilde{c}_6 < 1$ which does affect the validity of the tail bound in \eqref{eqn: sum rk rhok wk^2 - sum rk rhok concentration Bernstein step 0}.
Since we assumed $n \geq e^2$, $\widetilde{c}_6 < 1$ implies $n \geq e^{2\widetilde{c}_6}$. 
Hence we have $[ \log n / (2 \widetilde{c}_6) ]^2 \geq \log n / (2 \widetilde{c}_6)$.

With similar arguments, we can prove \eqref{eqn: nkrhokwk2/n concentration}.    
\end{proof}

\setcounter{equation}{0}
 \makeatletter
 \renewcommand{\theequation}{B.\@arabic\c@equation}
\makeatother

\setcounter{lemma}{0}
 \makeatletter
\renewcommand{\thelemma}{B.\@arabic\c@lemma}
\makeatother

\section{Proof of Theorem \ref{Thm: subG coverage prob}}
\label{sec: pf of all subG coverage prob}

Before giving the formal proof, we first provide some intuition for Theorem \ref{Thm: subG coverage prob}.
Recall that $\bv = \bepsilon / \| \bepsilon \|_2 = \bV^{1/2} \bw / \sqrt{\bw\mt \bV \bw}$ defined in Lemma \ref{lemma: self-normlz epsilon sub-G property} of the main paper.
We also define
\begin{equation}
    \bu_j = \frac{\bSigma^{-1/2} \bfe_j}{\sqrt{\bfe_j\mt \bSigma\inverse \bfe_j}} \in \mathbb{R}^d,
\label{eqn:u_j} 
\end{equation}
for $j=1,\dots,d$ such that $\| \bu_j \|_2 = 1$. 
Then the $t$ statistic $T_j$ can be rewritten as
\begin{eqnarray}
T_j 
&=& \frac{\widehat{\beta}_j - \beta_j}{L_j}
\nonumber \\
&=& \frac{\bfe_j\mt \lsk \bX\mt\bX\rsk\inverse \bX\mt \bepsilon}{\sqrt{\frac{\bepsilon\mt \lsk \bI_n -\bP_{\bX} \rsk \bepsilon}{n-d}} \sqrt{\bfe_j\mt \lsk \bX\mt\bX \rsk\inverse \bfe_j}} \nonumber \\
& {=} & \frac{\sqrt{n-d}\cdot \bfe_j\mt \bSigma^{-1/2} \lsk \bZ\mt\bZ \rsk\inverse \bZ\mt \bV^{1/2} \bw }{\sqrt{\bfe_j\mt \bSigma^{-1/2} \lsk \bZ\mt\bZ \rsk\inverse \bSigma^{-1/2} \bfe_j } \sqrt{\bw\mt \bV^{1/2} \lsk \bI_n-\bP_{\bZ} \rsk \bV^{1/2} \bw}}
\label{eqn: decomp of Lj^-1 ( hat(beta)_j - beta_j ) step i}
\\
& {=} & \bu_j\mt \lsk \bZ\mt \bZ \rsk\inverse \bZ\mt \bv \frac{\sqrt{n-d} }{\sqrt{\bu_j\mt \lsk \bZ\mt \bZ \rsk\inverse \bu_j}} \frac{1}{\sqrt{\bv\mt \lsk \bI_n-\bP_{\bZ} \rsk \bv}},
\label{eqn: decomp of Lj^-1 ( hat(beta)_j - beta_j ) step ii}
\end{eqnarray}
where \eqref{eqn: decomp of Lj^-1 ( hat(beta)_j - beta_j ) step i} follows from Assumption \ref{assumption::regularity-conditions} (i) and (ii).
If $n$ is large compared with $d$, we have $(\bZ\mt \bZ)\inverse \approx n\inverse \bI_d$, $[ \bu_j\mt (\bZ\mt\bZ)\inverse \bu_j]^{-1/2} \approx n^{-1/2}$, and $\bv\mt (\bI_n - \bP_{\bZ})\bv \approx 1$.

Then from \eqref{eqn: decomp of Lj^-1 ( hat(beta)_j - beta_j ) step ii}, we have 
\begin{equation}
T_j \approx \bu_j\mt \bZ\mt \bv = \sum_{i=1}^n \lsk \bu_j\mt \bz_i \rsk v_i, 
\label{eqn: Tj approx sum of ujtzi vi}
\end{equation}
where $\bz_i \in \mathbb{R}^d$ is the $i$th row of $\bZ$.
By Lemma \ref{lemma: self-normlz epsilon sub-G property}, the self-normalized quantities $v_i = \epsilon_i / \| \bepsilon \|_2$ for $i=1,\dots,n$ are approximately $n^{-1/2}$, even if  $\bepsilon$ has an unknown correlation structure.
Hence, $T_j$ in \eqref{eqn: Tj approx sum of ujtzi vi} is approximately a sum of independent elements $\bu_j\mt \bz_i$ for $i=1,\dots,n$, with weights of $n^{-1/2}$.
Thus, $\bu_j\mt \bZ\mt \bv$ in \eqref{eqn: Tj approx sum of ujtzi vi} can be treated as an asymptotically normal approximation for $T_j$.
Unlike the traditional approach, the $n^{-1/2}$ weight is provided by the self-normalized quantity $v_i$, which is robust to correlated errors.

The following lemma depicts different pieces of the Berry--Esseen bound in Theorem \ref{Thm: subG coverage prob}.

\begin{lemma}
Under Assumption \ref{assumption::regularity-conditions}, for any $\eta > 0$ and $t \in \mathbb{R}$, we have
\begin{equation}
\sup_{t \in \mathbb{R}}\labs \P \lbk T_j \leq  t \rbk - \Phi(t) \rabs \leq \P \lbk \mathcal{E}_\eta \rbk + \Delta + \frac{\eta}{\sqrt{2 \pi}}  ,
\label{eqn: berry esseen bd use E delta Delta Phi AC together}
\end{equation}
where 
\begin{eqnarray}
\mathcal{E}_{\eta} &=& \lbk \labs T_j -\bu_j\mt \bZ\mt \bv \rabs \geq \eta \rbk, \\
\label{eqn: def of berry esseen bd E delta}
\Delta &=& \sup_{t \in \mathbb{R}} \labs \P \lbk \bu_j\mt \bZ\mt \bv \leq t \rbk - \Phi(t) \rabs,
\label{eqn: def of berry esseen bd Delta}
\end{eqnarray}
and $\bv = \| \bepsilon \|_2^{-1} \bepsilon = (\bw\mt\bV\bw)^{-1/2} \bV^{1/2} \bw$
is defined in Lemma \ref{lemma: self-normlz epsilon sub-G property}.
\label{lemma: BEB decomposition}
\end{lemma}

The upper bound \eqref{eqn: berry esseen bd use E delta Delta Phi AC together} has three components.  
The $\p \{ \mathcal{E}_\eta \}$ term is the approximation error of using $\bu_j\mt \bZ\mt \bv$ to approximate $T_j$.
The $\Delta$ term is the Berry--Esseen bound of using $\bu_j\mt \bZ\mt \bv$ to approximate $\mathcal{N}(0,1)$.
The $\eta / \sqrt{2 \pi}$ term arises from the approximation errors passing from $\eta$ to $\mathcal{N}(0,1)$, which are $\Phi(t + \eta) - \Phi (t)$ and $\Phi (t) - \Phi(t - \eta)$.

\begin{proof}[Proof of Lemma \ref{lemma: BEB decomposition}]

For any $t \in \mathbb{R}$, we have 
\begin{equation}
\P \lbk \frac{\widehat{\beta}_j-\beta_j}{L_j}\leq  t \rbk \leq  \P \lbk \mathcal{E}_\eta \rbk + \P \lbk \frac{\widehat{\beta}_j-\beta_j}{L_j} \leq t\quad \mathrm{and} \quad \mathcal{E}_\eta^c \rbk 
{\leq}  \P \lbk \mathcal{E}_\eta \rbk + \P \lbk \bu_j\mt\bZ\mt\bv \leq t+\eta \rbk,
\label{eqn: berry esseen bd use E delta upper bd}
\end{equation}
where the last step in \eqref{eqn: berry esseen bd use E delta upper bd} follows from the fact that on the event $\mathcal{E}_\eta^c$, we have
\beq
\bu_j\mt \bZ\mt\bv 
\leq  \labs \bu_j\mt \bZ\mt \bv - \frac{\widehat{\beta}_j-\beta_j}{L_j} \rabs + \frac{\widehat{\beta}_j-\beta_j}{L_j}
\leq \eta + t.
\eeq
Then by the definition of $\Delta$ in \eqref{eqn: def of berry esseen bd Delta}, we have
\begin{eqnarray}
\P \lbk \bu_j\mt \bZ\mt\bv \leq  t+\eta \rbk 
\leq  \labs \P \lbk \bu_j\mt \bZ\mt\bv \leq  t+\eta \rbk - \Phi(t+\eta) \rabs + \Phi(t+\eta) 
\leq  \Delta + \Phi(t+\eta).
\label{eqn: berry esseen bd use Delta upper bd}
\end{eqnarray}
Combining \eqref{eqn: berry esseen bd use E delta upper bd} and \eqref{eqn: berry esseen bd use Delta upper bd}, we have
\begin{eqnarray}
& & \P \lbk \frac{\widehat{\beta}_j-\beta_j}{L_j} 
 \leq   t \rbk - \Phi(t) \nonumber\\
& {\leq} &
\P \lbk \mathcal{E}_\eta \rbk + \P \lbk \bu_j\mt\bZ\mt\bv \leq t+\eta \rbk - \Phi(t) 
\label{eqn: berry esseen bd use Phi AC upper bd step i}\\
& {\leq} & 
\P \lbk \mathcal{E}_\eta \rbk + \Delta + \Phi(t+\eta)-\Phi(t) 
\label{eqn: berry esseen bd use Phi AC upper bd step ii}\\
& \leq & 
\P \lbk \mathcal{E}_\eta \rbk + \Delta + \frac{\eta}{\sqrt{2 \pi}},
\label{eqn: berry esseen bd use Phi AC upper bd step iii}
\end{eqnarray}
where \eqref{eqn: berry esseen bd use Phi AC upper bd step i} follows from \eqref{eqn: berry esseen bd use E delta upper bd},
\eqref{eqn: berry esseen bd use Phi AC upper bd step ii} follows from \eqref{eqn: berry esseen bd use Delta upper bd}, and
\eqref{eqn: berry esseen bd use Phi AC upper bd step iii} follows from the mean value theorem with $\Phi'(x) \leq (2 \pi)^{-1/2}$ for all $x \in \mathbb{R}$.

Similar to \eqref{eqn: berry esseen bd use E delta upper bd}, we have
\begin{equation}
\P \lbk \bu_j\mt \bZ\mt \bv \leq  t-\eta \rbk \leq  \P \lbk \mathcal{E}_\eta \rbk + \P \lbk \bu_j\mt \bZ\mt \bv \leq  t-\eta \quad \mathrm{and} \quad \mathcal{E}_\eta^c \rbk  
{\leq}  \P \lbk \mathcal{E}_\eta \rbk + \P \lbk \frac{\widehat{\beta}_j-\beta_j}{L_j}\leq  t \rbk,
\label{eqn: berry esseen bd use E delta lower bd}
\end{equation}
where the last step in \eqref{eqn: berry esseen bd use E delta lower bd} follows from the fact that on the even $\mathcal{E}_\eta^c$, we have
\beq
\frac{\widehat{\beta}_j-\beta_j}{L_j}
\leq  \labs \frac{\widehat{\beta}_j-\beta_j}{L_j} - \bu_j\mt \bZ\mt \bv \rabs + \bu_j\mt\bZ\mt\bv
\leq \eta + t-\eta =t.
\eeq
Similar to \eqref{eqn: berry esseen bd use Delta upper bd}, we have
\begin{eqnarray}
\P \lbk \bu_j\mt \bZ\mt\bv \leq  t-\eta \rbk 
& = & \Phi(t-\eta) + \P \lbk \bu_j\mt \bZ\mt\bv \leq  t-\eta \rbk - \Phi(t-\eta) \nonumber\\
& \geq &  \Phi(t-\eta)  -\labs \P \lbk \bu_j\mt \bZ\mt\bv \leq  t-\eta \rbk - \Phi(t-\eta) \rabs  \nonumber\\
& \geq & \Phi(t-\eta) -\Delta.
\label{eqn: berry esseen bd use Delta lower bd}
\end{eqnarray}
Similar to \eqref{eqn: berry esseen bd use Phi AC upper bd step i}--\eqref{eqn: berry esseen bd use Phi AC upper bd step iii}, we have
\begin{eqnarray}
& & \P \lbk \frac{\widehat{\beta}_j-\beta_j}{L_j}\leq  t \rbk - \Phi(t) \nonumber\\
& {\geq} &
-\P \lbk \mathcal{E}_\eta \rbk + \P \lbk \bu_j\mt\bZ\mt\bv \leq t-\eta \rbk - \Phi(t) \label{eqn: berry esseen bd use Phi AC lower bd step i} \\
& {\geq} & 
-\P \lbk \mathcal{E}_\eta \rbk - \Delta + \Phi(t-\eta)-\Phi(t) 
\label{eqn: berry esseen bd use Phi AC lower bd step ii} \\
& \geq &
-\P \lbk \mathcal{E}_\eta \rbk - \Delta -  \frac{\eta}{\sqrt{2 \pi}},
\label{eqn: berry esseen bd use Phi AC lower bd step iii}
\end{eqnarray}
where \eqref{eqn: berry esseen bd use Phi AC lower bd step i} follows from \eqref{eqn: berry esseen bd use E delta lower bd}, \eqref{eqn: berry esseen bd use Phi AC lower bd step ii} follows from \eqref{eqn: berry esseen bd use Delta lower bd}, and \eqref{eqn: berry esseen bd use Phi AC lower bd step iii} follows from the mean value theorem.

Combining \eqref{eqn: berry esseen bd use Phi AC upper bd step iii} and \eqref{eqn: berry esseen bd use Phi AC lower bd step iii}, we have \eqref{eqn: berry esseen bd use E delta Delta Phi AC together}.
\end{proof}

\begin{proof}[Proof of Theorem \ref{Thm: subG coverage prob}] We prove Theorem \ref{Thm: subG coverage prob} by bounding $\Delta$ and $\p \{ \mathcal{E}_\eta \}$, $\eta/\sqrt{2\pi}$ in \eqref{eqn: berry esseen bd use E delta Delta Phi AC together} separately.

\begin{enumerate}[label=\textit{Step \Roman*}]
\item[{\bf Step 1.}]
Bound $\Delta$ in \eqref{eqn: def of berry esseen bd Delta}. 

In \eqref{eqn: Tj approx sum of ujtzi vi}, for any given $\bv$, the $ ( \bu_j\mt \bz_i )  v_i$ terms are independent with 
\beq
\mathbb{E} [ ( \bu_j\mt \bz_i ) v_i ] = 0, \quad
\cov [  ( \bu_j\mt \bz_i ) v_i ] = v_i^2, \quad
\mathbb{E}[  | \bu_j\mt \bz_i v_i|^3 ] 
=
|v_i|^3 \mathbb{E}[ | \bu_j\mt \bz_i |^3 ]
\leq 
C_1 |v_i|^3,
\eeq
where the upper bound for the third moment follows from Lemma \ref{lemma: sum of indep sub-Gaussians}:
\begin{equation}
{\|\bu_j\mt\bz_i \|^2_{\psi_2}} 
\leq
C \sum_{k=1}^d \|u_{kj} z_{ik}\|^2_{\psi_2}
= 
C \sum_{k=1}^d u_{kj}^2 \|z_{ik}\|^2_{\psi_2}
\leq 
C K_z^2 \sum_{k=1}^d u_{kj}^2 
= 
C K_z^2 = C_1,
\label{eqn: uj^T z_((i)) subG norm bd}    
\end{equation}
with $K_z$ absorbed by $C_1$ in the last step. 

So applying Berry--Esseen bound conditional on $\bv$ yields that,
\begin{eqnarray}
\Delta 
&{=}&
\sup_{t \in \mathbb{R}} \labs \mathbb{E}_{\bv} \lmk \mathbb{E}_{\bZ} \mathcal{I} \lbk \bu_j\mt \bZ\mt \bv \leq t \rbk \rmk - \Phi(t) \rabs 
\label{eqn: berry esseen bd control Delta without self-normalized lemma step 0}
\\
&=& 
\sup_{t \in \mathbb{R}} \labs \mathbb{E}_{\bv} \lmk \mathbb{E}_{\bZ} \mathcal{I} \lbk \bu_j\mt \bZ\mt \bv \leq t \rbk - \Phi(t) \rmk \rabs\nonumber \\
& \leq & \mathbb{E}_{\bv} \sup_{t \in \mathbb{R}} \labs \mathbb{E}_{\bZ} \mathcal{I} \lbk \bu_j\mt \bZ\mt \bv \leq t \rbk - \Phi(t) \rabs 
\label{eqn: berry esseen bd control Delta without self-normalized lemma step i}
\\
& \leq  & C_1 \mathbb{E}_{\bv} \lsk \sum_{i=1}^n |v_i|^3\rsk,
\label{eqn: berry esseen bd control Delta without self-normalized lemma step ii}
\end{eqnarray}
where \eqref{eqn: berry esseen bd control Delta without self-normalized lemma step 0} follows from $\bv \Perp \bZ$ and the inner expectation $\e_{\bZ}$ is taken with respect to $\bZ$ while treating ${\bv}$ as a constant,
\eqref{eqn: berry esseen bd control Delta without self-normalized lemma step i} follows from Jensen's inequality,
and \eqref{eqn: berry esseen bd control Delta without self-normalized lemma step ii} follows from the Berry--Esseen bound for non-identically distributed summands and $\sum_{i=1}^n v_i^2=1$. 
It remains to bound $\mathbb{E}_{\bv}(  \sum_{i=1}^n |v_i|^3 )$ in \eqref{eqn: berry esseen bd control Delta without self-normalized lemma step ii}.

From Lemma \ref{lemma: self-normlz epsilon sub-G property},  
$\| v_i \|_{\psi_2} \lesssim [n \lambda_{\min}(\bV)]^{-1/2}$ so that $\mathbb{E} (|v_i|^3) \lesssim [n \lambda_{\min}(\bV)]^{-3/2}$
and  $\mathbb{E} ( \sum_{i=1}^n |v_i|^3 ) \lesssim [\lambda_{\min}(\bV)]^{-3/2} n^{-1/2}$. Therefore, we have
\begin{equation}
    \Delta \leq C_1  \lmk \lambda_{\min}(\bV) \rmk^{-3/2} n^{-1/2}.
\label{eqn: berry esseen bd control Delta with self-normalized lemma}
\end{equation}

\item[{\bf Step 2.}] Bound $\p\{ \mathcal{E}_\eta \}$ and $\eta / \sqrt{2 \pi}$ together.

The $\eta$ in $\mathcal{E}_\eta$ is the approximation error of $\bu_j\mt \bZ\mt \bv$ for $(\widehat{\beta}_j - \beta_j)/L_j$.
This approximation error also appears in $\eta / \sqrt{2 \pi}$ originated from $\Phi(t + \eta) - \Phi(t)$ in \eqref{eqn: berry esseen bd use Phi AC upper bd step ii} and $\Phi(t) - \Phi(t - \eta)$ in \eqref{eqn: berry esseen bd use Phi AC lower bd step ii}.
So we will find $\eta$ such that both $\p\{ \mathcal{E}_\eta \}$ and $\eta / \sqrt{2 \pi}$ approach zero with the desired rate $\max(d, \log n)/\sqrt{n}$.

Finding such $\eta$ is started by decomposing $\p \{ \mathcal{E}_\eta \}$ as 
\begin{equation}
\p \{ \mathcal{E}_\eta \} =  p \{ \mathcal{E}_\eta \cap (\mathcal{E}_1 \cup \mathcal{E}_2) \} + \p \{ \mathcal{E}_\eta \cap \mathcal{E}_1^c \cap \mathcal{E}_2^c\} \leq 
 \p \{\mathcal{E}_1\} + \p \{ \mathcal{E}_2 \} + \p \{ \mathcal{E}_\eta \cap \mathcal{E}_1^c \cap \mathcal{E}_2^c\},
\label{eqn: subG coverage distance to normal proxy with widetilde Id PE1 PE2 PEetaE1E2C}    
\end{equation}
where 
\beq
\mathcal{E}_1 &=& \lbk \lambda_{\max}\lmk \lsk \bZ\mt\bZ \rsk\inverse \rmk > ( \sqrt{n}-\widetilde{C}\sqrt{d}-\alpha )^{-2} \rbk 
 \cup  
\lbk \lambda_{\min}\lmk \lsk \bZ\mt\bZ \rsk\inverse \rmk < ( \sqrt{n}+\widetilde{C}\sqrt{d}+\alpha )^{-2}\rbk,\\
\mathcal{E}_2
&=&
\lbk  \bv\mt \bZ\bZ\mt\bv > d+\kappa \rbk.
\eeq
The following proofs to find $\eta$ include 
(a) finding $\alpha$ in $\mathcal{E}_1$ and $\kappa$ in $\mathcal{E}_2$ such that $\p \{ \mathcal{E}_1 \}$ and $\p \{ \mathcal{E}_2 \}$ approach 0 with rate $n^{-1/2}$, 
and (b) finding $\eta \asymp \max (d, \log n) / \sqrt{n}$ with $\alpha$ and $\kappa$ from (a) such that 
$\p\{ \mathcal{E}_\eta \cap \mathcal{E}_1^c \cap \mathcal{E}_2^c \} = 0$:

\begin{enumerate}
\item 
From Lemmas \ref{lemma: concentration of eigenvalues of ZtZ inverse} and \ref{lemma: concentration of vtZZtv}, we have 
\beq
\P\lbk \mathcal{E}_1 \rbk {\leq} 4 \exp \lsk -\widetilde{c}_3 \alpha^2 \rsk, \quad
\P \lbk \mathcal{E}_2 \rbk {\leq} \exp  \{ -  \widetilde{c}_3 \min ( {\kappa^2}/{d}, \kappa ) \},
\eeq
where $\widetilde{c}_3=\min\lsk \widetilde{c}_1,\widetilde{c}_2 \rsk$ in Lemmas \ref{lemma: concentration of eigenvalues of ZtZ inverse} and \ref{lemma: concentration of vtZZtv}.

To obtain the $n^{-1/2}$ rate, for $\p \{ \mathcal{E}_1 \}$, by choosing $\exp \lsk -\widetilde{c}_3 \alpha^2 \rsk = 1/\sqrt{n}$, we find $\alpha=\sqrt{1/(2\widetilde{c}_3)}  \sqrt{\log n}:=\sqrt{c_3 \log n}$ where $c_3=1/(2\widetilde{c}_3)$.  

Similarly, for $\p \{ \mathcal{E}_2 \}$, let $\exp \{ -\widetilde{c}_3 \min \lsk {\kappa^2}/{d}, \kappa \rsk \} = 1/\sqrt{n}$. 
If $\kappa < d$, we choose $\exp \lsk -\widetilde{c}_3 \kappa^2/d \rsk = 1/\sqrt{n}$ so that $\kappa = \sqrt{c_3 d \log n}=\sqrt{d}\alpha$. 
If $\kappa \geq d$, we choose $\exp \lsk -\widetilde{c}_3 \kappa \rsk = 1/\sqrt{n}$ so that $\kappa = c_3 \log n = \alpha^2$. 
That is, $\kappa = \alpha \max(\sqrt{d}, \alpha)$.

\item With the $\alpha$ and $\kappa$ from (a), we will find $\eta$ such that $\p \{ \mathcal{E}_\eta \cap \mathcal{E}_1^c \cap \mathcal{E}_2^c \} = 0$ and $\eta \asymp \max(d, \log n) / \sqrt{n}$.

On the event $\mathcal{E}_\eta$, we have
\beq
\labs \frac{\widehat{\beta}_j-\beta_j}{L_j} - \bu_j\mt \bZ\mt \bv \rabs 
= \labs \bu_j\mt \lmk \frac{\sqrt{n-d} }{\sqrt{\bu_j\mt \lsk \bZ\mt \bZ \rsk\inverse \bu_j}} \frac{\lsk \bZ\mt\bZ \rsk\inverse}{\sqrt{\bv\mt \lsk \bI_n-\bP_{\bZ} \rsk \bv}}  - \bI_d \rmk \bZ\mt \bv \rabs.
\eeq
If we define 
\begin{equation}
\widetilde{\bI}_d= \frac{\sqrt{n-d} }{\sqrt{\bu_j\mt \lsk \bZ\mt \bZ \rsk\inverse \bu_j}} \frac{\lsk \bZ\mt\bZ \rsk\inverse}{\sqrt{\bv\mt \lsk \bI_n-\bP_{\bZ} \rsk \bv}},
\label{eqn: widetilde Id}
\end{equation}
then we have
\begin{eqnarray}
\labs \frac{\widehat{\beta}_j-\beta_j}{L_j} - \bu_j\mt \bZ\mt \bv \rabs
&=&
\labs  \bu_j\mt \lsk \widetilde{\bI}_d - \bI_d \rsk \bZ\mt\bv \rabs \nonumber\\
&=&
\labs  \frac{\bu_j\mt}{\| \bu_j \|_2} \lsk \widetilde{\bI}_d - \bI_d \rsk \frac{\bZ\mt\bv}{ \| \bZ\mt \bv \|_2 } \rabs \| \bu_j \|_2 \| \bZ\mt \bv \|_2  \nonumber \\
&\leq & 
\|\widetilde{\bI}_d - \bI_d\| \|\bZ\mt\bv\|_2 \nonumber\\
&=&  \max \lmk \labs \lambda_{\max}\lsk \widetilde{\bI}_d - \bI_d \rsk \rabs, \labs \lambda_{\min}\lsk \widetilde{\bI}_d - \bI_d \rsk \rabs \rmk \sqrt{\bv\mt\bZ\bZ\mt\bv} \nonumber \\
& = & \max \lmk \labs \lambda_{\max}\lsk \widetilde{\bI}_d \rsk -1 \rabs, \labs \lambda_{\min}\lsk \widetilde{\bI}_d \rsk -1 \rabs \rmk \sqrt{\bv\mt\bZ\bZ\mt\bv}.
\label{eqn: subG coverage distance to normal proxy with widetilde Id}    
\end{eqnarray}
On the even $\mathcal{E}_\eta \cap \mathcal{E}_1^c \cap \mathcal{E}_2^c$, by some algebra, we have
\allowdisplaybreaks
\beq
\lambda_{\max}(\widetilde{\bI}_d)-1 
& \leq & 
\frac{\sqrt{n} \sqrt{\frac{nd}{\lsk  \sqrt{n} - \widetilde{C} \sqrt{d} -\alpha \rsk^2}- d +\frac{n \max(d, \alpha^2)}{\lsk  \sqrt{n} - \widetilde{C} \sqrt{d} -\alpha \rsk^2}} + 3\alpha \sqrt{n} +3 \widetilde{C} \sqrt{dn} }{\lsk  \sqrt{n} - \widetilde{C} \sqrt{d} -\alpha \rsk^2 \sqrt{1-\frac{d+\max(d, \alpha^2)}{\lsk  \sqrt{n} - \widetilde{C} \sqrt{d} -\alpha \rsk^2}}},\\
\lambda_{\min}(\widetilde{\bI}_n)-1 
& \geq &
- \frac{\sqrt{nd} + 3\alpha \sqrt{n} + 3 \widetilde{C} \sqrt{nd} + (\alpha + \widetilde{C} \sqrt{d})^2}{\lsk \sqrt{n} + \widetilde{C} \sqrt{d} + \alpha \rsk^2},\\
\sqrt{\bv\mt\bZ\bZ\mt\bv} 
& \leq &
\sqrt{2} \max \lsk \sqrt{d}, \alpha \rsk.
\eeq

By $ d = o(\sqrt{n}) $ in Assumption \ref{assumption::regularity-conditions}, in $\lambda_{\max}(\widetilde{\bI}_d) - 1$, we have $n/(\sqrt{n} - \widetilde{C}\sqrt{d} - \alpha)^2 \rightarrow 1$ and $[d + \max(d, \alpha^2)] / (\sqrt{n} - \widetilde{C} \sqrt{d} - \alpha)^2 \rightarrow 0$.
Also in $\lambda_{\min}(\widetilde{\bI}_d) - 1$, we have $\sqrt{n} / (\sqrt{n} + \widetilde{C} \sqrt{d} + \alpha) \rightarrow 1$.
Hence there exists a positive integer $N$ such that for $n > N$, on the event $\mathcal{E}_\eta \cap \mathcal{E}_1^c \cap \mathcal{E}_2^c$, we have
\begin{eqnarray}
\labs \frac{\widehat{\beta}_j - \beta_j}{L_j} - \bu_j \bZ\mt \bv \rabs
& \leq &
\max \lmk \labs \lambda_{\max}\lsk \widetilde{\bI}_d \rsk -1 \rabs, \labs \lambda_{\min}\lsk \widetilde{\bI}_d \rsk -1 \rabs \rmk \sqrt{\bv\mt\bZ\bZ\mt\bv} \nonumber \\
& \leq & C_2  \frac{\max \lsk d, \alpha^2 \rsk}{\sqrt{n}},
\label{eqn: berry esseen bd control delta}   
\end{eqnarray}
where $C_2$ is an absolute constant depending on $\widetilde{C}$.
If $\eta$ is slightly greater than the upper bound in \eqref{eqn: berry esseen bd control delta}, $\P\{ \mathcal{E}_\eta \cap \mathcal{E}_1^c \cap \mathcal{E}_2^c \} = 0$. 

So we finally choose
\begin{equation}
\eta = (C_2 + 1)\frac{\max \lsk d, \alpha^2 \rsk}{\sqrt{n}}= (C_2 + 1)\frac{\max \lsk d,c_3 \log n \rsk}{\sqrt{n}}.
\label{eqn: eta chosen for subG coverage} 
\end{equation}
\end{enumerate}

The $\eta$ in \eqref{eqn: eta chosen for subG coverage} has the desired rate $\max (d, \log n) / \sqrt{n}$. Plugging $\P\{ \mathcal{E}_\eta \cap \mathcal{E}_1^c \cap \mathcal{E}_2^c \} = 0$ into \eqref{eqn: subG coverage distance to normal proxy with widetilde Id PE1 PE2 PEetaE1E2C}, we also have
\begin{equation}
\p \lbk \mathcal{E}_{\eta} \rbk \leq \P \{ \mathcal{E}_1 \}+\P \{ \mathcal{E}_1 \} \leq \frac{4}{\sqrt{n}} + \frac{1}{\sqrt{n}} = \frac{5}{\sqrt{n}}.
\label{eqn: berry esseen bd P(E_delta)}    
\end{equation}

\item[{\bf Step 3.}] Combine the results.
Collecting \eqref{eqn: berry esseen bd control Delta with self-normalized lemma}, \eqref{eqn: eta chosen for subG coverage}, and \eqref{eqn: berry esseen bd P(E_delta)},
we have
\beq
\P \lbk \mathcal{E}_\eta \rbk + \Delta + \frac{\eta}{\sqrt{2 \pi}} \lesssim \frac{\max \lsk d, c_3 \log n \rsk}{\sqrt{n} \lambda_{\min}^{3/2} \lsk \bV \rsk },
\eeq
which completes the proof by \eqref{eqn: berry esseen bd use E delta Delta Phi AC together}.
\end{enumerate}
\end{proof}

\setcounter{equation}{0}
 \makeatletter
 \renewcommand{\theequation}{F.\@arabic\c@equation}
\makeatother

\setcounter{lemma}{0}
 \makeatletter
\renewcommand{\thelemma}{F.\@arabic\c@lemma}
\makeatother


\setcounter{equation}{0}
 \makeatletter
 \renewcommand{\theequation}{C.\@arabic\c@equation}
\makeatother

\setcounter{lemma}{0}
 \makeatletter
\renewcommand{\thelemma}{C.\@arabic\c@lemma}
\makeatother

\section{Proofs of the Results on Power in Section \ref{all of power}}
\label{pf of all of power}

\subsection{Proof of Theorem \ref{theorem: power}}
\label{subsec: pf of power approximation}

If $h = 0$, then $\widehat{\beta}_j/L_j$ is $T_j$ in Theorem \ref{Thm: subG coverage prob}, which is proved in Section \ref{sec: pf of all subG coverage prob}. 
So in this section, we will focus on $h>0$.
We analyze the power function $\p \{ \widehat{\beta}_j / L_j > z_\alpha \} = \e \mathcal{I} \{ \widehat{\beta}_j / L_j > z_\alpha \}$ by first conditioning on $\bw$ and then averaging over the randomness of $\bw$.

Specifically, rewrite $\widehat{\beta}_j / L_j$ as $(\widehat{\beta}_j - \beta_j) / L_j + \beta_j / L_j$.
Since $\beta_j=h\sigma (\bfe_j\mt\bSigma\inverse\bfe_j)^{1/2} (n-d)^{-1/2}$, we have
\begin{equation}
\frac{\beta_j}{L_j}=\frac{h}{\sqrt{\bv\mt\lsk \bI_n - \bP_\bZ \rsk \bv }\sqrt{n\bu_j\mt\lsk \bZ\mt\bZ\rsk\inverse\bu_j} \sqrt{\bw\mt\bV\bw/n}},
\label{eqn: betaj/Lj in w and Z for pwr}    
\end{equation}
which can be approximated by $h (\bw\mt \bV \bw / n)^{-1/2}$.

From \eqref{eqn: decomp of Lj^-1 ( hat(beta)_j - beta_j ) step ii}, we have
\begin{equation}
\frac{\widehat{\beta}_j-\beta_j}{L_j}=\bu_j\mt \lsk \bZ\mt \bZ \rsk\inverse \bZ\mt \bv \frac{\sqrt{n-d} }{\sqrt{\bu_j\mt \lsk \bZ\mt \bZ \rsk\inverse \bu_j}} \frac{1}{\sqrt{\bv\mt \lsk \bI_n-\bP_{\bZ} \rsk \bv}},
\label{eqn: Tj in w and Z for pwr}    
\end{equation}
which can be approximated by $\bu_j\mt \bZ \bv$, where $\bv$ is defined in Lemma \ref{lemma: self-normlz epsilon sub-G property}.
From Theorem \ref{Thm: subG coverage prob},  $(\widehat{\beta}_j - \beta_j)/L_j $ converges weakly to $\mathcal{N}(0,1)$ with the randomness from $\bZ$ and the entries in $\bv$ are approximately $1/\sqrt{n}$.

Since $\bZ \Perp \bw$ in Assumption \ref{assumption::regularity-conditions} (iii), $h (\bw\mt \bV \bw / n)^{-1/2}$ and $\bu_j\mt \bZ \bv$ are approximately independent,
and so are $ \beta_j / L_j$ and $(\widehat{\beta}_j - \beta_j)/L_j $.
Such observation motivates us to characterize the asymptotic normality of  $(\widehat{\beta}_j - \beta_j)/L_j$ given $\bw$ first.
Then we include the randomness of $\bw$ in $h (\bw\mt \bV \bw / n)^{-1/2}$.

However, not every $\bw$ can lead to the asymptotic normality of $\bu_j\mt\bZ\bv$ in $(\widehat{\beta}_j - \beta_j)/L_j$.
We need to define a proper set of $\bw$ under which $\bu_j\mt \bZ \bv$ still converges to $\mathcal{N}(0,1)$.
Hence we define
\begin{equation}
\mathcal{E}_{\mathrm{prop}} = \mathcal{E}_1 \cap \mathcal{E}_2 \cap \mathcal{E}_3,
\label{eqn: mathcal(E)_prop of w for power}    
\end{equation}
with
\begin{eqnarray}
\mathcal{E}_1 &=& \lbk \bw\mt \bw  > n/2 \rbk,
\label{eqn: E1 of w for power}\\
\mathcal{E}_2 &=& \lbk \bw\mt\bV\bw - n < \gamma^2 n \rbk,
\label{eqn: E2 of w for power}\\
\mathcal{E}_3 &=& \lbk \forall\ i \in \lbk 1,\dots,n \rbk,\ \labs v_i \rabs \leq \sqrt{3/[2c \lambda_{\min}(\bV) ]} \sqrt{\log n/n} \rbk, 
\label{eqn: E3 of w for power}
\end{eqnarray}
and $\gamma = \sqrt{\widetilde{c}_5 \log n}$ in $\mathcal{E}_2$ is from Lemma \ref{lemma: E_2 of w}, $v_i$ and $c$ in $\mathcal{E}_3$ are from Lemma \ref{lemma: self-normlz epsilon sub-G property}.
Event $\mathcal{E}_1$ requires $\bw\mt\bw$ not to be too small to avoid the degenerate case.
Event $\mathcal{E}_2$ requires $\bw\mt \bV\bw$ not to deviate from $\e (\bw\mt \bV \bw) = n$ too far.
Event $\mathcal{E}_3$ requires $|v_i|$ not to surpass $n^{-1/2}$ too much which may fail the asymptotic normality of $\bu_j\mt \bZ\mt \bv$.

Given $\bw \in \mathcal{E}_{\mathrm{prop}}$, we characterize the approximations in \eqref{eqn: betaj/Lj in w and Z for pwr} and \eqref{eqn: Tj in w and Z for pwr} by the following events:
\begin{eqnarray}
\mathcal{E}_{\eta_1}^\bw &=& \lbk \labs \frac{\beta_j}{L_j} - \frac{h}{\sqrt{ \delta}} \rabs \geq \eta_1 \frac{h}{ \sqrt{ \delta } } \rbk, 
\label{eqn: E_eta1^w for power}
\\
\mathcal{E}_{\eta_2}^\bw &=& \lbk \labs \frac{\widehat{\beta}_j-\beta_j}{L_j} - \bu_j\mt\bZ\mt\bv \rabs \geq \eta_2 \rbk,
\label{eqn: E_eta2^w for power}
\end{eqnarray}
where the superscript $\bw$ refers to the fixed value of $\bw$, $\delta = \bw\mt\bV \bw/n$ is defined in Theorem \ref{theorem: power}, and the approximation errors are reflected by $\eta_1$ and $\eta_2$, .

The following lemma is an intermediate result for proving Theorem \ref{theorem: power}.

\begin{lemma}\label{lemma: framework of pf of power}
Under Assumption \ref{assumption::regularity-conditions}, we have 
\begin{eqnarray}
& & \labs \p \lbk \frac{\widehat{\beta}_j}{L_j} > z_\alpha \rbk - \pi(h, \bV) \rabs \nonumber\\
& \leq & 2 \e_{\bw} \lmk \mathcal{I}{ \lbk  \bw \in
\mathcal{E}_{\mathrm{prop}}^c \rbk } \rmk + \\
 && \nonumber \qquad  \e_{\bw} \lmk 
\mathcal{I}{ \lbk \bw \in \mathcal{E}_{\mathrm{prop}} \rbk } \lsk \p_\bZ\lbk \mathcal{E}_{\eta_1}^{\bw} \rbk + \p_\bZ \lbk \mathcal{E}_{\eta_2}^{\bw}\rbk + \Delta_{\bw} + \Gamma_{\bw, \eta_1, \eta_2}(h)  \rsk \rmk,
\label{eqn: framework of pf of power}    
\end{eqnarray}
where 
\begin{eqnarray}
\Delta_{\bw} &=& \sup_{t \in\mathbb{R}} \labs \p_{\bZ}\lbk \bu_j\mt\bZ\mt\bv \leq t \rbk - \Phi(t) \rabs,
\label{eqn: Delta_w for power}\\
\Gamma_{\bw, \eta_1, \eta_2}(h) 
&=&
\overline{\Phi} \lsk z_\alpha - \frac{h}{\sqrt{\delta}} - \frac{\eta_1 h }{\sqrt{\delta}} - \eta_2 \rsk
-  \overline{\Phi} \lsk z_\alpha - \frac{h}{\sqrt{\delta}} + \frac{\eta_1 h }{\sqrt{\delta}} + \eta_2 \rsk,
\label{eqn: Gamma_w,eta1,eta2(h) for power}
\end{eqnarray}
and 
$\delta = \bw\mt\bV \bw/n$ is defined in Theorem \ref{theorem: power}, 
the $\p_{\bZ}$ in $\p_\bZ\{ \mathcal{E}_{\eta_1}^{\bw} \}$ and $\p_\bZ\{ \mathcal{E}_{\eta_2}^{\bw} \}$ is the probability measure for $\bZ$ with given $\bw$,
and $\overline{\Phi}(t) = 1 - \Phi(t)$.
\end{lemma}

\begin{proof}[Proof of Lemma \ref{lemma: framework of pf of power}]
By $\bw \Perp \bZ$ in Assumption \ref{assumption::regularity-conditions} (iii), we have 
\allowdisplaybreaks
\begin{eqnarray}
& & \labs \p \lbk {\widehat{\beta}_j}/{L_j} > z_\alpha \rbk - \pi(h, \bV) \rabs \nonumber\\
&=&
\labs \e_{\bw} \lmk \e_{\bZ} \lsk \mathcal{I} \lbk  {\widehat{\beta}_j}/{L_j} > z_\alpha \rbk \rsk \rmk - \e_{\bw} \lmk \Phi\lsk \frac{h}{\sqrt{\delta}} - z_\alpha \rsk \rmk \rabs \nonumber\\
&\leq &
\e_{\bw} \labs \e_{\bZ} \lsk \mathcal{I} \lbk  {\widehat{\beta}_j}/{L_j} > z_\alpha \rbk \rsk - \Phi\lsk \frac{h}{\sqrt{\delta}} - z_\alpha \rsk \rabs \nonumber\\
&=& \e_{\bw} \labs \p_{\bZ} \lbk  {\widehat{\beta}_j}/{L_j} > z_\alpha \rbk - \overline{\Phi}\lsk z_\alpha - \frac{h}{\sqrt{\delta}} \rsk \rabs \nonumber\\
&\leq &
2 \e_{\bw} \lsk  \mathcal{I} \lbk \bw \in \mathcal{E}_{\mathrm{prop}}^c \rbk \rsk +  \e_{\bw} \lmk  \mathcal{I} \lbk \bw \in \mathcal{E}_{\mathrm{prop}} \rbk \labs  \p_{\bZ} \lbk  \frac{\widehat{\beta}_j}{L_j} > z_\alpha \rbk - \overline{\Phi}\lsk z_\alpha - \frac{h}{\sqrt{\delta}} \rsk \rabs \rmk.
\label{eqn: power framework by indep of w Z}    
\end{eqnarray}
Next we will bound the second term 
in \eqref{eqn: power framework by indep of w Z}.

Given $\bw \in \mathcal{E}_{\mathrm{prop}}$, we have
\begin{eqnarray}
& & \p_{\bZ} \lbk \bu_j\mt\bZ\mt\bv - \eta_2 + \frac{h}{\sqrt{\delta}} - \eta_1 \frac{h}{\sqrt{\delta}} > z_\alpha \rbk \nonumber\\
& \leq & 
\p_{\bZ} \lbk \bu_j\mt\bZ\mt\bv - \eta_2 + \frac{h}{\sqrt{\delta}} - \eta_1 \frac{h}{\sqrt{\delta}} > z_\alpha\ \mathrm{and}\ (\mathcal{E}_{\eta_1}^{\bw})^c \cap (\mathcal{E}_{\eta_2}^{\bw})^c \rbk + \p_{\bZ} \lbk \mathcal{E}_{\eta_1}^{\bw} \rbk + \p_{\bZ} \lbk \mathcal{E}_{\eta_2}^{\bw} \rbk \nonumber\\
&\leq & 
\p_{\bZ} \lbk \frac{\widehat{\beta}_j}{L_j} > z_\alpha \rbk + \p_{\bZ} \lbk \mathcal{E}_{\eta_1}^{\bw} \rbk + \p_{\bZ} \lbk \mathcal{E}_{\eta_2}^{\bw} \rbk, 
\label{eqn: pwr ujtZtv and h/sqrtdelta concentration left}    
\end{eqnarray}
where the last step follows from the fact that on the event $(\mathcal{E}_{\eta_1}^{\bw})^c \cap (\mathcal{E}_{\eta_2}^{\bw})^c$,
\beq
\frac{h}{\sqrt{\delta}} - \frac{\beta_j}{L_j} 
& \leq & 
\labs \frac{\beta_j}{L_j} -  \frac{h}{\sqrt{\delta}} \rabs < \eta_1 \frac{h}{\sqrt{\delta}},\\
\bu_j\mt \bZ\mt\bv - \frac{\widehat{\beta}_j - \beta_j}{L_j} 
&\leq &
\labs \frac{\widehat{\beta}_j - \beta_j}{L_j} - \bu_j\mt \bZ\mt\bv \rabs < \eta_2. 
\eeq
Then by the definition of $\Delta_{\bw}$ in \eqref{eqn: Delta_w for power}, we have
\begin{eqnarray}
& & \p_{\bZ} \lbk \bu_j\mt\bZ\mt\bv - \eta_2 + \frac{h}{\sqrt{\delta}} - \eta_1 \frac{h}{\sqrt{\delta}} > z_\alpha \rbk \nonumber \\
&=& 
\overline{\Phi} \lsk z_\alpha - \frac{h}{\sqrt{\delta}} + \frac{\eta_1 h}{\sqrt{\delta}} + \eta_2  \rsk 
+ 
\p_{\bZ} \lbk \bu_j\mt \bZ\mt \bv > z_\alpha - \frac{h}{\sqrt{\delta}} + \frac{\eta_1 h}{\sqrt{\delta}} + \eta_2  \rbk 
- 
\overline{\Phi} \lsk  z_\alpha - \frac{h}{\sqrt{\delta}} + \frac{\eta_1 h}{\sqrt{\delta}} + \eta_2  \rsk \nonumber\\
&\geq &
\overline{\Phi} \lsk z_\alpha - \frac{h}{\sqrt{\delta}} + \frac{\eta_1 h}{\sqrt{\delta}} + \eta_2  \rsk 
-
\labs \p_{\bZ} \lbk \bu_j\mt \bZ\mt \bv > z_\alpha - \frac{h}{\sqrt{\delta}} + \frac{\eta_1 h}{\sqrt{\delta}} + \eta_2  \rbk 
- 
\overline{\Phi} \lsk  z_\alpha - \frac{h}{\sqrt{\delta}} + \frac{\eta_1h}{\sqrt{\delta}} + \eta_2  \rsk \rabs \nonumber \\
& \geq &
\overline{\Phi} \lsk z_\alpha - \frac{h}{\sqrt{\delta}} + \frac{\eta_1h}{\sqrt{\delta}} + \eta_2  \rsk 
- \Delta_{\bw}.
\label{eqn: pwr ujtZtv Berry-Esseen left}    
\end{eqnarray}
Combine \eqref{eqn: pwr ujtZtv and h/sqrtdelta concentration left} and \eqref{eqn: pwr ujtZtv Berry-Esseen left} to obtain
\begin{eqnarray}
& & \p_{\bZ} \lbk \widehat{\beta}_j/L_j > z_\alpha \rbk - \overline{\Phi} \lsk z_\alpha - {h}/{\sqrt{\delta}} \rsk \nonumber\\
&\geq &
\overline{\Phi} \lsk z_\alpha - \frac{h}{\sqrt{\delta}} + \frac{\eta_1h}{\sqrt{\delta}} + \eta_2  \rsk - \overline{\Phi}\lsk  z_\alpha - \frac{h}{\sqrt{\delta}} \rsk
- \Delta_{\bw} - \p_{\bZ} \lbk \mathcal{E}_{\eta_1}^{\bw} \rbk - \p_{\bZ} \lbk \mathcal{E}_{\eta_2}^{\bw} \rbk \nonumber\\
&\geq &
-\Gamma_{\bw, \eta_1,\eta_2}(h) - \Delta_{\bw} - \p \lbk \mathcal{E}_{\eta_1}^{\bw} \rbk - \p \lbk \mathcal{E}_{\eta_2}^{\bw} \rbk.
\label{eqn: pwr given w overall approx left} 
\end{eqnarray}

Similar to \eqref{eqn: pwr ujtZtv and h/sqrtdelta concentration left}, we have
\begin{eqnarray}
\p_{\bZ} \lbk \frac{\widehat{\beta}_j}{L_j} > z_\alpha \rbk 
&\leq & 
\p_\bZ \lbk \frac{\widehat{\beta}_j - \beta_j}{L_j} +\frac{\beta_j}{L_j} > z_\alpha\ \mathrm{and}\ (\mathcal{E}_{\eta_1}^{\bw})^c \cap (\mathcal{E}_{\eta_2}^{\bw})^c \rbk + \p_\bZ \lbk \mathcal{E}_{\eta_1}^{\bw}  \rbk + \p_\bZ \lbk \mathcal{E}_{\eta_2}^{\bw}  \rbk \nonumber\\
&\leq & 
\p \lbk \bu_j\mt\bZ\mt\bv + \eta_2 + \frac{h}{\sqrt{\delta}} + \eta_1 \frac{h}{\sqrt{\delta}} > z_\alpha \rbk + \p_\bZ \lbk \mathcal{E}_{\eta_1}^{\bw}  \rbk + \p_\bZ \lbk \mathcal{E}_{\eta_2}^{\bw}  \rbk,
\label{eqn: pwr ujtZtv and h/sqrtdelta concentration right}    
\end{eqnarray}
where the last step follows from the fact that on the event $(\mathcal{E}_{\eta_1}^{\bw})^c \cap (\mathcal{E}_{\eta_2}^{\bw})^c$,
\beq
\frac{\beta_j}{L_j} - \frac{h}{\sqrt{\delta}}  
& \leq & 
\labs \frac{\beta_j}{L_j} -  \frac{h}{\sqrt{\delta}} \rabs < \eta_1 \frac{h}{\sqrt{\delta}},\\
\frac{\widehat{\beta}_j - \beta_j}{L_j}  - \bu_j\mt \bZ\mt\bv 
&\leq &
\labs \frac{\widehat{\beta}_j - \beta_j}{L_j} - \bu_j\mt \bZ\mt\bv \rabs < \eta_2. 
\eeq

Similar to \eqref{eqn: pwr ujtZtv Berry-Esseen left} , we have
\begin{eqnarray}
& & \p_{\bZ} \lbk \bu_j\mt\bZ\mt\bv + \eta_2 + \frac{h}{\sqrt{\delta}} + \eta_1 \frac{h}{\sqrt{\delta}} > z_\alpha \rbk \nonumber \\
&=& 
\overline{\Phi} \lsk z_\alpha - \frac{h}{\sqrt{\delta}} - \frac{\eta_1 h}{\sqrt{\delta}} - \eta_2  \rsk 
+ 
\p_{\bZ} \lbk \bu_j\mt \bZ\mt \bv > z_\alpha - \frac{h}{\sqrt{\delta}} - \frac{\eta_1 h}{\sqrt{\delta}} - \eta_2  \rbk 
- 
\overline{\Phi} \lsk  z_\alpha - \frac{h}{\sqrt{\delta}} - \frac{\eta_1 h}{\sqrt{\delta}} - \eta_2  \rsk \nonumber\\
&\leq &
\overline{\Phi} \lsk z_\alpha - \frac{h}{\sqrt{\delta}} - \frac{\eta_1 h}{\sqrt{\delta}} - \eta_2  \rsk 
+
\labs \p_{\bZ} \lbk \bu_j\mt \bZ\mt \bv > z_\alpha - \frac{h}{\sqrt{\delta}} - \frac{\eta_1 h}{\sqrt{\delta}} - \eta_2  \rbk 
- 
\overline{\Phi} \lsk  z_\alpha - \frac{h}{\sqrt{\delta}} - \frac{\eta_1h}{\sqrt{\delta}} - \eta_2  \rsk \rabs \nonumber \\
& \leq &
\overline{\Phi} \lsk z_\alpha - \frac{h}{\sqrt{\delta}} - \frac{\eta_1h}{\sqrt{\delta}} - \eta_2  \rsk 
+ \Delta_{\bw}.
\label{eqn: pwr ujtZtv Berry-Esseen right}    
\end{eqnarray}

Similar to \eqref{eqn: pwr given w overall approx left}, combine \eqref{eqn: pwr ujtZtv and h/sqrtdelta concentration right} and \eqref{eqn: pwr ujtZtv Berry-Esseen right} to obtain
\begin{eqnarray}
& & \p_{\bZ} \lbk \widehat{\beta}_j/L_j > z_\alpha \rbk - \overline{\Phi} \lsk z_\alpha - {h}/{\sqrt{\delta}} \rsk \nonumber\\
&\leq  &
\overline{\Phi} \lsk z_\alpha - \frac{h}{\sqrt{\delta}} -\frac{\eta_1h}{\sqrt{\delta}} - \eta_2  \rsk - \overline{\Phi}\lsk  z_\alpha - \frac{h}{\sqrt{\delta}} \rsk
+ \Delta_{\bw} + \p_{\bZ} \lbk \mathcal{E}_{\eta_1}^{\bw} \rbk + \p_{\bZ} \lbk \mathcal{E}_{\eta_2}^{\bw} \rbk \nonumber\\
&\leq &
\Gamma_{\bw, \eta_1,\eta_2}(h) + \Delta_{\bw} + \p \lbk \mathcal{E}_{\eta_1}^{\bw} \rbk +  \p \lbk \mathcal{E}_{\eta_2}^{\bw} \rbk.
\label{eqn: pwr given w overall approx right} 
\end{eqnarray}

Therefore, given $\bw \in \mathcal{E}_{\mathrm{prop}} $, combining \eqref{eqn: pwr given w overall approx left} and \eqref{eqn: pwr given w overall approx right}, we have
\begin{equation}
 \labs \p_{\bZ} \lbk  \frac{\widehat{\beta}_j}{L_j} > z_\alpha \rbk  - \overline{\Phi}\lsk z_\alpha -  \frac{h}{\sqrt{\delta}}  \rsk \rabs 
\leq 
\Gamma_{\bw, \eta_1,\eta_2}(h) + \Delta_{\bw} + \p_{\bZ} \lbk \mathcal{E}_{\eta_1}^{\bw} \rbk + \p_{\bZ} \lbk \mathcal{E}_{\eta_2}^{\bw} \rbk.
\label{eqn: pwr given w overall approx}    
\end{equation}

Plugging \eqref{eqn: pwr given w overall approx} into the second term of \eqref{eqn: power framework by indep of w Z} completes the proof.
\end{proof}

The following lemma bounds $\p_{\bZ} \{ \mathcal{E}_{\eta_1}^{\bw} \}$ and $\p_{\bZ} \{ \mathcal{E}_{\eta_2}^{\bw} \}$ given $\bw \in \mathcal{E}_{\mathrm{prop}}$ in \eqref{eqn: framework of pf of power}.

\begin{lemma}\label{lemma: E_delta1^w adn E_delta1^w given w}
Let 
\begin{equation}
\eta_1 = (C_5 + 1) \frac{\max ( \sqrt{d},\gamma )}{\sqrt{n}}, \quad
\eta_2 = (C_2 + 1) \frac{\max ( d, \gamma^2)}{\sqrt{n}},
\label{eqn: eta1 and eta2 for pwr}
\end{equation}
where $C_5$ is a sufficiently large constant depending on $\widetilde{C}$ in Lemma \ref{lemma: concentration of eigenvalues of ZtZ inverse}, $C_2$ is from \eqref{eqn: eta chosen for subG coverage},
and $\gamma = \sqrt{\widetilde{c}_5 \log n}$ is from in Lemma \ref{lemma: E_1 of w}.
Under Assumption \ref{assumption::regularity-conditions},  given $\bw \in \mathcal{E}_{\mathrm{prop}}$ defined in \eqref{eqn: mathcal(E)_prop of w for power},  
there exists a positive integer $N$ such that for any $n > N$, we have
$$\p_{\bZ} \lbk\mathcal{E}_{\eta_1}^\bw  \rbk 
\leq \frac{5}{\sqrt{n}}, \quad
\p_{\bZ} \lbk\mathcal{E}_{\eta_2}^\bw  \rbk 
\leq \frac{5}{\sqrt{n}}.$$
\end{lemma}

\begin{proof}[Proof of Lemma \ref{lemma: E_delta1^w adn E_delta1^w given w}] 
We first show the upper bound of $\p_{\bZ} \{ \mathcal{E}_{\eta_2}^\bw  \} \lesssim n^{-1/2}$ following \eqref{eqn: widetilde Id}--\eqref{eqn: berry esseen bd P(E_delta)} from the proof of Theorem \ref{Thm: subG coverage prob}. 
Specifically, we have
$$\p_{\bZ} \{ \mathcal{E}_{\eta_2}^\bw \} \leq \p_{\bZ,0} + \p_{\bZ,\eta_2},$$
where 
\begin{equation}
   \p_{\bZ,0} = \p_\bZ \lbk  (\mathcal{E}_4^{\bw})^c \cup (\mathcal{E}_5^{\bw})^c \cup (\mathcal{E}_6^{\bw})^c \rbk,\quad 
   \P_{\bZ,\eta_2}=\p_\bZ \lbk \mathcal{E}_{\eta_2}^\bw  \cap \mathcal{E}_4^{\bw} \cap \mathcal{E}_5^{\bw} \cap \mathcal{E}_6^{\bw} \rbk,
\label{eqn: P0 and P_delta2 given w} 
\end{equation}
with
\begin{eqnarray}
    \mathcal{E}_4^{\bw} &=& \lbk  \lambda_{\min}\lmk \lsk \bZ\mt\bZ \rsk\inverse\rmk \geq \lsk \sqrt{n} + \widetilde{C}\sqrt{d}+\gamma \rsk^{-2} \rbk, \nonumber\\
    \mathcal{E}_5^{\bw} &=& \lbk \ \lambda_{\max}\lmk \lsk \bZ\mt\bZ \rsk\inverse\rmk \leq \lsk \sqrt{n} - \widetilde{C}\sqrt{d}-\gamma \rsk^{-2} \rbk, \nonumber \\
    \mathcal{E}_6^{\bw} &=& \lbk  \bv\mt\bZ\bZ\mt\bv \leq d + \kappa \rbk.
    \label{eqn: good set of lambdamin or max of ZtZinverse and vtZZtv given w}
\end{eqnarray}

With fixed $\bw$, from Lemmas \ref{lemma: concentration of eigenvalues of ZtZ inverse} and \ref{lemma: concentration of vtZZtv}, we still have
\beq
\P_{\bZ,0} \leq  4 \exp \lsk -\widetilde{c}_{1} \gamma^2 \rsk  + \exp \lmk -\widetilde{c}_{2} \min \lsk \frac{\kappa^2}{d}, \kappa \rsk \rmk.
\eeq
Plugging in $\gamma = \sqrt{ \widetilde{c}_5 \log n}$ defined in Lemma \ref{lemma: E_3 of w}, we have $\exp ( -\widetilde{c}_{1} \gamma^2 ) \leq n^{-1/2}$. 
To make the term $ \exp  \{ -\widetilde{c}_{2} \min  ( \kappa^2/d, \kappa ) \} \leq n^{-1/2}$ we let $\kappa=\sqrt{d}\gamma$ if $\kappa < d$ and $\kappa=\gamma^2$ if $\kappa \geq d$. Combining $\gamma$ and $\kappa$, we have $\p_{\bZ,0} \leq 5/\sqrt{n}$.

For the probability $\p_{\bZ, \eta_2}$, note that on the intersection of the four events of $\p_{\bZ, \eta_2}$ in \eqref{eqn: P0 and P_delta2 given w}, $| L_j\inverse ( \widehat{\beta}_j - \beta_j ) - \bu_j\mt\bZ\mt\bv | \geq \eta_2$ implies that
\beq
\max \lmk \labs \lambda_{\max}\lsk \widetilde{\bI}_d \rsk -1 \rabs, \labs \lambda_{\min}\lsk \widetilde{\bI}_d \rsk -1 \rabs \rmk \sqrt{\bv\mt\bZ\bZ\mt\bv} \geq \eta_2,
\eeq
where $\widetilde{\bI}_d$ is defined in \eqref{eqn: widetilde Id}. 
Following \eqref{eqn: subG coverage distance to normal proxy with widetilde Id}--\eqref{eqn: berry esseen bd P(E_delta)} in the proof of Theorem \ref{Thm: subG coverage prob}, we set $\eta_2 = (C_2 + 1)\max( d, \gamma^2 ) / \sqrt{n}$.
Then there exists a positive integer $N$ such that for any $n > N$, we have $\p_{\bZ, \eta_2}=0$ and  $\p_\bZ \{ \mathcal{E}_{\eta_2}^\bw \} \leq 5/\sqrt{n}$.

We then bound $\p_\bZ \lbk \mathcal{E}_{\eta_1}^\bw \rbk$. 
From \eqref{eqn: betaj/Lj in w and Z for pwr}, we have
\beq
 \mathcal{E}_{\eta_1}^\bw = \lbk  \frac{h}{\sqrt{ \delta } } \labs \frac{1}{\sqrt{1-\bv\mt\bP_\bZ \bv} \sqrt{n\bu_j\mt\lsk \bZ\mt\bZ \rsk \inverse \bu_j}} -  1 \rabs \geq \eta_1 \frac{h}{\sqrt{ \delta } }\rbk 
 {\subset}\ \mathcal{E}_7^{\bw},
\eeq
where 
\beq
\mathcal{E}_7^{\bw} = \lbk \labs \lsk 1-\bv\mt\bP_\bZ \bv \rsk^{-1/2} \lmk n\bu_j\mt\lsk \bZ\mt\bZ \rsk \inverse \bu_j \rmk^{-1/2} -  1 \rabs \geq \eta_1 \rbk.
\eeq
Hence we have
$$
\p_\bZ \{ \mathcal{E}_{\eta_1}^\bw \} \leq \p \{ \mathcal{E}_7^{\bw} \}
\leq   \p_{\bZ, 0} + \p_{\bZ, +\eta_1} + \p_{\bZ, -\eta_1},
$$
where $\p_{\bZ,0}$ is defined in \eqref{eqn: P0 and P_delta2 given w}, $\p_{\bZ, +\eta_1}$ is the probability of 
\beq
\lbk \lsk 1-\bv\mt\bP_\bZ \bv \rsk^{-1/2} \lmk n\bu_j\mt\lsk \bZ\mt\bZ \rsk \inverse \bu_j \rmk^{-1/2} -  1 \geq \eta_1 \rbk
\cap  \mathcal{E}_4^{\bw}
\cap  \mathcal{E}_5^{\bw}
\cap  \mathcal{E}_6^{\bw},
\eeq 
while $\p_{\bZ, -\eta_1}$ is the probability of
\beq
\lbk \lsk 1-\bv\mt\bP_\bZ \bv \rsk^{-1/2} \lmk n\bu_j\mt\lsk \bZ\mt\bZ \rsk \inverse \bu_j \rmk^{-1/2} -  1 \leq -\eta_1 \rbk
\cap  \mathcal{E}_4^{\bw}
\cap  \mathcal{E}_5^{\bw}
\cap  \mathcal{E}_6^{\bw},
\eeq
with $\mathcal{E}_4^{\bw}$, $\mathcal{E}_5^{\bw}$, and $\mathcal{E}_6^{\bw}$ defined in \eqref{eqn: good set of lambdamin or max of ZtZinverse and vtZZtv given w}.
Therefore, we have
\beq
\p_{\bZ, +\eta_1} \leq \p_\bZ \lbk \frac{1}{\sqrt{1-\frac{d+\kappa}{\lsk \sqrt{n} - \widetilde{C} \sqrt{d} - \gamma \rsk ^2}}} \frac{1}{\sqrt{\frac{n}{\lsk \sqrt{n} + \widetilde{C} \sqrt{d} + \gamma \rsk^2}}} - 1 \geq \eta_1 \rbk.
\eeq
Since we have
\beq
& & 
\frac{1}{\sqrt{1-\frac{d+\kappa}{\lsk \sqrt{n} - \widetilde{C} \sqrt{d} - \gamma \rsk ^2}}} \frac{1}{\sqrt{\frac{n}{\lsk \sqrt{n} + \widetilde{C} \sqrt{d} + \gamma \rsk^2}}} - 1 
\\
&=&
\frac{1}{\sqrt{1 - \frac{d + \kappa}{(\sqrt{n} - \widetilde{C} \sqrt{d}- \gamma)^2}}}
\lmk  
1 - \sqrt{1 - \frac{d + \gamma}{(\sqrt{n} - \widetilde{C} \sqrt{d} - \gamma)^2}}
+
\frac{\widetilde{C} \sqrt{d} + \gamma}{\sqrt{n}}
\rmk
\\
&\leq &
C_3 \frac{\max (\sqrt{d}, \gamma)}{\sqrt{n}},
\eeq
where the last inequality holds for sufficiently large $n$ by $d = o(\sqrt{n})$ in Assumption \ref{assumption::regularity-conditions} and $(d + \kappa)/(\sqrt{n} - \widetilde{C} \sqrt{d} - \gamma)^2 \rightarrow 0$,
if we set $\eta_1 > C_3 \max (  \sqrt{d},\gamma  )/\sqrt{n}$, then $\p_{\bZ, +\eta_1} = 0$ for sufficiently large $n$. 

For $\p_{\bZ, -\eta_1}$, we have
\beq
\p_{\bZ, -\eta_1} &\leq & \p_\bZ \lbk  \lambda_{\max}\lmk \lsk \bZ\mt\bZ \rsk\inverse\rmk \leq \lsk \sqrt{n}-\widetilde{C}\sqrt{d}-\gamma \rsk^{-2},  \lmk n\bu_j\mt\lsk \bZ\mt\bZ \rsk \inverse \bu_j \rmk^{-1/2} -  1 \leq -\eta_1\rbk\\
&\leq & \p \lbk \frac{1}{\sqrt{ \frac{n}{\lsk \sqrt{n} - \widetilde{C} \sqrt{d} - \gamma \rsk^2} } }  -1 \leq -\eta_1 \rbk\\
&=& \p \lbk \frac{\widetilde{C} \sqrt{d} + \gamma}{\sqrt{n}} \geq \eta_1 \rbk\\
&\leq & \p \lbk C_4 \frac{\max ( \sqrt{d}, \gamma )}{\sqrt{n}}  \geq \eta_1 \rbk,
\eeq
where in the last inequality, we define $C_4 = \widetilde{C} + 1$.
If we set $\eta_1 > C_4 \max ( \sqrt{d},\gamma  )/\sqrt{n}$, then $\p_{\bZ, -\eta_1} = 0$.

So we choose $\eta_1 = (C_5 + 1) \max ( \sqrt{d},\gamma  )/\sqrt{n}$, where $C_5 = \max(C_3, C_4)$.
Then there exists a positive integer $N$ such that for all $n > N$, we have
$
\p_\bZ \{ \mathcal{E}_{\eta_1}^\bw \} \leq   \p_{\bZ, 0} \leq 5/\sqrt{n}
$.
Together with the bounds for $\p_\bZ \lbk \mathcal{E}_{\eta_2}^\bw \rbk$, we have proved Lemma \ref{lemma: E_delta1^w adn E_delta1^w given w}.
\end{proof}

The following lemma bounds $\Delta_{\bw}$ given $\bw \in \mathcal{E}_3$ in \eqref{eqn: framework of pf of power}.

\begin{lemma}\label{lemma: berry esseen bd conditional on E_2^c}
Under Assumption \ref{assumption::regularity-conditions}, recalling $\bw \in \mathcal{E}_3$ defined in \eqref{eqn: E3 of w for power} and $\Delta_{\bw}$ defined in \eqref{eqn: Delta_w for power}, we have
\beq
\Delta_{\bw} = \sup_{t \in\mathbb{R}} \labs \p\lbk \bu_j\mt\bZ\mt\bv \leq t \rbk - \Phi(t) \rabs \lesssim     \frac{\lsk \log n \rsk^{3/2}}{ \lambda_{\min}^{3/2}(\bV) \sqrt{n}}.
\eeq
\end{lemma}

\begin{proof}[Proof of Lemma \ref{lemma: berry esseen bd conditional on E_2^c}]
Recall that $\bu_j\mt\bZ\mt\bv= \sum_{i=1}^n \bu_j\mt \bz_i v_i$, where $\bz_i\in \mathbb{R}^d$ is the $i$th row of $\bZ$. 
As $\bw$ is given, $\bv$ defined in Lemma \ref{lemma: self-normlz epsilon sub-G property} is given, which implies that $\bu_j\mt \bz_i v_i$, $i=1,\dots,n$ are independent with zero mean, $\e (  \bu_j\mt \bz_i v_i  )^2 = v_i^2$, and
\beq
\e\labs \bu_j\mt \bz_i v_i \rabs^3 = \labs v_i \rabs^3 \e \labs \bu_j\mt \bz_i \rabs^3 {\leq} \e \labs \bu_j\mt \bz_i \rabs^3 \lmk \frac{3}{ 2c \lambda_{\min}(\bV) } \rmk^{3/2} \frac{\lsk \log n \rsk^{3/2}}{n^{3/2}},
\eeq
where the last step follows from the fact that $\bw \in \mathcal{E}_3$ defined in \eqref{eqn: E3 of w for power}.
By \eqref{eqn: uj^T z_((i)) subG norm bd}, we have $\e | \bu_j\mt \bz_i |^3 \leq C_1$ where $C_1$ is related to $K_z$.
From all these facts, the result follows by the Berry--Esseen bound for independent variables.
\end{proof} 

The following lemma bounds $\Gamma_{\bw, \eta_1, \eta_2} (h)$ given $\bw \in \mathcal{E}_1 \cap \mathcal{E}_2$ in \eqref{eqn: framework of pf of power}.

\begin{lemma}\label{lemma: Gamma_w,eta1,eta2 bd}
Under Assumption \ref{assumption::regularity-conditions}, given $\bw \in \mathcal{E}_1 \cap \mathcal{E}_2$ defined in \eqref{eqn: E1 of w for power} and \eqref{eqn: E2 of w for power}, suppose that $\eta_1 = (C_5 + 1) \max(\sqrt{d}, \gamma) / \sqrt{n}$ and $\eta_2 = (C_2 + 1) \max (d, \gamma^2) / \sqrt{n}$ are from Lemma \ref{lemma: E_delta1^w adn E_delta1^w given w} and $\gamma = \sqrt{\widetilde{c}_5 \log n}$ is defined in Lemma \ref{lemma: E_1 of w}.
Then there exists an integer $N>0$ such that for all $n > N$, we have
\beq
\Gamma_{\bw, \eta_1, \eta_2}(h) 
\leq 
\sqrt{\frac{2}{\pi}} \lbk \frac{2 \sqrt{2}}{ \sqrt{ \lambda_{\min}(\bV) } } \sqrt{1 + \gamma^2} \lmk z_\alpha + \sqrt{\log n} + C_6 \max (\sqrt{d}, \gamma) \rmk \eta_1  + \eta_2 \rbk,
\eeq
where $C_6$ is an absolute constant.
\end{lemma}

\begin{proof}[Proof of Lemma \ref{lemma: Gamma_w,eta1,eta2 bd}]
Define 
\begin{equation}
\tau(z_\alpha, \gamma) = 2 \sqrt{1 + \gamma^2} \lmk z_\alpha + \sqrt{\log n} + C_6 \max (\sqrt{d}, \gamma) \rmk.
\label{eqn: tau(z_alpha, gamma) for pwr}  
\end{equation}
We will discuss $h \geq \tau(z_\alpha, \gamma)$ and $h < \tau(z_\alpha, \gamma)$ separately.

If $h \geq \tau(z_\alpha, \gamma)$, we have
\begin{equation}
\Gamma_{\bw, \eta_1, \eta_2} (h)
\leq 
1 - \overline{\Phi} \lsk z_\alpha - \frac{h}{\sqrt{\delta}} + \eta_1 \frac{h}{\sqrt{\delta}} + \eta_2 \rsk 
=
\Phi \lsk z_\alpha - \frac{h}{\sqrt{\delta}} + \eta_1 \frac{h}{\sqrt{\delta}} + \eta_2 \rsk.
\label{eqn: large h Gamma_w,eta1,eta2 bd step 0}    
\end{equation}
Since $d = o(\sqrt{n})$ in Assumption \ref{assumption::regularity-conditions}, there exists a positive integer $N$ such that for any $n>N$,  we have $\eta_1 \leq 1/2$ and
\beq
\eta_2 
=
(C_2 + 1) \frac{\max^2(\sqrt{d}, \gamma)}{\sqrt{n}}
\leq
\frac{(C_2 + 1) \max(\sqrt{d}, \gamma)}{2 (C_5 + 1)} 
=
C_6 \max(\sqrt{d}, \gamma),
\eeq
where $C_6 = (C_2 + 1) / [2(C_5 + 1)]$.
Hence
\begin{eqnarray}
\Phi\lsk z_\alpha - \frac{h}{\sqrt{\delta}} + \eta_1 \frac{h}{ \sqrt{ \delta } } + \eta_2 \rsk 
&{\leq}& 
\Phi\lsk z_\alpha - \frac{h}{\sqrt{\delta}} + \frac{1}{2} \frac{h}{\sqrt{\delta}} + C_6 \max\{ \sqrt{d}, \gamma \} \rsk
\nonumber \\
& {\leq} & 
\Phi\lsk z_\alpha - \frac{1}{2}\frac{h}{\sqrt{1 + \gamma^2}} + C_6 \max \{\sqrt{d}, \gamma \} \rsk
\label{eqn: Phi (z_alpha - h/sqrt(delta) + eta1xx + eta2) bd}    
\\
&\leq &
\Phi \lsk -\sqrt{\log n} \rsk
\label{eqn: Phi (z_alpha - h/sqrt(delta) + eta1xx + eta2) leq Phi (- sqrt log n)}
\\
&\leq &
\frac{1}{\sqrt{2 \pi}} \frac{1}{\sqrt{ \log n} \sqrt{n}},
\label{eqn: Phi (z_alpha - h/sqrt(delta) + eta1xx + eta2) leq (2 pi log n n)pwr-0.5}
\end{eqnarray}
where \eqref{eqn: Phi (z_alpha - h/sqrt(delta) + eta1xx + eta2) bd} follows from $\bw \in \mathcal{E}_2$ defined in \eqref{eqn: E2 of w for power},
\eqref{eqn: Phi (z_alpha - h/sqrt(delta) + eta1xx + eta2) leq Phi (- sqrt log n)} replaces $h$ with $\tau(z_\alpha, \gamma)$ in \eqref{eqn: tau(z_alpha, gamma) for pwr},
and
\eqref{eqn: Phi (z_alpha - h/sqrt(delta) + eta1xx + eta2) leq (2 pi log n n)pwr-0.5} follows from Lemma \ref{lemma: gaussian tail}. 

Combining \eqref{eqn: large h Gamma_w,eta1,eta2 bd step 0}--\eqref{eqn: Phi (z_alpha - h/sqrt(delta) + eta1xx + eta2) leq (2 pi log n n)pwr-0.5}, for $h \geq \tau(z_\alpha, \gamma)$, we have, 
\begin{equation}
\Gamma_{\bw, \eta_1, \eta_2}(h) \leq \frac{1}{\sqrt{2 \pi}} \frac{1}{\sqrt{ \log n} \sqrt{n}}.
\label{eqn: large h Gamma_w,eta1,eta2 bd}
\end{equation}

If $h < \tau(z_\alpha, \gamma)$, we have
\begin{eqnarray}
\Gamma_{\bw, \eta_1, \eta_2}(h) 
& \leq & 
\sqrt{{2}/{\pi}} \lsk {h} \eta_1 / {\sqrt{\delta}}  + \eta_2 \rsk 
\label{eqn: small h Gamma_w,eta1,eta2 bd step 0}\\
& \leq &
\sqrt{{2}/{\pi}} \lsk \sqrt{ {2}/{\lambda_{\min}(\bV)} }\tau(z_\alpha, \gamma) \eta_1  + \eta_2  \rsk 
\label{eqn: small h Gamma_w,eta1,eta2 bd step i}
\\
& \leq &
\sqrt{\frac{2}{\pi}} \lbk \frac{2 \sqrt{2}}{ \sqrt{ \lambda_{\min}(\bV) } } \sqrt{1 + \gamma^2} \lmk z_\alpha + \sqrt{\log n} + C_6 \max (\sqrt{d}, \gamma) \rmk \eta_1  + \eta_2 \rbk.
\label{eqn: small h Gamma_w,eta1,eta2 bd step ii}   
\end{eqnarray}
where \eqref{eqn: small h Gamma_w,eta1,eta2 bd step 0} follows from the mean value theorem, 
\eqref{eqn: small h Gamma_w,eta1,eta2 bd step i} follows from $\delta \geq \lambda_{\min}(\bV) \bw\mt \bw/n$ and $\bw\mt\bw/n > 1/2$ in $\mathcal{E}_1$ defined in \eqref{eqn: E1 of w for power},
and \eqref{eqn: small h Gamma_w,eta1,eta2 bd step ii} follows from \eqref{eqn: tau(z_alpha, gamma) for pwr}.

Comparing \eqref{eqn: small h Gamma_w,eta1,eta2 bd step ii} and \eqref{eqn: large h Gamma_w,eta1,eta2 bd}, we choose the larger one \eqref{eqn: small h Gamma_w,eta1,eta2 bd step ii} to complete the proof.
\end{proof}

\begin{proof}[Proof of Theorem \ref{theorem: power}]
We prove Theorem \ref{theorem: power} by bounding the terms in \eqref{eqn: framework of pf of power}. 
From Lemmas \ref{lemma: E_1 of w}--\ref{lemma: E_3 of w}, we have
\begin{equation}
\e_{\bw} \lmk \mathcal{I}\lbk \bw \in \mathcal{E}_{\mathrm{prop}}^c \rbk \rmk  
= 
\p \lbk \bw \in \mathcal{E}_{\mathrm{prop}}^c \rbk 
\leq
\p \lsk \mathcal{E}_1^c \rsk + \p \lsk \mathcal{E}_2^c \rsk + \p \lsk \mathcal{E}_3^c \rsk
\leq 
\exp \lsk -\widetilde{c}_4 n \rsk + {2}/{\sqrt{n}} + {4}/{\sqrt{n}}.
\label{eqn: bd of P(w in E_prop^c)}    
\end{equation}
From Lemma \ref{lemma: E_delta1^w adn E_delta1^w given w}, we have
\begin{equation}
\e_\bw \lmk \mathcal{I}\lbk \bw \in \mathcal{E}_{\mathrm{prop}} \rbk \lsk  \p_{\bZ}\lbk \mathcal{E}_{\eta_1}^{\bw} \rbk + \p_{\bZ}\lbk \mathcal{E}_{\eta_2}^{\bw} \rbk \rsk \rmk \lesssim n^{-1/2}.
\label{eqn: bd of Ew I(w in E_prop) P(E_eta1^w) + P(E_eta2^w)}    
\end{equation}
From Lemma \ref{lemma: berry esseen bd conditional on E_2^c}, we have
\begin{equation}
\e_\bw \lmk \mathcal{I}\lbk \bw \in \mathcal{E}_{\mathrm{prop}} \rbk \Delta_{\bw} \rmk \lesssim     \frac{\lsk \log n \rsk^{3/2}}{ \lambda_{\min}^{3/2}(\bV) \sqrt{n}}.
\label{eqn: bd of Ew I(w in E_prop) Delta_w }    
\end{equation}
From Lemma \ref{lemma: Gamma_w,eta1,eta2 bd}, we have
\begin{eqnarray}
\e_\bw \lmk \mathcal{I}\lbk \bw \in \mathcal{E}_{\mathrm{prop}} \rbk \Gamma_{\bw, \eta_1,\eta_2}(h) \rmk
& \leq & 
\sqrt{\frac{2}{\pi}} \lbk \frac{2 \sqrt{2}}{ \sqrt{ \lambda_{\min}(\bV) } } \sqrt{1 + \gamma^2} \lmk z_\alpha + \sqrt{\log n} + C_6 \max (\sqrt{d}, \gamma) \rmk \eta_1  + \eta_2 \rbk \nonumber\\
& \lesssim &
 \frac{\sqrt{\log n} \max (d, \log n)}{\sqrt{\lambda_{\min}(\bV)} \sqrt{n}},
\label{eqn: bd of Ew I(w in E_prop) Gamma_w,eta1,eta2}      
\end{eqnarray}  
where $\eta_1 = (C_5 + 1) {\max ( \sqrt{d},\gamma )}/{\sqrt{n}}$ and $
\eta_2 = (C_2 + 1) {\max ( d, \gamma^2)}/{\sqrt{n}}$ are defined in \eqref{eqn: eta1 and eta2 for pwr}, and $\gamma = \sqrt{\widetilde{c}_5 \log n}$ is defined in Lemma \ref{lemma: E_1 of w}.

By Lemmas \ref{lemma: E_1 of w}--\ref{lemma: E_3 of w} and Lemmas \ref{lemma: E_delta1^w adn E_delta1^w given w}--\ref{lemma: Gamma_w,eta1,eta2 bd},
collecting all the bounds in \eqref{eqn: bd of P(w in E_prop^c)}--\eqref{eqn: bd of Ew I(w in E_prop) Gamma_w,eta1,eta2} completes the proof.
\end{proof}

\subsection{Proofs of Corollary \ref{corollary: power pihat_G(h,rho)} }
\label{subsec: pf of power pihat_G(h, rho)}

We will prove a more general corollary below with Corollary \ref{corollary: power pihat_G(h,rho)} being a special case when $K=1$:

\begin{corollary}
\textup{(Block-diagonal correlation structure)}
Under Assumption \ref{assumption::regularity-conditions}, 
assume $\bw \sim \mathcal{N} ( \bzero_n, \bI_n )$. 
Let $\bV=\diag\{ \bV_1\ \dots\ \bV_K \}\in \mathbb{R}^{n \times n}$, 
where $\bV_k=\rho_k \bone_{n_k}\bone_{n_k}\mt + (1-\rho_k) \bI_{n_k}$ for $k=1,\dots,K$, and $K$ is a constant integer. 
The sizes of diagonal blocks in $\bV$ satisfy $\sum_{k=1}^K n_k=n$, and for any $k=1,\dots,K$, $| n_k/n - r_k | \leq 1/\sqrt{n}$, where $r_k \in (0,1]$ are constants for $k=1,\dots,K$ such that $\sum_{k=1}^K r_k = 1$. 
Given any absolute constant $c_{\min} \in (0,1)$, for any $h \geq 0$ and $\rho_k \in   [0 , 1-c_{\mathrm{min}}]$ for $ k=1,\dots,K$, we have
\begin{equation}
\labs \p \lbk \frac{\widehat{\beta}_j}{L_j} > z_\alpha \rbk -  {\pi} (h,\rho_1,\dots,\rho_K,r_1,\dots,r_K) \rabs 
\lesssim  
\frac{\sqrt{\log n} \max \lsk  d, \log n \rsk}{ \sqrt{n} },
\label{eqn: power difference Gaussian and Diagonal V}
\end{equation}
where 
$$ {\pi} (h,\rho_1,\dots,\rho_K,r_1,\dots,r_K) = \mathbb{E}\Phi \lsk \frac{h}{\sqrt{\sum_{k=1}^K r_k \lmk \rho_k Z_k^2 + 1 -\rho_k  \rmk}}  - z_\alpha \rsk,$$
and $Z_1,\dots,Z_K \overset{iid}{\sim} \mathcal{N}(0,1)$.
\label{corollary: power pihat_G(h, vecrho, vecr)}
\end{corollary}

In this corollary and Corollary \ref{corollary: power pihat_G(h,rho)} in the main paper, the Gaussian assumption is not essential. we adopt it only for theoretical convenience.
From \eqref{eqn: power difference general} in Theorem \ref{theorem: power}, we only need to show that
\begin{equation}
 | {\pi}(h, \bV) -  {\pi} (h,\rho_1,\dots,\rho_K,r_1,\dots,r_K)| \lesssim \frac{\sqrt{\log n} \max \lsk  d, \log n \rsk}{ \sqrt{n}}.
 \label{eqn: NTS in pf of piwidetilde(h, vecrho, vecr)}
\end{equation}

In Corollary \ref{corollary: power pihat_G(h, vecrho, vecr)},   let  $\bQ\bLambda \bQ\mt$ denote the eigen-decomposition of $\bV$.
The diagonal matrix $\bLambda$ consists of eigenvalues of $\bV$, which are $n_k \rho_k + 1 - \rho_k$ with multiplicity 1, and $1-\rho_k$ with multiplicity $n_k-1$ for $k=1,\dots,K$.

Under the condition $\bw \sim \mathcal{N}(\bzero, \bI_n)$ and using the fact that $\bQ\mt \bw \overset{L}{=} \bw$, we have
\beq
\pi (h,\bV)=\e \Phi \lsk \frac{h}{\sqrt{\bw\mt\bV\bw/n}} - z_\alpha \rsk = \e \Phi \lsk \frac{h}{\sqrt{\bw\mt\bLambda\bw/n}} - z_\alpha \rsk.
\eeq

Note that $\bw\mt\bLambda\bw=\sum_{k=1}^K n_k \rho_k w_k^2 + \bw\mt \bLambda' \bw$, where $\bLambda'= \mathrm{diag}\{ (1-\rho_1)\bI_{n_1}, \dots, (1-\rho_K) \bI_{n_K} \} \in \mathbb{R}^{n \times n}$.
Since
\beq
\frac{\bw\mt \bLambda \bw}{n} = \sum_{k=1}^K \frac{n_k}{n} \rho_k w_k^2 + \frac{\bw\mt\bLambda'\bw}{n} \overset{L}{\longrightarrow} \sum_{k=1}^K r_k \lmk \rho_k w_k^2 + (1-\rho_k) \rmk,
\eeq
by the Portmanteau Theorem, we have
\beq
 {\pi}(h,\bV) \overset{L}{\longrightarrow}  \mathbb{E}\Phi \lsk \frac{h}{\sqrt{\sum_{k=1}^K r_k \lmk \rho_k w_k^2 + 1 -\rho_k  \rmk}}  - z_\alpha \rsk.
\eeq
The following proof only characterizes the convergence rate.

\begin{proof}[Proof of Corollary \ref{corollary: power pihat_G(h, vecrho, vecr)}]

First, from the definition of $\pi(h,\bV)$ and $\pi(h,\rho_1,\dots,\rho_K,r_1,\dots,r_K)$, we have
\begin{eqnarray}
& & \labs \pi(h,\bV) - \pi(h,\rho_1,\dots,\rho_K,r_1,\dots,r_K) \rabs \nonumber\\
&\leq & \mathbb{E}\labs \Phi\lsk \frac{h}{\sqrt{\bw\mt\bLambda\bw/n}} - z_\alpha \rsk - \Phi \lsk \frac{h}{\sqrt{\sum_{k=1}^K r_k \lmk \rho_k w_k^2 + 1 -\rho_k  \rmk}}  - z_\alpha \rsk \rabs \nonumber\\
& {=} & \mathbb{E}\mathcal{I}_{\lbk \mathcal{E}_1 \rbk} \labs \Phi\lsk \frac{h}{\sqrt{\bw\mt\bLambda\bw/n}} - z_\alpha \rsk - \Phi \lsk \frac{h}{\sqrt{\sum_{k=1}^K r_k \lmk \rho_k w_k^2 + 1 -\rho_k  \rmk}}  - z_\alpha \rsk \rabs + I_1 \nonumber
\\
&\leq & 2 \p \lbk \mathcal{E}_1 \rbk + I_1,
\label{eqn: power pihat_G(h, vecrho, vecr) unified inequality}
\end{eqnarray}
where
\beq
\mathcal{E}_1 = \lbk \labs \frac{\bw\mt \bLambda' \bw}{n} - \sum_{k=1}^K r_k (1-\rho_k) \rabs > \frac{1}{2} \sum_{k=1}^K r_k (1-\rho_k) \rbk,
\eeq
and
\beq
I_1= \mathbb{E}\mathcal{I}_{\lbk \mathcal{E}_1^c \rbk} \labs \Phi\lsk \frac{h}{\sqrt{\bw\mt\bLambda\bw/n}} - z_\alpha \rsk - \Phi \lsk \frac{h}{\sqrt{\sum_{k=1}^K r_k \lmk \rho_k w_k^2 + 1 -\rho_k  \rmk}}  - z_\alpha \rsk \rabs.
\eeq

Next, we will bound $I_1$ given different ranges of $h$. 
\paragraph{Case 1.}
If $h > \lsk z_\alpha + \sqrt{\log n} \rsk \sqrt{5/2 + \log n / (2 \widetilde{c}_6) }$, where $\widetilde{c}_6$ is from Lemma \ref{lemma: power pihat_G(h, vecrho, vecr) Concentration}, we have
\beq
I_1 &=& \mathbb{E}\mathcal{I}_{\lbk \mathcal{E}_1^c \rbk} \labs \Phi\lsk \frac{h}{\sqrt{\bw\mt\bLambda\bw/n}} - z_\alpha \rsk -1 + 1 - \Phi \lsk \frac{h}{\sqrt{\sum_{k=1}^K r_k \lmk \rho_k w_k^2 + 1 -\rho_k  \rmk}}  - z_\alpha \rsk \rabs \\
&=& \mathbb{E}\mathcal{I}_{\lbk \mathcal{E}_1^c \rbk} \labs \Phi\lsk z_\alpha - \frac{h}{\sqrt{\bw\mt\bLambda\bw/n}} \rsk - \Phi \lsk z_\alpha -  \frac{h}{\sqrt{\sum_{k=1}^K r_k \lmk \rho_k w_k^2 + 1 -\rho_k  \rmk}} \rsk \rabs \\
&\leq & I_2 + I_3,
\eeq
where 
\beq
I_2=\mathbb{E}\mathcal{I}_{\lbk \mathcal{E}_1^c \rbk} \Phi\lsk z_\alpha - \frac{h}{\sqrt{\bw\mt\bLambda\bw/n}} \rsk, \ I_3=\mathbb{E}\mathcal{I}_{\lbk \mathcal{E}_1^c \rbk} \Phi \lsk z_\alpha -  \frac{h}{\sqrt{\sum_{k=1}^K r_k \lmk \rho_k w_k^2 + 1 -\rho_k  \rmk}} \rsk.
\eeq
To bound $I_2$, note that on the event 
$\mathcal{E}_1^c$, we have
\beq
\frac{\bw\mt \bLambda' \bw }{n} \leq \frac{3}{2} \sum_{k=1}^K r_k(1-\rho_k) \leq \frac{3}{2} \sum_{k=1}^K r_k = \frac{3}{2} \
\Rightarrow \ 
\frac{\bw\mt \bLambda \bw }{n} \leq \sum_{k=1}^K \frac{n_k}{n} \rho_k w_k^2 + \frac{3}{2}.
\eeq
Hence 
\begin{eqnarray}
    I_2 &\leq& \mathbb{E} \Phi\lsk z_\alpha - \frac{h}{\sqrt{ \sum_{k=1}^K \frac{n_k}{n} \rho_k w_k^2 + \frac{3}{2} }} \rsk\nonumber\\
&=& \mathbb{E} \mathcal{I} \lbk \sum_{k=1}^K \frac{n_k}{n} \rho_k w_k^2 - \sum_{k=1}^K \frac{n_k}{n} \rho_k > \frac{\log n }{2 \widetilde{c}_6 } \rbk \Phi\lsk z_\alpha - \frac{h}{\sqrt{ \sum_{k=1}^K \frac{n_k}{n} \rho_k w_k^2 + \frac{3}{2} }} \rsk \nonumber\\
& & + \mathbb{E} \mathcal{I} \lbk \sum_{k=1}^K \frac{n_k}{n} \rho_k w_k^2 - \sum_{k=1}^K \frac{n_k}{n} \rho_k \leq \frac{\log n }{2 \widetilde{c}_6 } \rbk \Phi\lsk z_\alpha - \frac{h}{\sqrt{ \sum_{k=1}^K \frac{n_k}{n} \rho_k w_k^2 + \frac{3}{2} }} \rsk.
\label{eqn: power pihat_G(h, vecrho, vecr) I_2 unified}
\end{eqnarray}
By Lemma \ref{lemma: power pihat_G(h, vecrho, vecr) I2 I3}, the first term in \eqref{eqn: power pihat_G(h, vecrho, vecr) I_2 unified}
is bounded by
\beq
\p \lbk \sum_{k=1}^K \frac{n_k}{n} \rho_k w_k^2 - \sum_{k=1}^K \frac{n_k}{n} \rho_k > \frac{\log n }{2 \widetilde{c}_6  } \rbk \leq \frac{1}{\sqrt{n}}.
\eeq
The second term in \eqref{eqn: power pihat_G(h, vecrho, vecr) I_2 unified} satisfies
\begin{eqnarray}
& &\mathbb{E} \mathcal{I} \lbk \sum_{k=1}^K \frac{n_k}{n} \rho_k w_k^2 - \sum_{k=1}^K \frac{n_k}{n} \rho_k \leq \frac{\log n }{2  \widetilde{c}_6 } \rbk \Phi\lsk z_\alpha - \frac{h}{\sqrt{ \sum_{k=1}^K \frac{n_k}{n} \rho_k w_k^2 + \frac{3}{2} }} \rsk \nonumber \\
&\leq& 
\mathbb{E} \Phi\lsk z_\alpha - \frac{h}{\sqrt{ \sum_{k=1}^K \frac{n_k}{n} \rho_k + \frac{\log n}{2 \widetilde{c}_6} + \frac{3}{2} }} \rsk \nonumber \\
& {\leq}& \mathbb{E} \Phi\lsk z_\alpha - \frac{h}{\sqrt{ \frac{\log n}{2 \widetilde{c}_6 } + \frac{5}{2} }} \rsk 
\label{eqn: E I Phi for rhovec rvec concentrate part step i} \\
& {\leq}& \frac{1}{\sqrt{2 \pi n}},
\label{eqn: E I Phi for rhovec rvec concentrate part step ii}    
\end{eqnarray}
where \eqref{eqn: E I Phi for rhovec rvec concentrate part step i} follows from $\sum_{k=1}^K n_k\rho_k/n \leq 1$, \eqref{eqn: E I Phi for rhovec rvec concentrate part step ii} follows from $h > \lsk z_\alpha + \sqrt{\log n} \rsk \sqrt{5/2 + \log n / ( 2 \widetilde{c}_6 ) }$ and $\Phi(-\sqrt{\log n}) = 1 - \Phi(\sqrt{\log n}) \leq (2\pi n)^{-1/2}$ by Lemma \ref{lemma: gaussian tail}.
Combining the bounds of the first and the second term, we have $I_2 \leq n^{-1/2} + (2\pi n)^{-1/2}$.

By similar arguments with Lemma \ref{lemma: power pihat_G(h, vecrho, vecr) I2 I3}, we have
$I_3 \leq n^{-1/2} + (2\pi n)^{-1/2}$ which implies that if  $h > \lsk z_\alpha + \sqrt{\log n} \rsk \sqrt{5/2 + \log n / ( 2 \widetilde{c}_6 ) }$, we have $I_1 \lesssim 1/\sqrt{n}$ in \eqref{eqn: power pihat_G(h, vecrho, vecr) unified inequality}.

\paragraph{Case 2.} If $0 \leq h \leq \lsk z_\alpha + \sqrt{\log n} \rsk \sqrt{ 5/2 + \log n /( 2 \widetilde{c}_6 ) }$, we can bound $I_1$ in \eqref{eqn: power pihat_G(h, vecrho, vecr) unified inequality} with the mean value theorem: \allowdisplaybreaks
\beq
I_1 &{\leq }& \frac{h}{\sqrt{2 \pi}} \mathbb{E} \mathcal{I}_{ \{ \mathcal{E}_1^c \} } \labs  \frac{1}{\sqrt{\frac{\bw\mt \bLambda \bw}{n}}}  - \frac{1}{\sqrt{  \sum_{k=1}^K r_k \lmk \rho_k w_k^2 + 1 -\rho_k  \rmk }}\rabs\\
&=&  \frac{h}{\sqrt{2 \pi}} \mathbb{E} \mathcal{I}_{ \{ \mathcal{E}_1^c \} }  \frac{\labs \sqrt{\frac{\bw\mt \bLambda \bw}{n}} - \sqrt{  \sum_{k=1}^K r_k \lmk \rho_k w_k^2 + 1 -\rho_k  \rmk } \rabs}{\sqrt{\frac{\bw\mt \bLambda \bw}{n}} \sqrt{  \sum_{k=1}^K r_k \lmk \rho_k w_k^2 + 1 -\rho_k  \rmk }}.
\eeq
On the event $\mathcal{E}_1^c$, the denominator $\bw\mt \bLambda \bw/n \geq \bw\mt\bLambda'\bw/n \geq \frac{1}{2} \sum_{k=1}^K r_k(1-\rho_k)$, and $\sum_{k=1}^K r_k \lsk \rho_k w_k^2 + 1 -\rho_k  \rsk \geq \sum_{k=1}^K r_k (1-\rho_k)$ which implies that
\beq
I_1 &\leq& \frac{h}{\sqrt{2 \pi}} \frac{\sqrt{2}}{\sum_{k=1}^K r_k(1-\rho_k)} \mathbb{E} \labs \sqrt{\frac{\bw\mt \bLambda \bw}{n}} - \sqrt{  \sum_{k=1}^K r_k \lmk \rho_k w_k^2 + 1 -\rho_k  \rmk } \rabs \\
&=& \frac{h}{\sqrt{2 \pi}} \frac{\sqrt{2}}{\sum_{k=1}^K r_k(1-\rho_k)} \mathbb{E} \frac{\labs \frac{\bw\mt\bLambda\bw}{n} - \sum_{k=1}^K r_k \rho_k w_k^2 - \sum_{k=1}^K r_k(1-\rho_k) \rabs}{\sqrt{\frac{\bw\mt \bLambda \bw}{n}} + \sqrt{  \sum_{k=1}^K r_k \lmk \rho_k w_k^2 + 1 -\rho_k  \rmk }}.
\eeq
Since $\bw\mt\bLambda\bw / n = \sum_{k=1}^K (n_k/n) \rho_k w_k^2 + \bw\mt\bLambda'\bw/n $, we split the expression above into two parts,
\begin{eqnarray}
I_1 &\leq& \frac{h}{\sqrt{2 \pi}} \frac{\sqrt{2}}{\sum_{k=1}^K r_k(1-\rho_k)} \mathbb{E} \lmk \frac{ \labs \sum_{k=1}^K \frac{n_k}{n} \rho_k w_k^2  - \sum_{k=1}^K r_k \rho_k w_k^2 \rabs}{\sqrt{\frac{\bw\mt \bLambda \bw }{n}} + \sqrt{\sum_{k=1}^K \rho_k r_k w_k^2 + \sum_{k=1}^K r_k (1-\rho_k)}} \rmk \nonumber\\
& & + \frac{h}{\sqrt{2 \pi}} \frac{\sqrt{2}}{\sum_{k=1}^K r_k(1-\rho_k)} \mathbb{E} \lmk \frac{\labs \frac{\bw\mt \bLambda' \bw}{n} - \sum_{k=1}^K r_k (1-\rho_k) \rabs}{\sqrt{\frac{\bw\mt \bLambda \bw }{n}} + \sqrt{\sum_{k=1}^K \rho_k r_k w_k^2 + \sum_{k=1}^K r_k (1-\rho_k)}} \rmk \nonumber\\
& {\leq } &  \frac{h}{\sqrt{2 \pi}} \frac{\sqrt{2}}{\sum_{k=1}^K r_k(1-\rho_k)} \lmk  \sqrt{\frac{K}{n \lsk \min_{1 \leq k \leq K} r_k \rsk}} +  \frac{20}{\sqrt{3} \widetilde{c}_6 }  \frac{\sqrt{\sum_{k=1}^K r_k(1-\rho_k)}}{\lsk \min_{1 \leq k \leq K} r_k \rsk \sqrt{n}}  \rmk \label{eqn: power pihat_G(h, vecrho, vecr) I_1 h <=... unified step i}\\
& {\leq } & \frac{hK}{\sqrt{\pi} c_{\mathrm{min}} \lsk \min_{1 \leq k \leq K} r_k \rsk \sqrt{n} } + \frac{20 h}{\sqrt{3\pi} \widetilde{c}_6 \sqrt{c_{\mathrm{min}}} \lsk \min_{1 \leq k \leq K} r_k \rsk \sqrt{n}} \label{eqn: power pihat_G(h, vecrho, vecr) I_1 h <=... unified step ii}\\
& {\lesssim } &   \frac{ \log n}{  \sqrt{n} },
\label{eqn: power pihat_G(h, vecrho, vecr) I_1 h <=... unified step iii}    
\end{eqnarray}
where \eqref{eqn: power pihat_G(h, vecrho, vecr) I_1 h <=... unified step i} follows from Lemma \ref{lemma: pwr pihat_G(h, vecrho, vecr) bd dist conc}, \eqref{eqn: power pihat_G(h, vecrho, vecr) I_1 h <=... unified step ii} follows from $\sum_{k=1}^K r_k(1-\rho_k) \geq c_{\mathrm{min}} \sum_{k=1}^K r_k = c_{\mathrm{min}}$, and \eqref{eqn: power pihat_G(h, vecrho, vecr) I_1 h <=... unified step iii} uses the fact that $0 \leq h \leq \lsk z_\alpha + \sqrt{\log n} \rsk \sqrt{5/2 + \log n / ( 2 \widetilde{c}_6 ) }$.

Collecting the results from {\bf Case 1} and {\bf Case 2}, we conclude that for any $h \geq 0$, $I_1 \lesssim  \log n /   \sqrt{n} $. 
Additionally, by Lemma \ref{lemma: power pihat_G(h, vecrho, vecr) Concentration}, $\p \lbk \mathcal{E}_1 \rbk \leq 2 K \exp \lmk -\widetilde{c}_6 \lsk\min_{1 \leq k \leq K} r_k \rsk n \rmk$. 
Plugging the  bounds of $I_1$ and $\p\{\mathcal{E}_1\}$ into \eqref{eqn: power pihat_G(h, vecrho, vecr) unified inequality}, we have
\beq
\labs \pi(h,\bV) - \pi(h,\rho_1,\dots,\rho_K,r_1,\dots,r_K) \rabs 
&{\lesssim }&  \exp \lmk -\widetilde{c}_6 \lsk\min_{1 \leq k \leq K} r_k \rsk n \rmk +   \frac{ \log n}{  \sqrt{n} }  \\
& \lesssim &  \frac{\log n}{  \sqrt{n} },
\eeq
which concludes the proof of \eqref{eqn: NTS in pf of piwidetilde(h, vecrho, vecr)} and the final bound in \eqref{eqn: power difference Gaussian and Diagonal V}.

\end{proof}

\subsection{Some Properties of $\pi(h, \rho) - \pi(h, 0)$ in Corollary \ref{corollary: power pihat_G(h,rho)}}
\label{subsec: pf of Dhat(h, rho) property}

Recall the power approximation $\pi( h, \rho) =
\e \{ \Phi(h( \rho \chi_1^2 + 1 - \rho)^{-1/2} - z_\alpha)   \}$ as defined in Corollary \ref{corollary: power pihat_G(h,rho)}, where $\rho \in [0, 1-c_{\min}]$ with $c_{\min} \in (0,1)$ is the correlation in $\bV_{1,\rho} = \rho \bone_n \bone_n\mt + (1-\rho) \bI_n$, and $h \in [0, \infty)$ is the signal in $\beta_j$.
Define the power difference between $\rho \geq 0$ and $\rho = 0$ as 
\begin{equation}
\Delta \pi(h, \rho) = \pi(h, \rho) - \pi(h, 0).
\label{eqn: D pi(h, rho)}
\end{equation}

\begin{lemma} \label{lemma: properties of D pi}
For $\Delta \pi(h, \rho)$ defined in \eqref{eqn: D pi(h, rho)}, we have
\begin{enumerate}
    \item Power gain with small signal: 
    given $\rho \in (0, 1-c_{\mathrm{min}}]$, if $h \leq \frac{1}{2}\sqrt{1-\rho} ( z_\alpha + \sqrt{z_\alpha^2 + 12} )$, then $\Delta \pi (h, \rho)>0$.
    \item Power loss with large signal: 
    given $\rho \in (0, 1-c_{\mathrm{min}}]$, if $h > \max ( 2 z_\alpha, \{ z_\alpha \int_1^{\infty} f(t) [ 1- (\rho t + 1 -\rho)^{-1/2} ] \textup{d}t \}^{-1} \p \{ \chi_1^2<1 \} )$, then $\Delta \pi (h, \rho)<0$, where $f(t)$ is the density of $\chi_1^2$.
    \item Diminishing power difference: given $\rho \in [0, 1-c_{\mathrm{min}}]$, we have $\lim_{h \to \infty} \Delta \pi (h, \rho) = 0$.
\end{enumerate}
\end{lemma}

\begin{proof}[Proof of Lemma \ref{lemma: properties of D pi}]
\

\begin{enumerate}

\item 
Denote the $\chi_1^2$ random variable by $T$, \eqref{eqn: D pi(h, rho)} reduces to  $\Delta \pi (h, \rho) = \mathbb{E}_T [  \Phi (  h/\sqrt{\rho T + 1 - \rho} - z_\alpha  ) - \Phi (  h - z_\alpha )  ]$. 
The integrand has second order derivative
\begin{eqnarray}
& &\frac{\partial^2 \lmk \Phi \lsk h/\sqrt{\rho T + 1 - \rho} - z_\alpha \rsk - \Phi \lsk h - z_\alpha \rsk \rmk}{\partial T^2} \nonumber\\
&=& \frac{-h \rho\pwtwo \phi \lsk h/\sqrt{\rho T + 1 - \rho} - z_\alpha  \rsk}{4(\rho T + 1 - \rho)^{5/2}} \lmk  \lsk h/\sqrt{\rho T + 1 - \rho} \rsk^2 - z_\alpha \lsk h/ \sqrt{\rho T + 1 - \rho}\rsk - 3\rmk,
\label{eqn: Phi''(h/sqrt(rho T  + 1 - rho)) - Phi(h - z_alpha)}  \nonumber  
\end{eqnarray}
where $\phi(\cdot)$ is the density of $\mathcal{N}(0,1)$.
If $T> \rho^{-1} \{ [ 2h/( z_\alpha + \sqrt{z_\alpha^2 + 12}) ]^2 -(1-\rho) \}$, the second order derivative is positive. 
Hence if $[ 2h/( z_\alpha + \sqrt{z_\alpha^2 + 12}) ]^2 -(1-\rho) \leq 0$, the second order derivative is positive on the support $T>0$.
By Jensen's inequality,  $\Phi \lsk h/\sqrt{\rho T + 1 - \rho} - z_\alpha \rsk - \Phi(h - z_\alpha)$ is strictly convex, which implies that  $\Delta \pi (h, \rho) > 0$.

\item 
Rewrite $\Delta \pi (h, \rho)$ as
\begin{eqnarray}
    \Delta \pi (h, \rho) &=& \int_0^1 f(t) \lmk \Phi\lsk  \frac{h}{\sqrt{\rho t + 1 - \rho}} - z_\alpha \rsk - \Phi \lsk h-z_\alpha \rsk  \rmk \textup{d}t \nonumber\\
& & - \int_1^\infty f(t) \lmk \Phi \lsk h-z_\alpha \rsk - \Phi\lsk  \frac{h}{\sqrt{\rho t + 1 - \rho}} - z_\alpha \rsk \rmk \textup{d}t.
\label{eqn: Dhat_G(h,rho) upper bd unified for large h}
\end{eqnarray}
For the first term in \eqref{eqn: Dhat_G(h,rho) upper bd unified for large h}, if $h > z_\alpha$, by Lemma \ref{lemma: gaussian tail}, it is bounded by
\begin{eqnarray}
\int_0^1 f(t) \lmk 1- \Phi \lsk h-z_\alpha \rsk \rmk dt \leq \frac{\p \lbk \chi_1^2 < 1 \rbk \phi \lsk h - z_\alpha \rsk}{h-z_\alpha}.
\label{eqn: Dhat_G(h,rho) large h term 1}    
\end{eqnarray}
For the second term in \eqref{eqn: Dhat_G(h,rho) upper bd unified for large h}, we have
\begin{eqnarray}
& &\int_1^\infty f(t) \lmk \Phi \lsk h-z_\alpha \rsk - \Phi\lsk  \frac{h}{\sqrt{\rho t + 1 - \rho}} - z_\alpha \rsk \rmk \textup{d}t \nonumber\\
&\geq& \int_1^\infty f(t) \phi(h - z_\alpha) \lsk h-z_\alpha - \frac{h}{\sqrt{\rho t + 1 - \rho}} + z_\alpha\rsk \textup{d}t \label{eqn: Dhat_G(h,rho) large h term 2 step i} \\
&\geq& h \phi \lsk h-z_\alpha  \rsk  \int_1^\infty f(t) \lmk  1 - \lsk \rho t + 1 - \rho \rsk^{-1/2} \rmk \textup{d}t,
\label{eqn: Dhat_G(h,rho) large h term 2 step ii}  
\end{eqnarray}
where \eqref{eqn: Dhat_G(h,rho) large h term 2 step i} follows from the mean value theorem and the fact that if $h > 2 z_\alpha$ and $t > 1$, we have $\labs h/\sqrt{\rho t + 1 - \rho} - z_\alpha \rabs \leq  h-z_\alpha$. 

Plugging \eqref{eqn: Dhat_G(h,rho) large h term 1} and \eqref{eqn: Dhat_G(h,rho) large h term 2 step ii} into \eqref{eqn: Dhat_G(h,rho) upper bd unified for large h}, if $h>2 z_\alpha$, we have
\beq
\Delta \pi (h,\rho) \leq \phi(h-z_\alpha) \lbk \frac{\p \lsk \chi_1^2 < 1 \rsk}{z_\alpha} - h \int_1^\infty f(t)  \lmk  1 - \lsk \rho t + 1 - \rho \rsk^{-1/2} \rmk \textup{d}t \rbk,
\eeq
which is negative when $h> \{ z_\alpha  \int_1^\infty f(t)  [   1 - \lsk \rho t + 1 - \rho \rsk^{-1/2}  ] \textup{d}t \}\inverse \p \{  
\chi_1^2 <1 \}$.

\item By the monotone convergence theorem, for any $\rho \in [0, 1-c_{\min}]$, we have $\lim_{h \to \infty} \pi (h, \rho) = \e \{ \lim_{h \to \infty} \Phi(h( \rho \chi_1^2 + 1 - \rho)^{-1/2} - z_\alpha) \} = 1$, which completes the proof.
\end{enumerate}
\end{proof}


\begin{thebibliography}{}

\bibitem[\protect\citeauthoryear{Abadie, Athey, Imbens, and Wooldridge}{Abadie
  et~al.}{2023}]{abadie2023should}
Abadie, A., S.~Athey, G.~W. Imbens, and J.~M. Wooldridge (2023).
\newblock When should you adjust standard errors for clustering?
\newblock {\em Q. J. Econ.\/}~{\em 138\/}(1), 1--35.

\bibitem[\protect\citeauthoryear{Barrios, Diamond, Imbens, and
  Koles{\'a}r}{Barrios et~al.}{2012}]{barrios2012clustering}
Barrios, T., R.~Diamond, G.~W. Imbens, and M.~Koles{\'a}r (2012).
\newblock Clustering, spatial correlations, and randomization inference.
\newblock {\em J. Am. Stat. Assoc.\/}~{\em 107\/}(498), 578--591.

\bibitem[\protect\citeauthoryear{Bickel and Freedman}{Bickel and
  Freedman}{1982}]{bickel1983bootstrapping}
Bickel, P.~J. and D.~Freedman (1982).
\newblock Bootstrapping regression models with many parameters.
\newblock In {\em A Festschrift for Erich L. Lehmann}, pp.\  28--48. Belmont,
  Calif.: Wadsworth Statist./Probab. Ser.

\bibitem[\protect\citeauthoryear{Billingsley}{Billingsley}{1995}]{patrick1995probability}
Billingsley, P. (1995).
\newblock {\em Probability and Measure\/} (3 ed.).
\newblock John Wiley \& Sons.

\bibitem[\protect\citeauthoryear{Chetverikov, Hahn, Liao, and
  Santos}{Chetverikov et~al.}{2023}]{chetverikov2023standard}
Chetverikov, D., J.~Hahn, Z.~Liao, and A.~Santos (2023).
\newblock Standard errors when a regressor is randomly assigned.
\newblock {\em arXiv:2303.10306\/}.

\bibitem[\protect\citeauthoryear{Eicker}{Eicker}{1967}]{eicker1967limit}
Eicker, F. (1967).
\newblock Limit theorems for regressions with unequal and dependent errors.
\newblock {\em Proc. 5th Berkeley Symp. Math. Stat. Probab.\/}~{\em 1\/}(1),
  59--82.

\bibitem[\protect\citeauthoryear{Lehmann and Romano}{Lehmann and
  Romano}{2005}]{lehmann2005testing}
Lehmann, E.~L. and J.~P. Romano (2005).
\newblock {\em Testing Statistical Hypotheses\/} (3 ed.).
\newblock New York, NY, USA: Springer.

\bibitem[\protect\citeauthoryear{Lei and Bickel}{Lei and
  Bickel}{2021}]{lei2021assumption}
Lei, L. and P.~J. Bickel (2021).
\newblock An assumption-free exact test for fixed-design linear models with
  exchangeable errors.
\newblock {\em Biometrika\/}~{\em 108\/}(2), 397--412.

\bibitem[\protect\citeauthoryear{Liang and Zeger}{Liang and
  Zeger}{1986}]{liang1986longitudinal}
Liang, K.-Y. and S.~L. Zeger (1986).
\newblock Longitudinal data analysis using generalized linear models.
\newblock {\em Biometrika\/}~{\em 73\/}(1), 13--22.

\bibitem[\protect\citeauthoryear{Lin}{Lin}{2013}]{lin2013agnostic}
Lin, W. (2013).
\newblock Agnostic notes on regression adjustments to experimental data:
  Reexamining {Freedman}'s critique.
\newblock {\em Ann. Appl. Stat.\/}~{\em 7\/}(1), 295--318.

\bibitem[\protect\citeauthoryear{Lin and Bai}{Lin and
  Bai}{2010}]{lin2010probability}
Lin, Z. and Z.~Bai (2010).
\newblock {\em Probability Inequalities}.
\newblock Springer.

\bibitem[\protect\citeauthoryear{Newey and West}{Newey and
  West}{1987}]{newey1987simple}
Newey, W.~K. and K.~D. West (1987).
\newblock A simple, positive semi-definite, heteroskedasticity and
  autocorrelation consistent covariance matrix.
\newblock {\em Econometrica\/}~{\em 55}, 703–708.

\bibitem[\protect\citeauthoryear{Vershynin}{Vershynin}{2010}]{vershynin2010introduction}
Vershynin, R. (2010).
\newblock Introduction to the non-asymptotic analysis of random matrices.
\newblock {\em arXiv preprint arXiv:1011.3027\/}.

\bibitem[\protect\citeauthoryear{Vershynin}{Vershynin}{2018}]{vershynin2018high}
Vershynin, R. (2018).
\newblock {\em High-Dimensional Probability: An Introduction with Applications
  in Data Science}.
\newblock Cambridge University Press.

\bibitem[\protect\citeauthoryear{White}{White}{1980}]{white1980heteroskedasticity}
White, H. (1980).
\newblock A heteroskedasticity-consistent covariance matrix estimator and a
  direct test for heteroskedasticity.
\newblock {\em Econometrica\/}~{\em 48}, 817--838.

\end{thebibliography}
\end{document}